\documentclass[11pt,leqno]{amsart}

\usepackage{anysize}
\marginsize{3.5cm}{3.5cm}{3cm}{3cm}

\usepackage[utf8]{inputenc}

\usepackage{amsmath}
\usepackage{amsfonts,amssymb}
\usepackage{enumerate}
\usepackage{mathrsfs}
\usepackage{amsthm}
\usepackage{xcolor}
\usepackage[normalem]{ulem}
\usepackage{enumitem}
\usepackage{bm}

\numberwithin{equation}{section}
\usepackage{etex}
\usepackage{pstricks-add}
\usepackage{latexsym}
\usepackage{mathrsfs}
\usepackage{pictexwd}
\usepackage{amsmath}
\usepackage{fancybox,color}
\usepackage{enumerate}
\usepackage[utf8]{inputenc}
\usepackage[colorlinks]{hyperref}
\usepackage{eurosym}
\usepackage{color}
\usepackage{url}
\usepackage{pgf,tikz}
\usepackage{mathrsfs}\usetikzlibrary{arrows}

\numberwithin{figure}{section}

\newcommand{\kom}[1]{}
\renewcommand{\kom}[1]{{\bf [#1]}}

\newcommand{\be}{\begin{equation}}
\newcommand{\ee}{\end{equation}}

\usepackage{hyperref}
\hypersetup{
    colorlinks=true,                          
    linkcolor=black, % Colors for internal links
    citecolor=black, % Colors for the biblio references
    urlcolor=black  } 
 \def\1{\raisebox{2pt}{\rm{$\chi$}}}

\newtheorem{theorem}{Theorem} %[section]

\newtheorem{remark}{Remark}[section]

\theoremstyle{plain}
\newtheorem{theo}{Theorem}[section]
\newtheorem*{theo*}{Theorem}
\newtheorem{prop}[theo]{Proposition}
\newtheorem{lemm}[theo]{Lemma}

\newtheorem{defi}[theo]{Definition}
\theoremstyle{definition}

\newcommand{\R}{{\mathbb R}}

\newcommand{\N}{{\mathbb N}}

 \newcommand{\eps}{\varepsilon}
 \def\1{\raisebox{2pt}{\rm{$\chi$}}}
 
\usepackage{pstricks-add}

\newcommand{\ds}{\displaystyle} 
\newcommand{\ts}{\textstyle}

\DeclareMathOperator{\cn}{div}

\DeclareMathOperator{\di}{d}
\DeclareMathOperator{\diff}{d}

\def\dsigma{\diff \! \sigma}

\def\dt{\diff \! t}
\def\dx{\diff \! x}

\def\dz{\diff \! z}

\def\dy{\diff \! y}

\def\eps{\varepsilon}

\def\le{\leq}

\newcommand{\mH}{\mathcal{H}}

\begin{document}

\title[cell motion model with boundary signal production]{Convergence, concentration and critical mass phenomena 
in a model of cell motion with \\
boundary signal production}

\author[Meunier]{Nicolas Meunier}%
\address{LaMME, UMR 8071 CNRS, Universit\'e \'Evry Val d'Essonne, Evry, France}
\email{nicolas.meunier@univ-evry.fr}

\author[Souplet]{Philippe Souplet}%
\address{LAGA, UMR 7539 CNRS, Universit\'e Sorbonne Paris Nord, Villetaneuse, France}
\email{souplet@math.univ-paris13.fr}


\begin{abstract}	
We consider a model of cell motion with boundary signal production which describes some aspects of eukaryotic cell migration. Generic polarity markers located in the cell are transported by actin which they help to polymerize, i.e. the actin velocity depends on the asymmetry of the marker concentration profile. This leads to a problem whose mathematical novelty is the nonlinear and nonlocal destabilizing term in the boundary condition.   This model is a more rigorous version of a toy model first introduced in \cite{Maiuri2}.
	
We provide a detailed study of the qualitative properties of this model, namely local and global existence, convergence and blow-up of solutions. We start with a complete analysis of local existence-uniqueness in Lebesgue spaces. This turns out to be particularly relevant, in view of the mass conservation property and of the existence of $L^p$ Liapunov functionals, also obtained in this paper. 
 The optimal $L^p$ space of our local theory agrees with the scale invariance of the problem.

With the help of this local theory, we next study the global existence and convergence of solutions. In particular, in the case of quadratic nonlinearity, for any space dimension, we find an explicit, sharp mass threshold for global existence vs.~finite time blow-up of solutions. The proof is delicate, based on the possiblity to control the solution by means of the entropy function via an $\eps$-regularity type argument. This critical mass phenomenon is somehow reminiscent of the well-known situation for the $2d$ Keller-Segel system. For nonlinearitities with general power growth, under a suitable smallness condition on the initial data, we show that solutions exist globally and converge exponentially to a constant. 
	
As for the possibility of blow-up for large initial data, it turns out to occur only for nonlinearities with quadratic or superquadratic growth, whereas all solutions are shown to be global and bounded in the subquadratic case, thus revealing the existence of a sharp critical exponent for blow-up. Finally, we analyse some aspects of the blow-up asymptotics of solutions in time and space.	 
\end{abstract}	

\maketitle

\section{Introduction}

\subsection{Motivations and aims of the paper}

Cell migration is a fundamental process that is involved in many physiological and pathological functions (immune response, morphogenesis, cancer metastasis, etc). Therefore, understanding its key features despite the variety of cellular behaviors is a challenging task. Recently, a biophysical approach has shown that even though different migration modes coexist, cell migration follows a very general principle. 
	
 Following the first World Cell Race in 2011 \cite{Maiuri1},  it was possible to perform and then analyze extensive biological experiments in \cite{Maiuri2}. A directional correlation of the trajectories, exponentially correlated to the velocity of the cells, was then brought to light. The fastest cells have a more directional migration. This coupling was called UCSP (Universal Coupling between cell Speed and Persistence), see \cite{Maiuri2}. Other experiments suggest that it relies on actin fluxes, which exist in all cell types.

Actin filaments are essential components of the cytoskeleton of eukaryotic cells. These filaments are polar. They polymerize and grow at one end, and retract by depolymerizing at the other end. In migrating cells, the polymerizing ends are located near the cell membrane, which resists filament growth. Therefore, in the cell's frame of reference, actin filaments move away from the membrane, forming what are called retrograde actin flows.

Large, fast actin flows enhance cell polarity, and thus typical cell persistence time. Biological experiments suggest that this results from the advection of polarity signals, i.e. molecules involved in the regulation of cytoskeleton activity.

We consider here the case when 
the cell is a fixed bounded domain $\Omega$ of $\R^n$ 
and we are led  to the equation (see Section~\ref{Sec:Bio} for more details):
\begin{subequations}\label{pbmP}
\begin{align}
	\hfill\partial_t c&=\nabla\cdot\Bigl(\nabla c-c\ds\int_{\partial\Omega}f(c)\nu \dsigma\Bigr),
	 &&x\in\Omega,\ t>0,\\
	\hfill 0&=\Bigl(\nabla c-c\ds\int_{\partial\Omega}f(c)\nu \dsigma\Bigr)\cdot \nu,&&x\in\partial\Omega,\ t>0, \\
	\hfill c(x,0)&=c_0(x),&&x\in\Omega,
\end{align}
\end{subequations}
where $f: \R\to\R$ and $c$ denotes the concentration of some solute located in the cell, which might be any cytoplasmic protein controlling the active force-generation/adhesion machinery of the cell.

Note that the no-flux condition ensures the conservation of the total amount of solute, its mass, which will be noted  
	\[
	M:=\int_{\Omega} c_0(x) \dx. 
	\]
The convective vector field will be denoted by
$$A(t)=A(c(t)):= \int_{\partial\Omega}f(c)\nu \dsigma$$
(the variable $c$ being dropped when no confusion may arise).

The model \eqref{pbmP} describes the feedback loop between actin fluxes $A(t)$
and a molecular species, $c$: molecules are advected by actin fluxes and can be activated at the cell membrane. Activated molecules affect the speed of actin flow: the higher their concentration gradient across the cell, the faster the actin flow;
see Section~\ref{Sec:Bio} for more details on the model derivation.
The model \eqref{pbmP} is more realistic than the model initially introduced in \cite{Maiuri2}. Indeed, in this latter model the molecular dynamics is one-dimensional and in the chosen time scale the molecular dynamics is at equilibrium.

Throughout this paper, unless otherwise mentioned,
$\Omega$ is either a smoothly bounded domain of $\R^n$ ($n\ge 1$), or a finite cylinder
\be\label{defCyl}
\Omega=(0,L)\times B'_R,
\ee
where $L, R>0$ and $B'_R=B_R(O)\subset \R^{n-1}$ ($n\ge 2$) is the ball of radius $R$ centered 
at~$O$.
In all cases, we denote by $\Gamma$ the regular part of $\partial\Omega$, namely:
$$
\Gamma=
\begin{cases}
  \partial\Omega&\hbox{if $\Omega$ is smooth,}\\
  \noalign{\vskip 1mm}
\partial\Omega\setminus(\{0,L\}\times \partial B'_R)   &\hbox{if $\Omega$ is a cylinder \eqref{defCyl}.}
   \end{cases}$$
   Also $\nu$ denotes the outward unit normal and,
in case \eqref{defCyl}, the boundary conditions in \eqref{pbmP} are understood 
only on $\Gamma$.

In this model, the diffusion competes with the aggregating velocity, $\int_{\partial\Omega}f(c)\nu \dsigma$, and we address here the question whether or not this 
advection-diffusion equation could lead to concentration and possibly finite time blow-up, 
should the diffusion not be strong enough to balance the attractive velocity. More precisely, our main purpose here is to investigate the influence of the nonlinearity $f$ on the global or nonglobal solvability of problem \eqref{pbmP}.
To this end we will consider functions which typically grow like a power for large values of $c$, namely
$$f(c)\sim c^m\quad\hbox{as $c\to\infty$,\ for some $m>0$,}$$
and will highlight phenomena of convergence, critical mass and finite time blow-up with concentration on the boundary,
depending on the values of $m$.
In a forthcoming paper \cite{MeuSou2}, we will study the existence and properties of traveling waves for bounded increasing 
nonlinearities involving saturation effects, typically of the form $f(c)=Lc/(c+\alpha)$
with $L,\alpha>0$.

\smallskip

Let us briefly comment on the existing  literature related with \eqref{pbmP}. 
Non-linear and non-local convection-diffusion problems have raised many interesting and challenging mathematical issues, 
see \cite{QSb} e.g.~for a  partial, recent review. One of the most famous examples is the celebrated Keller-Segel model,
which was introduced in the early 70's by E.F. Keller and L.A. Segel in order to 
describe self-organization of amoebae and bacteria colonies \cite{Keller_Segel_70,Keller_Segel_71} (see \cite{Hillen_Painter,Horstmann,Blanchet} for modelling and analysis reviews). 
 One of the most intensively studied Keller-Segel type models, from the mathematical point of view, is the parabolic-elliptic 
cross-diffusive system 
\be\label{pbmKS}
\left\{\begin{aligned}
	\hfill\partial_t u&=\nabla\cdot\bigl(\nabla u-u\nabla v\bigr),&&x\in\Omega,\ t>0,\\
	\hfill 0&=\Delta v-kv+u,&&x\in\Omega,\ t>0,\\
	\hfill \partial_\nu u=\partial_\nu v&=0,&&x\in\partial\Omega,\ t>0, \\
	\hfill u(x,0)&=u_0(x),&&x\in\Omega,
	\end{aligned}
	\right.
\ee
where $k\ge 0$ is a constant.
Observe that the coupled PDE's in \eqref{pbmKS} can be rewritten as the nonlocal, advection-diffusion scalar equation
$$\partial_t u=\nabla\cdot\bigl(\nabla u-(\nabla G)\ast u\bigr),$$
where $G$ is the Neumann Green kernel for $-\Delta+k$.
As a key difference with \eqref{pbmKS}, we note that the nonlocal feature in \eqref{pbmP}
 affects both the PDE and the boundary conditions,
and is caused by an action at the boundary (polymerization), rather than inside the domain.
In models such as \eqref{pbmKS}, 
cells randomly diffuse and are attracted by a chemical signal which is secreted by the cells themselves
($u$ and $v$ respectively denote the concentrations of cells and of chemical signal,
and the term $-kv$ represents the degradation of the chemical signal).
 In the recent work \cite{FLW}, the system \eqref{pbmKS} with the modified, 
mixed Neumann/Dirichlet boundary conditions $\partial_\nu u-u\partial_\nu v=u=0$,
 was studied in a different context, namely as a model for a single crawling keratocyte. The unknown $v$ 
 now stands for the mechanical stress in the cytoskeleton and 
 $u$ is the density of myosin motor proteins, which actively generate stress by binding to and pulling on the actin filaments constituting the cytoskeleton meshwork. As for the term $-kv$, it models the dissipation of stress via traction with the substrate to which
the actin gel is linked by adhesion molecules. 
The authors show that, in any space dimension $n\ge 2$, this change in boundary conditions is accompanied by a substantial, 
$k$-depending change in the potential of the model to support the emergence of singular structures
such as finite time concentration.

In the context of polarization and motility of eukaryotic cells on substrates,
 other convection-diffusion 2d models were designed in \cite{Berlyand2, Berlyand1,Berlyand_2021,LMVC,recho_1,recho_2,CMM2, JMB_EMV,PLOS}, but either they were mathematically studied in the 1d case in \cite{CMM1,CalvezMeunier,CHMV,Lepoutre_Meunier_Muller_JMPA,EMV,Lepoutre_Meunier_CMS}, or they were only partially studied in the 2d case \cite{Berlyand_2021,AlaMaMeu,CMM2} (traveling wave solutions). In a different context, the work \cite{DHMRW}
studies a related 1d nonlocal and nonlinear electroreaction-diffusion model.

Let us mention connections with some other 
problems. First note that, by the divergence theorem, the nonlinearity in  \eqref{pbmP}  can be rewritten as 
\[
\nabla\cdot\Bigl(c\ds\int_{\partial\Omega}f(c)\nu \dsigma\Bigr)=\nabla c\cdot \int_\Omega f'(c)\nabla c\dx.
\]
We thus see that the equation has a quadratic growth with respect to the gradient.
There is a very large literature on such problems, with so-called natural growth,
but with local nonlinearities (see,~e.g.,~\cite{BMP, BGa, BPo, Alaa, BSW}).
On the other hand, there is a large literature on semilinear parabolic equations with nonlocal nonlinearities,
but the structure of \eqref{pbmP}  is peculiar and quite different from that in previously studied problems.
In particular, being of conservative convective form, problem \eqref{pbmP}
enjoys the mass conservation property. 
Let us mention some examples which nevertheless share some features with~\eqref{pbmP}.
For nonlocal Neumann problems with zero order nonlinearities, of the form 
\be\label{pbmP1d}
\begin{cases}
\hfill u_t-\Delta u&=f_1(u)\ds\int_\Omega f_2(u)\dx ,\quad x\in\Omega,\ t>0, \\
\noalign{\vskip 1mm}
\hfill u_\nu&=g_1(u) \ds\int_\Omega g_2(u)\dx , \quad x\in\partial\Omega,\ t>0, \\
\end{cases}
\ee
results on blow-up and global existence can be found in, e.g.,~\cite{Pao, LWSL, Glad} and the references therein.
Equations with nonlocal gradient terms, of the form 
$$u_t-u^m\Delta u=u^p\int_\Omega |\nabla u|^2\dx $$
(and homogenous boundary conditions) have been studied in \cite{Dl,S02,KS, KLW, LW}.
The case $m=p=1$ arises in a model of replicator dynamics.
Regarding nonlocal problems with mass conservation, let us mention the equation 
$$u_t-\Delta u=u^p-|\Omega|^{-1}\int_\Omega u^p\dx $$
(with homogeneous Neumann conditions), 
which has been studied in \cite{JK, WTL, HY, BDS}.
As another nonlocal problem with mass control, the equation
$$u_t-\Delta u=\ds \frac{\lambda e^u}{\int_\Omega e^u \dx }$$
which arises,~i.a.,~in connection with the Keller-Segel system,
has also received a lot of attention (see \cite{KS07, Wo, KS}).

\subsection{Main results}
 
 \subsubsection{Local and global existence and convergence results}
 
 By standard theory \cite{LSU, Fr, Am, Lie},
 if $f\in C^1$, 
 $c_0$ is sufficiently smooth, say $c_0\in C^1(\overline\Omega)$,
  and $\Omega$ is smooth,
 then problem \eqref{pbmP} admits a unique maximal, classical solution. 
Our first concern is the solvability for low regularity initial data, 
 especially in Lebesgue spaces. This question is natural in view of the mass conservation property and,
 besides its own interest, this will turn out to be relevant for the study of the global behavior.
For (local) problems with so-called natural growth with respect to the gradient (see the previous subsection),
existence for positive $L^1$ data holds provided the nonlinear term has a ``good'' sign,
i.e.~acts as an absorption (see,~e.g.,~\cite{BMP, BGa, BPo, BSW}), whereas it fails when it acts as a source (see \cite{Alaa, BSW}).
Here, in addition to the nonlocal feature, the convection term does not have clear sign and the techniques in the above-mentioned papers 
cannot be used.
Our approach is different, relying on semigroup techniques originating from \cite{W80}
(see also \cite{Am2} and the references therein),
 combined with suitable trace theorems and fractional Sobolev spaces, and exploiting the conservative structure of the problem.

Before stating our main results, we note that, 
in the case when $\Omega$ is a cylindrical domain (cf.~\eqref{defCyl}),
 $C^{2,1}$ regularity up to the corner points
of the boundary is not expected (see related counter-examples in Remark~\ref{remcorners}).
 Therefore, some care is needed regarding the definition of classical solution, 
 which we make precise in the following:

  \begin{defi} \label{defsol}
  For given $\tau\in(0,\infty]$, we set
	\be\label{defEtau}
  E_\tau:=
  \begin{cases}
  C^{2,1}(\overline\Omega\times(0,\tau))&\hbox{if $\Omega$ is smooth,}\\
    \noalign{\vskip 1mm}
   C^{2,1}((\Omega\cup\Gamma)\times(0,\tau))\cap C^{1,0}(\overline\Omega\times(0,\tau))
    &\hbox{if $\Omega$ is a cylinder \eqref{defCyl}.}
   \end{cases}
\ee
 For $p\in[1,\infty)$ and $c_0\in L^p(\Omega)$, by a classical solution of \eqref{pbmP} on $[0,\tau)$,
 we mean a function $u\in E_\tau\cap C([0,\tau);L^p(\Omega))$ with $c(\cdot,0)=c_0$, which solves, in the pointwise sense,
  the PDE in $\Omega\times(0,\tau)$ and the boundary conditions
  on $\Gamma\times(0,\tau)$.
  \end{defi}

 We begin with the range $m\ge 1$. We note that the case $m=1$
corresponds to a quadratic growth of the nonlinear part of the equation.
It will turn out from the results below that this value plays a critical role in the problem.

 \goodbreak

\begin{theorem} \label{thm1}
	Let $1\le m\le p<\infty$ 
	and assume $f(c)=|c|^{m-1}c$ and $c_0\in L^p(\Omega)$.
\begin{itemize}[topsep=-1pt]\setlength\itemsep{-1pt}
\item[(i)] Problem \eqref{pbmP} admits a unique maximal classical solution
$$
	c\in  E_{T^*}
	\cap C([0,T^*);L^p(\Omega)).
$$
\item[(ii)]  If $c_0\ge 0$, $c_0\not\equiv 0$, then $c>0$ in $(\Omega\cup\Gamma)\times(0,T^*)$.
	\smallskip
\item[(iii)] If $p>m$ and $T^*<\infty$, then $\lim_{t\to T^*}\|c(t)\|_p=\infty$.
\end{itemize}
\end{theorem}

\begin{remark} 
Although it remains an open problem whether local existence may fail when $1\le p<m$, 
our results exhibit the critical role played by the $L^m$ norm.
Indeed, the continuation property (iii) fails for $1\le p<m$, since the $L^p$ norms may remain bounded as $t\to T^*<\infty$,
as a consequence of Theorem~\ref{BUprofileaisym}(ii) below. 
It also fails for $p=m=1$, in view of the mass conservation property 
(cf.~\eqref{Conserv} below). However, interestingly, for $m=1$ and positive solutions, 
the entropy blows up whenever $T^*<\infty$, 
namely
\[
\lim_{t\to T^*} \int_\Omega (c\log c)(t)\dx =\infty
\]
(this follows from the proof of Theorem~\ref{thm2sharp} below).

The condition $p \ge m$ in Theorem~\ref{thm1} is actually also natural in view of the scaling properties of the problem 
(see subsection~\ref{SubsecScaling} and cf.,~e.g.,~\cite[pp.158-159]{QSb} and
\cite{CW98} for similar situations in other evolution PDE's,
such as the Fujita, the nonlinear Schr\"odinger and the Navier-Stokes equations).

It is also worth noting
from Theorem~\ref{thm1} that the $L^1$ scaling critical case belongs to the local existence range,
unlike for the Fujita equation (see \cite{BC96, CZ}).
\end{remark}

We next consider the lower range $m<1$ for the growth of $f$ at infinity. This corresponds to a subquadratic growth
	of the nonlinear part of the equation. It turns out that, in this case,
local {\it and global} existence holds for any nonnegative $L^1$ initial data. Moreover, exponential stabilization to a constant 
 steady state also occurs 
for small mass.

\goodbreak
Let $\bar c_0=\frac{1}{|\Omega|}\int_\Omega c_0\dx$ denote the average of $c_0$.
\begin{theorem} \label{thm2}
	Let $f\in  C^1([0,\infty))$ satisfy 
		\be\label{hypfm}
		0\le f(s)\le Cs^m,  \ \ s\ge 0, \quad\hbox{ for some $m\in(0,1)$}
		\ee
		 with some $C>0$, and also assume 
	that $f$ is globally Lipschitz continuous.
	Let $c_0\in L^1(\Omega)$, with $c_0\ge 0$.
			\smallskip
			
\begin{itemize}[topsep=-1pt]\setlength\itemsep{-1pt}
\item[(i)] Problem \eqref{pbmP} admits a unique 
global, nonnegative classical solution
	\be\label{regulc}
	c\in  E_\infty
	\cap C([0,\infty);L^1(\Omega)).
	\ee
	Moreover,  we have
		\be\label{boundeps}
	\ds\sup_{t\ge\eps}\|c(t)\|_\infty<\infty,\quad\hbox{for each $\eps>0$.}
		\ee
	Also, if $c_0\not\equiv 0$, then $c>0$ in $(\Omega\cup\Gamma)\times(0,\infty)$.
		\medskip
		
\item[(ii)] There exist $\eta_0,\lambda>0$ such that if $\|c_0\|_1\le \eta_0$, then 
	\be\label{stabilc}
	\|c(t)-\bar c_0\|_\infty\le Ce^{-\lambda t},\quad t\ge 1 
	\ee
	for some constant $C>0$. 
\end{itemize}
\end{theorem}

We shall next give conditions which determine the global existence or nonexistence of solutions for $m\ge 1$.
This will complete the picture in terms of the criticality of the value $m=1$.
We begin with the case $m=1$.

\begin{theorem} \label{thm2sharp} 	
Let $f(c)=c$, $c_0\in L^1(\Omega)$, $c_0\ge 0$ and set $M=\|c_0\|_1$.

\smallskip

\begin{itemize}[topsep=-1pt]\setlength\itemsep{-1pt}
\item[(i)] If $M\le 1$, then  $T^*=\infty$ and $c$ satisfies \eqref{boundeps}.

	\smallskip
	
\item[(ii)] Morever, if $M<1$, then 
\be\label{stabilc1}
	\lim_{t\to \infty} \|c(t)-\bar c_0\|_\infty=0,
\ee
\end{itemize}
	\end{theorem}

	It follows from Theorems~\ref{thm2sharp} and \ref{thmBU} 
	below that for $f(c)=c$, problem \eqref{pbmP} exhibits a critical mass phenomenon, with sharp  threshold $M=1$.
Namely, all initial data with mass $M\le 1$ yield global bounded solutions whereas,
for any mass $M>1$, finite time blow-up occurs for a large class of initial data in suitable domains.
This is reminiscent of the well-known situation for the $2d$ Keller-Segel system.
A main difference is however that the critical mass phenomenon is dimension-independent in case of problem \eqref{pbmP}.
Also, solutions with critical mass remain bounded,
unlike in the critical mass case for the $2d$ Keller-Segel system.

The proof of Theorem~\ref{thm2sharp} is delicate. Starting from the observation that the entropy stays bounded when $M\le 1$, it is based on the possiblity to control $L^p$ norms of the solution by means of the entropy,
	refining on semigroup arguments and smoothing effects from the proof of Theorem~\ref{thm1}. 
	In this process, an important step is a kind of $\eps$-regularity property. Namely (see Proposition~\ref{lem:decomp}), we establish a uniform lower bound on the classical existence time of the solution when the initial data is decomposed as the sum of a
	suitably small $L^1$ part and of a bounded part.

Our next theorem partially extends  Theorem~\ref{thm2sharp} to all $m\ge 1$, with an additional exponential stabilization property,
the mass being now replaced with the $L^m$ norm.
However we do not have an explicit sharp smallness condition in general.

\begin{theorem} \label{thm2b} 	
Let $m\ge 1$ and $f(c)=c^m$.
There exists $\eta_0>0$ such that, if $c_0\in L^m(\Omega)$, $c_0\ge 0$ satisfies $\|c_0\|_m\le \eta_0$, then  $T^*=\infty$.
Moreover, $c$ satisfies the exponential stabilization property \eqref{stabilc}.
\end{theorem}

\begin{remark}  
\begin{itemize}
\item[(i)]
The small data assumptions for convergence to the constant steady-state
 in Theorem~\ref{thm2}(ii) and Theorem~\ref{thm2b} are necessary.
 Indeed, considering problem \eqref{pbmP} with $n=1$ and $\Omega=(0,1)$, this a consequence of the following results (see~\cite{MeuSou2}),
 where $\mathcal{S}$ denotes the set of nonconstant steady-states and
$\mathcal{S}_M=\mathcal{S}\cap \{\|c\|_1=M\}$.

 	\smallskip
	
$\bullet$ Let $f(s)=c^m$ with $m\ge 1$.  
Then $\mathcal{S}\ne\emptyset$ and any $c\in\mathcal{S}$ satisfies $\|c\|_m=N_0:=m^{-1/m}$.
Moreover,  for $m>1$, we have $\mathcal{S}_M\ne\emptyset$ if and only if $M\in(0,N_0)$.
	\smallskip
	
$\bullet$ Let $f$ be such that $f(s)=s^m$ for $s$ large and some $m\in(0,1)$,
with $f\in C^1([0,\infty))$, $f$~concave and $f'$ bounded.
Then there exists $M_0>0$ such that $\mathcal{S}_M=\emptyset$ for $M\in(0,M_0)$,
and $\#\mathcal{S}_M\ge 2$ for $M>M_0$.
	\smallskip
	
\item[(ii)]The global Lipschitz assumption on $f$ in Theorem~\ref{thm2} (which, in view of \eqref{hypfm},
does not significantly restrict the behavior of $f$ at infinity),
is only used to prove uniqueness in $L^1$.
Actually, by suitable modifications of the proof of Theorem~\ref{thm2}, one can show that,
 except for the uniqueness statement, the result remains valid
for $f(c)=c^m$ with $m\in(0,1)$.
	\smallskip
	
\item[(iii)] A result related to Theorem~\ref{thm2sharp} was obtained in~\cite{CHMV}
for the analogous problem on the half-line $I=(0,\infty)$,
i.e.~$c_t=(c_x-c(0,t)c)_x$ with zero flux condition at $x=0$.
Namely, under the assumptions $0<c_0\in L^1(I,(1+x)dx)$, $c_0\log c_0\in L^1(I)$ and $M\le 1$, 
the existence of a global, suitable weak solution was proved (the classical regularity of the solution was not established).
Note that, although entropy is also used in \cite{CHMV}, our proof is quite different from the proof
in \cite{CHMV}, which does not use semigroup techniques nor smoothing effects and depends to some rather large extent on the one-dimensional nature of the problem. 
\end{itemize}
\end{remark}

\subsubsection{Blow-up}

We now provide sharp conditions under which blow-up occurs. 
To this end we specialize to the case of the domains
\be\label{hypcyl1}
\hbox{$\Omega=(0,L)\times B'_R\subset \R^n$ if $n\ge 2$, \ or $\Omega=(0,L)$ if $n=1$,}
\ee
with $B'_R=B_R(0)\subset \R^{n-1}$ and $L, R>0$.
As for the initial data we shall assume
\be\label{hypcyl2}
\hbox{$c_0\in C^1(\overline\Omega)$, $c_0\ge 0$ is axisymmetric with respect to $e_1$ (if $n\ge 2$),}
\ee
\be\label{hypcyl3}
\hbox{$\partial_{x_1}c_0\le 0$ and $\partial_{x_1}c_0\not\equiv 0$.}
\ee

\begin{theorem}\label{thmBU}
Assume $m\ge 1$,  $f(c)=c^m$ and \eqref{hypcyl1}-\eqref{hypcyl3}.
\begin{itemize}[topsep=-1pt]\setlength\itemsep{-1pt}
\item[(i)] 
Let $m=1$. If $M>1$, then the solution of \eqref{pbmP} 
blows up in finite time.
	\smallskip
	
\item[(ii)] 
Let $m>1$ and $M>0$. 
Assume that
\end{itemize}
\be\label{condBUc0}
\int_\Omega  x_1c_0\dx\le K:=
\begin{cases}
C_1M\min\bigl\{L,LR^{1-n}(M/M_0),R^{1-n}M^{m/(m-1)}\bigr\},
&\hbox{if $n\ge 2$} \\
\noalign{\vskip 1mm}
C_1M\min\bigl\{L,M^{m/(m-1)}\bigr\},
&\hbox{if $n=1$}
\end{cases}
\ee
\begin{itemize}[topsep=-1pt]\setlength\itemsep{-1pt}
\item[] where $C_1=C_1(n)>0$ and $M_0=\bigl\|\int_0^L c_0(x_1,\cdot)dx_1\bigr\|_{L^\infty(B'_R)}$.
Then the solution of \eqref{pbmP} 
blows up in finite time.
\end{itemize}
\end{theorem}

\begin{remark}
\begin{itemize}
\item[(i)] 
Condition \eqref{condBUc0} is satisfied for a large class of initial data $c_0$
 which verify \eqref{hypcyl2}-\eqref{hypcyl3} and are sufficiently concentrated near $\partial\Omega\cap\{x_1=0\}$.
For instance, for any compactly supported, nontrivial $g\in C^1([0,\infty))$ with $g'\le 0$, and any radially symmetric 
$h\in C^1(\overline B'_R)$, it suffices to consider the initial data $c_0(x)=kg(kx_1)h(x')$
with $k>0$ sufficiently large, depending only on $\psi$ (note that $M=\|g\|_1\|h\|_1$ is independent of $k$).
	
	\smallskip
	
\item[(ii)] The blow-up condition $M>1$ is optimal for $m=1$ in view of Theorem~\ref{thm2sharp}.
The case $m>1$ is strikingly different, since blow-up solutions are shown to exist for arbitrary positive mass.
As a partial result in the direction of Theorem~\ref{thmBU}(i), it was shown in~\cite{CHMV} that if $m=n=1$, $M>1$ 
and \eqref{hypcyl3} holds, then blow-up occurs under the additional assumption $\int_0^L xc_0(x)dx<M/4$.
As in \cite{CHMV}, the proof of Theorem~\ref{thmBU} relies on the first moment $\phi(t)=\int_\Omega  x_1c(t)\dx$ of the solution.
Here, by a more refined analysis of its time variation, we can remove the additional assumption $\phi(0)<M/4$,
as well as extend the result to higher dimensions and to $m>1$.
	
	\smallskip
	
\item[(iii)] Assumptions \eqref{hypcyl1}-\eqref{hypcyl3} guarantee that $\partial_{x_1}c\le 0$ and that the advective 
vector field is colinear to $e_1$, 
properties which are needed in our proof of blow-up.
Regarding the second part of  \eqref{hypcyl3} we note that, if $\partial_{x_1}c_0\equiv 0$, then $c$ exists globally
under assumptions \eqref{hypcyl1}-\eqref{hypcyl2}.
Indeed it is easy to check that $\partial_{x_1}c(\cdot,t)\equiv 0$ and that $c$ 
just solves the heat equation in $B'_R$ with homogeneous Neumann boundary conditions.
	
	\smallskip
	
\item[(iv)] We are so far unable to prove blow-up in more general bounded domains.
The main difficulty is to establish suitable spatial monotonicity of the solution,
in view of the nonlocal, nonlinear nature of the convection term and Neumann boundary conditions.
\end{itemize}
\end{remark}

To conclude this subsection, we observe that 
the occurence of blow-up is 
also conditioned by the sign of the convection term.
Indeed, the following result shows that no blow-up can occur for $f(c)=-c^m$ with $m\ge 1$.

\begin{theorem} \label{thm1aneg}
	Let $1\le m\le p<\infty$ and assume $f(c)=-c^m$. Let $c_0\in L^p(\Omega)$, $c_0\ge 0$, $c_0\not\equiv 0$.
Then problem \eqref{pbmP} admits a unique, global classical solution
	$$		c\in  E_\infty
			\cap C([0,\infty);L^p(\Omega)),$$
	with $c>0$ in $(\Omega\cup\Gamma)\times(0,\infty)$.
		Moreover,   we have $\ds\sup_{t\ge\eps}\|c(t)\|_\infty<\infty$ for each $\eps>0$ 
		and $\lim_{t\to\infty} \|c(t)-\bar c_0\|_\infty=0$. 
\end{theorem}

\subsubsection{Asymptotic behavior of blow-up solutions}

For general solutions, it is easy to see that 
$$\limsup_{t\to T^*}\, \ds\Bigl|\int_{\partial\Omega}f(c)\nu \dsigma\Bigr|=\infty$$
(otherwise, the solution could be extended by means of linear estimates).
Since the mass is preserved, it is to be expected that blow-up singularities should occur only near the boundary,
i.e.~the solution remains bounded away from the boundary.
Under assumptions \eqref{hypcyl1}-\eqref{hypcyl3}, we can rigorously confim this,
along with upper estimates of the spatial blow-up profile.

\begin{theorem}\label{BUprofileaisym}
Assume $m\ge 1$,  $f(c)=c^m$ and \eqref{hypcyl1}-\eqref{hypcyl3}.

\begin{itemize}
\item[(i)] Assume $n\ge 2$. Then we have
\be\label{estimKx1}
c(x,t)\le M_0x_1^{-1}\quad\hbox{ in $\Omega\times(0,T^*)$}.
\ee
In particular, if $T^*<\infty$ then the blow-up set is a subset of $\partial\Omega\cap\{x_1=0\}$.
	\medskip
	
\item[(ii)] Assume $n=1$. Then we have 
 \be\label{profileBU}
 c(x,t) \le (mx)^{-1/m}+K_0,\quad\hbox{in $(0,L]\times(0,T^*)$}.
 \ee
for some constant $K_0=K_0(c_0)>0$.
In particular, if $T^*<\infty$ then $0$ is the only blow-up point.
\end{itemize}
\end{theorem}

\begin{remark} 
\begin{itemize}
\item[(i)]Estimate \eqref{profileBU} is 
optimal in some sense: one cannot replace $x^{-1/m}$ by $x^{-\alpha}$ with $\alpha<1/m$ 
(since the $L^p$ norm with $p>m$ must blow up as $t\to T^*$).
We do not know whether the exponent in estimate \eqref{estimKx1} can be improved to $1/m$.
The precise spatial shape of the solution at $t=T^*$ is also an open problem.
	\smallskip
	
\item[(ii)]
Under assumptions \eqref{hypcyl1}-\eqref{hypcyl3}, if moreover $c_0$ is decreasing with respect to $\rho=|x'|$, then it is not difficult to show that this remains true for the solution~$c$.
It follows that $c(\cdot,t)$ has a unique maximum at the origin,
so that in particular the origin is a blow-up point.
It is an open question whether there are solutions blowing up only at the origin.
Similarly, if $c_0$ is increasing with respect to $\rho=|x'|$, 
then so is $c$ and $c(\cdot,t)$ takes its maximum at every point of the edge $x_1=0, \rho=R$
(and only there).
\end{itemize}
\end{remark}

Regarding the time rate of blow-up, we have the following lower estimates in general domains.

\begin{theorem} \label{thmBUrate}
Let $f(c)=c^{m}$ with $m\ge 1$.
Let $c_0$ be as in Theorem~\ref{thm1} and assume $T^*<\infty$.
Then we have the lower blow-up estimate:
 \be\label{rateBU1}
 \|c(t)\|_\infty \ge C_1(T^*-t)^{-1/2m}, \ 0<t<T^*,
 \ee
for some constant $C_1(\Omega,m,a)>0$.
Moreover, we have
 \be\label{rateBU2}
\int_0^{T^*}\Bigl|\int_{\partial\Omega}c^m\nu \dsigma\Bigr|^2 \dt=\infty.
  \ee
\end{theorem}

\begin{remark} 
\begin{itemize}
\item[(i)] The upper time rate estimate is a (possibly difficult) open problem.
	\smallskip
	
\item[(ii)] Under the assumptions of Theorem~\ref{thm1}(ii), for the $L^p$ norms, we also have the lower blow-up estimate
$$\|c(t)\|_p\ge C_\eps(T^*-t)^{-\frac{p-m}{2pm} -\eps},\quad t\to T^*,$$
 for any $\eps>0$. In the critical case $p=m\ge 1$, we do not know whether the $L^p$ norm blows up at $T^*$
(except for $m=1$ for which it stays bounded).
See \cite{BC96, MiSou, MT} and the references therein for studies on this question for the Fujita equation.
However, the proof of Theorem~\ref{thm1} shows that a blow-up solution cannot stay in a compact set of $L^m$. 
More precisely, this follows from the property:
\be\label{compactexist}
\begin{aligned}
&\hbox{for any compact set $\mathcal{K}$ of $L^m(\Omega)$,
there is a uniform time $T=T(\mathcal{K})$} \\
&\hbox{such that, for any $c_0\in \mathcal{K}$, the solution of  \eqref{pbmP} exists on $[0, T]$.}
\end{aligned}
\ee
\end{itemize}
\end{remark}

\subsection{Contributions of this work}

After a nondimensionalization of the model, see Section~\ref{Sec:Bio}, we obtain equation \eqref{pbmP}.
To be rich enough to describe the different migratory behaviors of the cells, the model \eqref{pbmP} must give rise to different behaviors depending on the function $f$ and the mass $M$. In this work, we investigate whether both static and moving solutions can be observed. Roughly speaking, the static (resp.,~moving) solution  describes the average behavior of Brownian  (resp.,~ballistic) cells.  Theorems \ref{thm2} and \ref{thm2b} show that the function $f$ and the mass $M$ play an important role. Indeed, when $m \in (0,1)$ and if $M$ is small enough, the cell velocity converges to zero. This corresponds to a Brownian behavior. It is the same when $m\ge 1$ and $\|c_0\|_m$ is small enough. Note that this latter condition implies in particular that $M$ is small enough. Moreover, the stabilization property \eqref{stabilc} gives information about the characteristics of the persistence. On the other hand Theorem~\ref{thmBU} shows that if $m=1$ and $M>1$ then the cell velocity becomes infinite in a finite time for a toy cell described as a cylinder.
This type of behavior corresponds to a persistent trajectory. These conclusions give a rigorous justification to the UCSP first introduced in \cite{Maiuri2} since the model we study here is the full model \eqref{pbmP} and not only the one at steady state.

In a forthcoming 
 paper \cite{MeuSou2}, we study the situation of  bounded increasing 
nonlinearities $f$ involving saturation effects, typically of the form $f(c)=Lc/(c+\alpha)$
with $L,\alpha>0$. In such a case one can prove the existence  of traveling wave solutions. This will complete the picture and validate the interest of this model to describe the different migratory modes of the cells.

\subsection{Outline} 
This work is organized as follows. We give some biological justification and present some basic properties of 
problem \eqref{pbmP} in Section~\ref{Sec:Bio}. Section~\ref{sec:pre} contains preliminaries that will be used in the proofs.  Sections~\ref{Sec:proof_th1} and \ref{sec:proof:theorems} contain, respectively, the proofs of Theorem~\ref{thm1} and of Theorems~\ref{thm2}, \ref{thm2b} and \ref{thm1aneg}. Theorem~\ref{thm2sharp} is proved in Section~\ref{sec:sharp}. The blow-up is studied in Section~\ref{sec:blowup}, where Theorem~\ref{thmBU} is proved, and in Section~\ref{sec:blowup:asym}, where the results on blow-up asymptotics are established 
(Theorems~\ref{BUprofileaisym} and \ref{thmBUrate}). 
 Finally we state and prove in Appendix some auxiliary 
results on the solvability and regularity of linear and nonlinear inhomogeneous Neumann problems
in the case of nonsmooth (cylindrical) domains,
which are important for many of our results and for which we have been unable to find a suitable reference.
We note that the regularity results there are specific to the case of cylindrical domains and fail in general Lipschitz domains.

\smallskip

{\bf Acknowledgement.} 
The second author is partially supported by the Labex Inflamex (ANR project 10-LABX-0017).
 The authors are grateful to the referees for careful reading and very useful comments.

\section{Model construction and first properties}\label{Sec:Bio}

In this section we justify, from a biological point of view, the interest of  \eqref{pbmP} and we derive some properties. 

\subsection{Biological justification - Model construction }

Cell motility is a prime example of self-propulsion and one of the simplest example of active system. Here we fix the geometry of the cell described by $\Omega(t)$ with
\begin{equation}\label{eq:domaine_rigide}
	\Omega(t)=\Omega+\int_0^t u(s) \diff \! s,  
\end{equation}
where $\Omega$ is a smoothly bounded domain of $\R^n$, 
and $u(t)$ is the cell velocity, given by
\begin{equation}\label{eq:vit_rigid}
	u(t) = -\frac{\chi}{|\Omega|}\int_{\partial \Omega(t)}f(c) \nu \dsigma,
\end{equation}
where $\chi >0$ and $\nu$ denotes the outward unit normal of $\partial \Omega(t)$.

One of the main ingredients of the model relies on the key assumption that the value $v(t)$ of the actin flow is governed by the asymmetry of the cue concentration profile on the cell membrane: 
\begin{equation}\label{eq:actin}
v(t)= \int_{\partial \Omega(t)}f(c) \nu \dsigma , 
\end{equation}
where $f$ is a function that controls the intensity of the coupling between the actin flow and the asymmetry of the cue
concentration profile.

The phenomenological coupling \eqref{eq:actin} covers the cases where actin flows are generated by asymmetric distributions of either actin polymerization regulators (such as Arpin) or activators of contractility (such as Myosin II or a Myosin II activator), 
which are the main two %%% two main
scenarios that were proposed in \cite{Maiuri2}. We here do not aim at describing in details the biochemical steps involved in the process.

We recognize that \eqref{eq:vit_rigid} represents the external force balance on $\Omega(t)$. Hence, the parameter $\chi>0$ describes how the internal solute affects the cell dynamics, namely its velocity.

We describe now the internal solute dynamics and its coupling with the actin flow. The internal solute transport problem is formulated as follows: 
\begin{eqnarray}
	& \partial _t c(x,t)= \cn \left( \nabla c(x,t) + (a-1)c(x,t) u(t) \right) \qquad &  \mbox{ in } \Omega(t), \label{eq:marqueur_rigid}\\
	& \left(\nabla c(x,t) +ac(x,t)u(t) \right)\cdot \nu=0& \mbox{ on } \partial \Omega(t),\label{eq:marqueur_bord_rigid} 
\end{eqnarray}
where $a\in[0,1]$.

In the bulk of the cell, $\Omega(t)$,  fast adsorption on an adhered cortex is assumed. With rapid on and off rates, the transport dynamics are given by \eqref{eq:marqueur_rigid}-\eqref{eq:marqueur_bord_rigid}  where $a$ is the steady fraction of adsorbed molecules not convected by the average actin flow and the effective diffusion coefficient is assumed to be 1.

If we consider the equation in the moving frame of the cell rather than the fixed frame of the lab, we obtain
\begin{eqnarray}\label{mod:bio_1}
	& \partial _t c= \cn \left( \nabla c - \frac{a\chi}{|\Omega|}c \int_{\partial \Omega}f(c) \nu \dsigma \right) \qquad &  \mbox{ in } \Omega, 
	\\
	& \left(\nabla c  -\frac{a\chi}{|\Omega|}c\int_{\partial \Omega}f(c) \nu \dsigma\right) \cdot \nu=0& \mbox{ on } \partial \Omega.\label{mod:bio_2}
\end{eqnarray}
Note that problem \eqref{mod:bio_1}-\eqref{mod:bio_2} is a rigid version of an hydrodynamic model of polarization, migration and deformation of a living cell confined between two parallel surfaces first introduced in \cite{LMVC}. In this latter model, the cell cytoplasm is an out of equilibrium system thanks to the active forces generated in the cytoskeleton. The cytoplasm is described as a passive viscous droplet in the Hele-Shaw flow regime. 
Here, we consider the situation where the surface tension goes to infinity. Informally, this amounts to fixing the geometry of the cell and considering only one condition at the boundary in the model of \cite{LMVC}. Problem \eqref{pbmP} corresponds to a nondimensionalised version of \eqref{mod:bio_1}-\eqref{mod:bio_2}.

\subsection{Conservation of the marker content, nondimensionalization and scaling}\label{SubsecScaling}

Let $M(t)$ denote the mass of  molecular content:
\[
M(t) =\int_{\Omega(t)} c(x,t)\dx. 
\]
On the boundary \eqref{eq:marqueur_bord_rigid} we impose no-flux condition  so that
\begin{align*}
	\frac{\di }{\dt} M(t) &=  \int_{\Omega} \partial _t c \dx = 0.
\end{align*}
Thus, in  \eqref{eq:marqueur_rigid}-\eqref{eq:marqueur_bord_rigid}, we formally have 
conservation of molecular content:
\begin{equation}\label{eq:mass_marker2}
	M(t)=M(0)=M.
\end{equation}

Problem \eqref{mod:bio_1}-\eqref{mod:bio_2} 
 with \eqref{eq:mass_marker2} has the following parameters $a>0$, $\chi>0$, $|\Omega|$, $M>0$. In the case where $f(c)=c^m$, considering $c=\lambda \tilde c$ with $\lambda ^m = \frac{|\Omega|}{a \chi}$, we obtain that $\tilde c$ is solution of \eqref{pbmP} with the only parameter $\tilde M = \int_{\Omega} \tilde c \dx$.

Finally it is useful to identify the scale invariance of the problem and the scale invariant norm.
For $k\in\R$, consider the scaling transformation
\begin{equation}\label{scalingtransfo}
c_\lambda(x,t)=\lambda^kc(\lambda x,\lambda^2t), 
\quad \lambda>0.
\end{equation}
Note that this definition of course requires scale invariant spatial domains, i.e.,~cones.
(However the effect of scaling on the problem is expected to be significant even in the bounded domain case.)
Thus consider the simplest half-space case $\mH=\{x\in\R^n;\ x_n>0\}$.
Assuming that the integral is well defined, we have
$$\partial_t c_\lambda-\Delta c_\lambda=-\lambda^{k+2} \Bigr[\nabla c(\lambda x,\lambda^2t)\cdot \Bigr(\ds\int_{\partial\mH}c^m(y,\lambda^2t) e_n \dsigma_y\Bigr) \Bigr].$$
Since
$$\int_{\partial\mH}c^m(y,\lambda^2t) e_n \dsigma_y
=\lambda^{n-1}\int_{\partial\mH}c^m(\lambda z,\lambda^2t) e_n \dsigma_z
=\lambda^{n-1-km}\int_{\partial\mH}c^m_\lambda(z,t) e_n \dsigma_z,$$
we obtain
$$\partial_t c_\lambda-\Delta c_\lambda=-\lambda^{n-km}  \nabla c_\lambda\cdot \Bigr(\ds
\int_{\partial\mH}c^m_\lambda(z,t) e_n \dsigma_z\Bigr).
$$
Therefore, the PDE in \eqref{pbmP} with $f(c)=\pm c^m$ 
is invariant by the transformation \eqref{scalingtransfo} with $k=n/m$.
We also observe that the $L^q$-norm is invariant by this transformation (only) for $q=m$, namely:
$$\|c_\lambda(\cdot,0)\|_{L^m(\mH)}\equiv \|c(\cdot,0)\|_{L^m(\mH)},\quad \lambda>0.$$

\section{Preliminaries to the existence-uniqueness proofs}\label{sec:pre}

In this section we gather a number of semigroup and regularity properties that will be useful 
in the 
proofs of Theorems~\ref{thm1}, \ref{thm2} and \ref{thm2b}.
Denote by $S(t)$ the Neumann heat semigroup of $\Omega$,
by $G(x,y,t)$ the Neumann heat kernel,
and define the differentiated semigroup
\be\label{defKnabla}
[K_\nabla(t)\psi](x)=\int_\Omega \nabla_yG(x,y,t)\cdot\psi(y)\diff \! y,\quad t>0,\ \psi\in (L^1(\Omega))^n.
\ee
For $q\in [1,\infty)$ and $k\in [0,\infty)$, denote by $\|\cdot\|_{k,q}$ the norm of the 
Sobolev-Slobodecki space $W^{k,q}(\Omega)$.
We shall frequently use the trace embedding
\be\label{traceimb}
W^{k,q}(\Omega)\subset L^q(\partial\Omega),\quad 1\le q<\infty,\quad 
\ee
 with $k=q=1$ or $k>1/q$ if $q>1$
(valid in bounded Lipschitz domains, see \cite{JeK} for $q>1$ and, e.g., \cite{AMR} for $q=1$).
Our first lemma provides the basic linear smoothing estimates.

\begin{lemm}\label{lem:gaussestim}
	Let $T\in(0,\infty)$.
	
		(i) (Gaussian heat kernel estimates) For $i,j\in\{0,1\}$, $G$ satisfies
	\be\label{GaussEst}
	|D^i_xD^j_y G(x,y,t)|\le C(T)t^{-(n+i+j)/2} e^{-C_2|x-y|^2/t},\quad x,y\in\overline\Omega, \ 0<t<T.
	\ee
	
	(ii) Let $1\le q\le r\le\infty$ and $k\in [0,1]$. For all $\phi\in L^q(\Omega)$, we have
	\be\label{smooth1}
	\|S(t)\phi\|_{k,r}\le C(T)t^{-\frac{k}{2}-\frac{n}{2}(\frac{1}{q}-\frac{1}{r})}\|\phi\|_q,\quad 0<t<T
	\ee
	and, for all $\psi\in (L^q(\Omega))^n$,
	\be\label{smooth1b}
	\|K_\nabla(t)\psi\|_r\le C(T)t^{-\frac{1}{2}-\frac{n}{2}(\frac{1}{q}-\frac{1}{r})}\|\psi\|_q,\quad 0<t<T.
	\ee
	
	(iii) Let $q\in[1,\infty]$, $\ell\in[0,\frac{1}{q}]$ and $k\in[\ell,1]$. We have
	\be\label{smooth2}
	\|K_\nabla(t)\psi\|_{k,q}\le C(T)t^{-(1+k-\ell)/2}\|\psi\|_{\ell,q},\quad 0<t<T,\quad  \psi\in (W^{\ell,q}(\Omega))^n,
	\ee
	as well as
		\be\label{smooth3}
		\|K_\nabla(t)\psi\|_{1,q}\le C(T)t^{-\frac{1}{2}-\frac{n}{2}(1-\frac{1}{q})}\|\psi\|_{1,1},\quad 0<t<T,
		\quad  \psi\in (W^{1,1}(\Omega))^n.
		\ee
\end{lemm}

\begin{proof}
		(i) First consider the case when $\Omega$ is smooth.
		The case $j=0$ of \eqref{GaussEst} is classical, see e.g.~\cite[Section~2]{Mora} and the references therein. 
		The case $i=0$, $j=1$ follows from the case $i=1$, $j=0$
		and the fact that $\partial_{x_k}G(\xi,\zeta,t)=\partial_{y_k}G(\zeta,\xi,t)$, for all $k\in\{1,\dots,n\}$ and $\xi,\zeta\in\Omega$
		(owing to $G(\xi,\zeta,t)=G(\zeta,\xi,t)$).
		To check the case $i=j=1$, using the semigroup property $G(x,y,t)=\int_\Omega G(x,z,t/2)G(z,y,t/2) dz$ and differentiating, we write
		$$\partial_{x_k}\partial_{y_\ell}G(x,y,t)=\int_\Omega \partial_{x_k}G(x,z,t/2) \partial_{y_\ell}G(z,y,t/2) \diff \!  z.$$
		Applying the estimates for $\partial_{x_k}G$, $\partial_{y_\ell}G$ and next the convolution properties of Gaussians, we thus obtain
	$$|\partial_{x_k}\partial_{y_\ell}G(x,y,t)| \le C_1t^{-1}\int_\Omega t^{-n/2}e^{-C|x-z|^2/t} t^{-n/2}e^{-C|y-z|^2/t} \diff \!  z\le C_2t^{-(n/2)-1} e^{-C|x-y|^2/t}.$$
	
Now consider the case of a finite cylinder $\Omega=(0,L)\times B'_R$.
		It is well known that 
			\be\label{productG1G2}
		G(x,y,t)=G_1(x_1,y_1,t)G_2(x',y',t),  
		\ee
		where $x=(x_1,x')$ and $G_1$ and $G_2$ denote the respective kernels of $(0,L)$ and $B'_R$.
		The estimate then immediately follows from the estimates for $G_1, G_2$.

	(ii) Estimate \eqref{smooth1} for $k\in\{0,1\}$ is a direct consequence of \eqref{GaussEst} with $j=0$
	and the Young convolution convolution inequality. The general case follows by interpolation.
	As for \eqref{smooth1b}, it is a direct consequence of \eqref{GaussEst} with $i=0$, $j=1$
		and the Young convolution convolution inequality.

	(iii) We start with the case $\ell=0$.
		By \eqref{GaussEst} with $j=1$ and $i=0$ (resp., $i=1$), we have 
		$$\bigl|[D^iK_\nabla(t)\psi](x)\bigr|\le \int_\Omega |D_x^iD_yG(x,y,t)|\,|\psi(y)|\diff \! y
		\le C \int_\Omega\int_\Omega  t^{-(n+i)/2} e^{-C_1|x-y|^2/t}\,|\psi(y)| \diff \! y $$
		and using the Young convolution inequality, 
		we obtain \eqref{smooth2} with $\ell=0$ and $k=0$ (resp., $k=1$), and the case $\ell=0$, $k\in[0,1]$ follows by interpolation.
		
		We next establish the case $q=k=\ell=1$.
		Integrating by parts, we have
		\be\label{KnablaDecomp}
		[K_\nabla(t)\psi](x)
		=\int_{\partial\Omega} G(x,y,t)\psi(y)\cdot\nu \diff \!  \sigma-\int_\Omega G(x,y,t)\nabla_y\cdot \psi(y)\diff \! y \equiv L_1(x,t)+L_2(x,t).
		\ee
		Applying \eqref{GaussEst} with $i\in\{0,1\}$ and $j=0$ and using Fubini's theorem, we have
		$$\begin{aligned}
			\|D^iL_2(\cdot,t)\|_1
			&\le \int_\Omega\int_\Omega |D^i_xG(x,y,t)|\,|\nabla_y\cdot \psi(y)| \diff \! y  \diff \!  x \\
			& \le \int_\Omega\int_\Omega  t^{-(n+i)/2} e^{-C_1|x-y|^2/t}\,|\nabla_y\cdot \psi(y)| \diff \! y  \diff \!  x \\
			&\le Ct^{-i/2} \int_\Omega\Bigl(\int_\Omega t^{-n/2} e^{-C_1|x-y|^2/t}dx\Bigr)\,|\nabla_y\cdot \psi(y)| \diff \! y 
			\le Ct^{-i/2} \|\psi\|_{1,1},
		\end{aligned}$$
		and
		$$\begin{aligned}
			\|D^iL_1(\cdot,t)\|_1
			&\le \int_\Omega\int_{\partial\Omega} |D_x^iG(x,y,t)|\,|\psi(y)| \diff \!  \sigma \diff \! x
			\le \int_\Omega\int_{\partial\Omega}  t^{-(n+i)/2} e^{-C_1|x-y|^2/t}\,|\psi(y)| \diff \!  \sigma \diff \!  x \\
			&\le Ct^{-i/2} \int_{\partial\Omega}\Bigl(\int_\Omega t^{-n/2} e^{-C_1|x-y|^2/t}dx\Bigr)\,|\psi(y)| \diff \!  \sigma \\
			&\le Ct^{-i/2} \int_{\partial\Omega}|\psi(y)| \diff \!  \sigma
			=Ct^{-i/2}\|\psi\|_{L^1(\partial\Omega)} \le Ct^{-i/2}\|\psi\|_{1,1},
		\end{aligned}$$
		where we also used the trace inequality in the last step.
		This guarantees \eqref{smooth2} for $q=k=\ell=1$.
		
		The case $q\in(1,\infty)$, $\ell\in(0,\frac{1}{q}]$, $k\in[\ell,1]$ follows by interpolation
		between the above two cases. Namely, choosing 
		$s=\frac{k-\ell}{1-\ell}\in[0,1]$ and $p=\frac{1-\ell}{(1/q)-\ell}\in(1,\infty]$
		and observing that
		$k=\ell\cdot 1+(1-\ell)s$ and $1/q=\ell+(1-\ell)/p$,
		we write
		$$\begin{aligned}
			\|K_\nabla(t)\|_{\mathcal{L}(W^{\ell,q},W^{k,q})}
			&\le \|K_\nabla(t)\|_{\mathcal{L}(W^{1,1},W^{1,1})}^\ell \|K_\nabla(t)\|_{\mathcal{L}(L^p,W^{s,p})}^{1-\ell} \\
			&\le Ct^{-\ell/2}t^{-(1+s)(1-\ell)/2}=Ct^{-(1+k-\ell)/2}.
		\end{aligned}$$

Finally, \eqref{smooth3} with $q=\infty$ is a direct consequence of \eqref{GaussEst} with $i=1, j=0$, \eqref{KnablaDecomp},
	and the trace inequality.
	The general case follows by interpolating between $q=1$ 
 	i.e., \eqref{smooth2} with $q=k=\ell=1$) and $q=\infty$.
\end{proof}

\begin{remark}
By minor modifications of the above proof, 
estimates \eqref{GaussEst}-\eqref{smooth3} can be made independent of $T$, upon replacing
the factors of the form $C(T)t^{-\alpha}$ by $C_1(1+t^{-\alpha})$, with $C_1=C_1(\Omega)>0$.
\end{remark}

Recall the notation
\be\label{NewRepres2}
A(t)=\int_{\partial\Omega}f(c(\cdot,t))\nu \dsigma.
\ee
Our next lemma provides a convenient 
representation formula for problem \eqref{pbmP} in terms of the operator $K_\nabla$.

\begin{lemm} \label{lem:repres}
Assume $f$ continuous and
\be\label{hypHoldersmoothin1}
|f(s)|\le C|s|^m,\quad s\in\R,
\ee
for some $m,C>0$. Let $\max(m,1)\le p<\infty$, 
\be\label{casesgamma}
\begin{cases}
1/p<\gamma<2/m,&\hbox{if $p>1$,} \\
\noalign{\vskip 1mm}
\gamma=1,&\hbox{if $p=1$.}
\end{cases}
\ee
Let $T>0$, $c_0\in L^p(\Omega)$ and let $c$ be a classical solution of \eqref{pbmP}, with 
\be\label{regulRepres}
	c\in  E_T
\cap C([0,T);L^p(\Omega)),\quad 
\sup_{t\in(0,T)} t^{\gamma/2}\|c(t)\|_{\gamma,p}<\infty.
\ee
Then we have
\be\label{NewRepres}
c(t)=S(t)c_0+\int_0^t K_\nabla(t-s)[A(s) c(\cdot,s)]\diff \! s,\quad 0<t<T,
\ee
where the integral in \eqref{NewRepres} is absolutely convergent in $L^p(\Omega)$.
\end{lemm}

\begin{proof}
First rewrite \eqref{pbmP} as
$$\begin{aligned}
	\hfill\partial_t c-\Delta c&=-A(t)\cdot\nabla c,
	 &&x\in\Omega,\ 0<t<T,\\
	\hfill \nabla c \cdot \nu 
	&=(A(t)\cdot \nu)c,&&x\in\partial\Omega,\ 0<t<T, \\ 
	\hfill c(x,0)&=c_0(x),&&x\in\Omega.
\end{aligned}
$$
For fixed $\tau\in(0,T)$, setting $c_\tau(t)=c(t+\tau)$, $A_\tau(t)=A(t+\tau)$,
the standard Green representation formula guarantees that
$$ 
\begin{aligned}
c_\tau(x,t)
&=\int_\Omega G(x,y,t)c(y,\tau)\diff \! y-\int_0^t \int_\Omega G(x,y,t-s) \bigl(A_\tau(s)\cdot\nabla c_\tau(y,s)\bigr) \diff \! y\diff \! s \\
&\qquad +\int_0^t \int_{\partial\Omega} G(x,y,t-s) (A_\tau(s)\cdot\nu)c_\tau(y,s)\diff \! \sigma_y \diff \! s,
\quad x\in\Omega,\ 0<t<T-\tau.
\end{aligned}
$$ 
Next, using $A(s)\cdot\nabla c(y,s)=\nabla\cdot\bigl(c(y,s)A(s)\bigr)$ and the divergence theorem, we get
$$\int_{\partial\Omega} G(x,y,t-s) c_\tau(y,s)A_\tau(s)\cdot\nu\diff \! \sigma_y
=\int_\Omega \nabla_y\cdot \bigl(G(x,y,t-s)c_\tau(y,s)A_\tau(s)\bigr) \dy.$$
It follows that
$$
c_\tau(x,t)=\int_\Omega G(x,y,t)c(y,\tau)\diff \! y 
+\int_0^t \int_\Omega \nabla_yG(x,y,t-s)\cdot \bigl(c_\tau(y,s) A_\tau(s)\bigr)\diff \! y \diff \! s,
$$
i.e.,
\be\label{NewRepresTau}
c_\tau(t)=S(t)c(\tau)+\int_0^t K_\nabla(t-s)[A_\tau(s)c_\tau(s)]\diff \! s,\quad 0<t<T-\tau.
\ee
Now fix $t\in(0,T)$. In view of \eqref{regulRepres},
the first term of the RHS of \eqref{NewRepresTau} converges to $S(t)c_0$ in $L^p(\Omega)$ 
as $\tau\to 0$, 
and the integrand in the second term converges to $K_\nabla(t-s)[A(s)c(s)]$ in $L^p(\Omega)$ 
for each $s\in (0,t)$.
Moreover, by the trace embedding \eqref{traceimb}, assumption \eqref{hypHoldersmoothin1} and H\"older's inequality, we have
\be\label{ATauTrace}
|A_\tau(s)|\le C\||c_\tau(s)|^m\|_{L^1(\partial\Omega)}
\le C\|c_\tau(s)\|_{L^p(\partial\Omega)}^m\le C\|c_\tau(s)\|_{{\blue \gamma},p}^{m},\quad s+\tau<t.
\ee
Consequently, for all $\tau\in(0,(T-t)/2)$ and $s\in(0,t)$, we deduce from \eqref{smooth1b},
\eqref{casesgamma} and \eqref{regulRepres} that
$$\begin{aligned}
 \|K_\nabla(t-s)[A_\tau(s)c_\tau(s)]\bigr\|_p
 &\le C(t-s)^{-1/2}|A_\tau(s)|\bigl\|c_\tau(s)\bigr\|_p\\
 &\le C(t-s)^{-1/2}\|c(\tau+s)\|_{{\blue \gamma},p}^{m}\le C(t-s)^{-1/2}s^{-m{\blue \gamma}/2}\in L^1(0,t).
\end{aligned}$$
By dominated convergence, it follows that we may pass to the limit in $L^p(\Omega)$ as $\tau\to 0$ in \eqref{NewRepresTau},
which yields \eqref{NewRepres}.
\end{proof}

For the critical case $p=m$ we shall also need the following lemma,
which is a variant of \cite[Lemma 8]{BC96} (see also \cite{W80}).

\begin{lemm}\label{lem:compactsmoothin}
Let $p\in[1,\infty)$ and $k\in (0,1]$. Let $\mathcal{K}$ be a compact subset of $L^p(\Omega)$.
Set 
\be\label{compactsmoothing1}
\delta(T)=\delta_{k,p,\mathcal{K}}(T):=\sup_{\phi\in\mathcal{K},\, t\in(0,T)} t^{k/2}\|S(t)\phi\|_{k,p}.
\ee
Then $\lim_{T\to 0}\delta(T)=0$.
\end{lemm}

\begin{proof}
We first claim that \eqref{compactsmoothing1} holds when $\mathcal{K}$ is a singleton.
Let thus fix $\phi\in L^p(\Omega)$, and pick $\phi_j\in C^\infty_0(\Omega)$ such that $\phi_j\to\phi$ in $L^p(\Omega)$.
It is known that
\be\label{compactsmoothing2}
\|S(t)\psi\|_{1,p}\le C\|\psi\|_{1,p},\quad 0<t\le 1,\ \psi\in W^{1,p}(\Omega),
\ee
where $C=C(n,k,p)>0$
(for $\Omega$ smooth see, e.g.,~\cite[Theorem 51.1(iv) and Example 51.4(ii)]{QSb} 
and the references therein, and cf.~Lemma~\ref{lem:SGW} if $\Omega$ is a cylinder \eqref{defCyl}).
Fix $\eps>0$ and choose $j$ large enough so that $\|\phi-\phi_j\|_p\le\eps$. 
It follows from \eqref{smooth1} and \eqref{compactsmoothing2} that, for all $t\in(0,1)$,
$$\begin{aligned}
t^{k/2}\|S(t)\phi\|_{k,p}
&\le t^{k/2}\|S(t)(\phi-\phi_j)\|_{k,p}+ t^{k/2}\|S(t)\phi_j\|_{k,p} \\
&\le C\|\phi-\phi_j\|_p+Ct^{k/2}\|S(t)\phi_j\|_{1,p}\le C\eps+Ct^{k/2}\|\phi_j\|_{1,p}.
\end{aligned}$$
Consequently, for any $T\in (0,1)$, we have
$\delta_{k,p,\{\phi\}}(T)\le C\eps+CT^{k/2}\|\phi_j\|_{1,p}$.
Therefore,
$\displaystyle\limsup_{T\to 0}\delta_{k,p,\{\phi\}}(T)\le C\eps$, which proves the claim.

Next let $\mathcal{K}$ be a compact subset of $L^p(\Omega)$ and assume for contradiction that 
\eqref{compactsmoothing1} fails.
Then there exist $\eta>0$ and sequences $\phi_i\in \mathcal{K}$, $t_i\to 0$ such that
$t_i^{k/2}\|S(t_i)\phi_i\|_{k,p}\ge \eta.$
By compactness we may assume (passing to a subsequence) that $\phi_i\to \phi$ for some $\phi\in L^p$.
Using \eqref{smooth1}, we then have
$$\eta\le t_i^{k/2}\|S(t_i)\phi_i\|_{k,p}\le t_i^{k/2}\|S(t_i)\phi\|_{k,p}+t_i^{k/2}\|S(t_i)(\phi_i-\phi)\|_{k,p}
\le\delta_{k,p,\{\phi\}}(t_i)+C\|\phi_i-\phi\|_p\to 0$$
as $i\to\infty$ by the above claim: a contradiction.
\end{proof}

Our last 
lemma provides a useful Schauder regularity estimate for problem \eqref{pbmP}.

\begin{lemm} \label{lem:Holdersmoothin}
Let $f$ be continuous and satisfy \eqref{hypHoldersmoothin1}
for some $m,C>0$.
Let $\max(m,1)\le p<\infty$, $q>n$  and $\gamma$ satisfy \eqref{casesgamma}. 
 Let $T_0,N>0$, $T\in(0,\min(T_0,T^*))$, $\eps\in(0,T)$ 
 and let $c$ be a classical solution of \eqref{pbmP} on $[0,T]$ such that
\be\label{hypHoldersmoothin2}
\|c(t)\|_q+\|c(t)\|_{\gamma,p}\le N,\quad t\in [0,T].
\ee
 There exists $\alpha=\alpha(n)\in(0,1)$ such that, if $\Omega$ is smooth then
\be\label{regC2alpha}
\|c(t)\|_{C^{2+\alpha,1+(\alpha/2)}(\overline\Omega\times[\eps,T])}\le C(N,\eps,T_0)
\ee
 and, if $\Omega$ is a cylinder \eqref{defCyl} then, for any 
any compact subset $\Sigma\subset \Omega\cup\Gamma$,
\be\label{regC2alphaB}
 \|c\|_{C^{2+\alpha,1+\alpha/2}(\Sigma\times[\eps,T])}\le C(N,\eps,T_0,\Sigma),
\qquad \|c\|_{C^{1+\alpha,\alpha/2}(\overline\Omega\times[\eps,T])}\le C(N,\eps,T_0).
\ee
\end{lemm}

\begin{proof}
We first establish an $L^\infty$ estimate away from $t=0$. 
Similarly as \eqref{ATauTrace}, 
 recalling the notation \eqref{NewRepres2},
we have
\be\label{controlAB_0}
|A(t)|\le C N^m.
\ee
Using \eqref{smooth1b} and~\eqref{NewRepres}, we get
$$\begin{aligned}
\|c(t)\|_\infty
&\le \|S(t)c_0\|_\infty+\int_0^t |A(s)| \|K_\nabla(t-s)c(s)\|_\infty\diff \! s \\
&\le Ct^{-n/2q}\|c_0\|_q+C N^m\int_0^t (t-s)^{-\frac12-\frac{n}{2q}} \|c(s)\|_q \diff \! s,
\end{aligned}$$
hence
$$
\|c(t)\|_\infty\le C N t^{-n/2q}+C N^{m+1}t^{\frac12(1-\frac{n}{q})}.
$$
This guarantees that
\be\label{controlcepsT}
\|c\|_{L^\infty(\Omega\times[\eps,T])}\le C(N,\eps,T_0).
\ee

 Assume that $\Omega$ is smooth.
In this case we can apply parabolic Schauder regularity theory for the Neumann problem.
Fix any $\eps\in(0,T)$. Next setting $\tilde c_\eps(x,t)=t\,c(x,\eps+t)$, we see that $\tilde c_\eps$ 
satisfies the problem
$$L\tilde c_\eps:=\partial_t \tilde c_\eps-\Delta \tilde c_\eps+b_{1,\eps}\cdot\nabla \tilde c_\eps=b_{3,\eps}\ \hbox{ in $Q_{T-\eps}$, \quad
$\nabla\tilde c_\eps \cdot \nu=b_{2,\eps}\tilde c_\eps$ on $S_{T-\eps}$},$$
 with zero initial data, where 
 $$b_{1,\eps}(t)=A(\eps+t), \quad b_{2,\eps}(t)=A(\eps+t)\cdot\nu, 
 \quad b_{3,\eps}(x,t)=c(x,\eps+t).$$
Since $|b_{i,\eps}|\le C(N,\eps)$ owing to 
\eqref{controlAB_0}, \eqref{controlcepsT},
it follows from \cite[Theorem~6.44]{Lie} 
that, for some $\alpha=\alpha(n)\in(0,1)$, 
$\|\tilde c_\eps\|_{C^{\alpha,\alpha/2}(\overline\Omega\times[0,T-\eps])}\le C(N,\eps,T_0)$, hence
$$\|c\|_{C^{\alpha,\alpha/2}(\overline\Omega\times[\eps,T])}\le C(N,\eps,T_0).$$
Since now the $b_{i,\eps}$ satisfy H\"older bounds (depending on $N, \eps, T_0$),
it follows from \cite[Theorems~4.30]{Lie} 
that $\|\nabla \tilde c_\eps\|_{C^{\alpha,\alpha/2}(\overline\Omega\times[0,T-\eps])}\le C(N,\eps)$, hence
$$\|\nabla c\|_{C^{\alpha,\alpha/2}(\overline\Omega\times[\eps,T])}\le C(N,\eps).$$
Estimate \eqref{regC2alpha} then follows similarly from \cite[Theorem~4.31]{Lie}.

 Next consider the case when $\Omega$ is a cylinder \eqref{defCyl}.
We cannot apply the above and rely instead on the estimates in Proposition~\ref{locexist-cyl}(vi) from the Appendix.
Namely, viewing $c$ as the unique classical solution $\tilde c$,
given by Proposition~\ref{locexist-cyl}, of the initial boundary value problem
starting at initial time $\eps/2$, \eqref{regC2alphaB} follows from 
 estimates \eqref{regulCalpha} applied to $\tilde c$
 (with $\eps$ replaced by $\eps/2$), along with \eqref{controlcepsT}.
\end{proof}

\begin{remark} \label{remconservmass}
As formally observed above, for any classical solution 
$c\in  E_T\cap C([0,T);L^1(\Omega))$ of \eqref{pbmP},
we have the conservation property
\be\label{Conserv}
\int_\Omega c(x,t)\dx=\int_\Omega c_0(x)\dx,\quad 0<t<T,
\ee
as a consequence of
$\frac{\di}{\dt}\int_\Omega c(x,t)\dx=0$ on $(0,T)$.
 If $\Omega$ is smooth, this immediately follows by integrating the equation in space and using the boundary conditions.
 If $\Omega$ is a cylinder~\eqref{defCyl}, we have
\be\label{regulH1L2}
c\in  L^2_{loc}((0,T);H^2(\Omega))\cap H^1_{loc}((0,T);L^2(\Omega))
\ee
by Proposition~\ref{locexist-cyl} (applied at each initial time $\eps\in(0,T^*)$), so that the above remains valid.
\end{remark}

\section{Proof of Theorem~\ref{thm1}}\label{Sec:proof_th1}

In view of eventually establishing Theorems~\ref{thm2} and \ref{thm2b}, we shall prove
Theorem~\ref{thm1} jointly with the following proposition.

\begin{prop} \label{thm2prop}
	(i) Let $m\in(0,1)$ and let $f$ be globally Lipschitz continuous on $\R$ and satisfy 
	\be\label{Hypfsm}
	0\le f(s)\le C|s|^m,\quad s\in\R,
	\ee
	for some $C>0$.
	Let $c_0\in L^1(\Omega)$.
	Problem \eqref{pbmP} admits a unique, maximal classical solution
$$c\in  E_{T^*}\cap C([0,T^*);L^1(\Omega)),$$
	which also satisfies $c>0$ in $(\Omega\cup\Gamma)\times(0,T^*)$ if $c_0\ge 0$, $c_0\not\equiv 0$.
		Moreover, setting $q=n/(n-1)$ (any finite $q$ if $n=1$), we have $T^*>1$ and
		\be\label{regulc2a}
		t^{1/2}\|c(t)\|_q\le C(\Omega)\|c_0\|_1,\quad 0<t<1.
		\ee
 Furthermore, if $T^*<\infty$, then $\lim_{t\to T^*}\|c(t)\|_1=\infty$. 
	 	\smallskip
		
	 (ii) Let $m\ge 1$ and $f(s)=s^m$.
	Let $c_0\in L^m(\Omega)$ and assume that $\|c_0\|_m\le\eta_0$,
	where $\eta_0$ is from \eqref{defT0cases}.
	Problem \eqref{pbmP} admits a unique classical solution
	$c\in  E_1\cap C([0,1);L^m(\Omega))$.
	Moreover, setting $q=nm/(n-1)$ (any finite $q$ if $n=1$), we have 
	\be\label{regulc2b}
	t^{1/2m}\|c(t)\|_q\le C(\Omega,m)\|c_0\|_m,\quad 0<t<1.
	\ee
\end{prop}

\subsection{Proof of Theorem~\ref{thm1} and Proposition~\ref{thm2prop}: existence}
\label{subseclocexist}

{\bf Step 1.} {\it Approximating problem.} 
In the subcritical case $p>m$ (including $p=1>m$ for Proposition~\ref{thm2prop}), we fix $M\ge 1$, pick $c_0\in L^p$ with $\|c_0\|_p\le M$, 
and we select a sequence of initial data $c_{0,i}\in C^\infty_0(\Omega)$
 such that $c_{0,i}\to c_0$ in $L^p(\Omega)$ and $\|c_{0,i}\|_p\le 2\|c_0\|_p$.
  If $c_0\ge 0$, we may assume that the $c_{0,i}$ are nonnegative.

In the critical case $p=m\ge 1$, in view of proving also property \eqref{compactexist}
(this property will be useful also for the uniqueness part), 
we fix a compact set $\mathcal{K}$ of $L^m(\Omega)$. 
Since $\mathcal{K}$ can be covered by finitely many balls of given radius,
one can construct a sequence of functions $\phi_j\in C^\infty_0(\Omega)$
such that, for each given $c_0\in\mathcal{K}$, there exists a subsequence $c_{0,i}$ of $\phi_j$
which converges to $c_0$ in $L^m(\Omega)$. Moreover, 
\be\label{DeftildeK}
\tilde{\mathcal{K}}=\overline{\{\phi_j,\,j\in\N\}} {}^{L^m}
\ee
 is also compact in $L^m$.
 If $\mathcal{K}$ consists of nonnegative functions, we may furthermore assume that the $\phi_j$ are nonnegative.

 Problem \eqref{pbmP} with initial data $c_{0,i}$ admits a unique, maximal classical solution 
\be\label{Regulci}
\hbox{$c_i\in  E_{T_i}\cap C([0, T_i); W_q)$ 
for all $q\in[1,\infty)$,}
\ee
 with $W_q=W^{2,q}(\Omega)$ if $\Omega$ is smooth and 
$W_q=W^{1,q}(\Omega)$ if $\Omega$ is a cylinder \eqref{defCyl}.
Moreover, if its maximal existence time $T_i$ is finite, then $\lim_{t\to T_i} \|c_i(t)\|_\infty=\infty$.
 When $\Omega$ is smooth,
this follows, e.g.,~from (a minor modification of the proof of) \cite[Theorem~6.1]{Am},
noting that the required compatibility condition $\partial_\nu c_0 =(A(0)\cdot \nu)c_0 \ (=0)$ on $\partial\Omega$ is in particular satisfied.
See also \cite{LSU, Fr, Lie} for more classical approaches.
 In the case of nonsmooth (cylindrical) domains, for which we have been unable to find a suitable reference,
this is proved in Proposition~\ref{locexist-cyl} of the Appendix.

\smallskip
{\bf Step 2.} {\it Auxiliary norm estimates.} 
Let $0<T<\min(1,T_i)$.
 We need to distinguish several cases according to whether $p$ is larger or equal to $m$ or to $1$.
To this end we fix $\gamma, \theta, a$ such that
\be\label{casesgamma2}
\begin{cases}
\frac{1}{p}<\gamma<\min(1,\frac{1}{m}),\quad\theta=0,&\quad\hbox{if $p>\max(m,1)$,} \\
\noalign{\vskip 1mm}
\gamma=1,\quad\theta=\frac{1}{2p},&\quad\hbox{if $p=\max(m,1)$,} \\
\end{cases}
\qquad
a=
\begin{cases}
m,&\quad\hbox{if $p>m$,} \\
\noalign{\vskip 2mm}
1,&\quad\hbox{if $p=m$} 
\end{cases}
\ee
(note that $\gamma a\le 1$) and, for $\mu\in\{0,\theta\}$, we set
\be\label{NewSpace}
M_{i,T}=\sup_{t\in(0,T)} \|c_i(t)\|_p,
\quad N_{i,T}=\sup_{t\in(0,T)} t^{\gamma/2}\|c_i(t)\|_{\gamma,p},
\quad  L_{\mu,i,T}=N_{i,T}^{a+\mu}M_{i,T}^{m+1-a-\mu}.
\ee
We claim that
\be\label{estimMiTpart1}
M_{i,T} \le C_0 \|c_0\|_p + C_0 L_{0,i,T}T^{\frac12(1-\gamma a)},
\ee
and
\be\label{estimMiTpart2}
N_{i,T} \le \sup_{t\in(0,T)} t^{\gamma/2}\|S(t)c_{i,0}\|_{\gamma,p}+ 
C_0 L_{\theta,i,T} T^{\frac12(1-\gamma a)}
\ee
with some $C_0=C_0(\Omega,m,p,\gamma)>1$.

Let $k\in[0,1)$ and $0<t<T<\min(1,T_i)$. Set
$$A_i(t)=\int_{\partial\Omega}f(c_i(\cdot,t))\nu \dsigma.$$
By \eqref{smooth2} with $\ell=0$ and~\eqref{NewRepres}, we have
\be\label{NLestim1}
\|c_i(t)-S(t)c_{0,i}\|_{k,p}\le C\int_0^t (t-s)^{-(k+1)/2}|A_i(s)|\|c_i(s)\|_p \diff \! s.
\ee
Here and below, $C$ denotes a generic positive constant depending only on $\Omega, m, p, \gamma,k$.
On the other hand,  we observe that
\be\label{controlAB}
 |A_i(t)| \le C\||c_i(t)|^m\|_{L^1(\partial\Omega)} \le C\|c_i(t)\|^{m-a}_p\|c_i(t)\|^a_{\gamma,p}.
\ee
Indeed, by the trace embedding \eqref{traceimb}, H\"older's inequality and the boundedness of $\partial\Omega$, 
  if $p>m=a$ then we have
$\||c_i(t)|^m\|_{L^1(\partial\Omega)}
\le C\|c_i(t)\|_{L^p(\partial\Omega)}^m\le C\|c_i(t)\|_{{\blue \gamma},p}^{m}$.
  Whereas, if $p=m$, hence $a=\gamma=1$ then, using \eqref{traceimb} and $\|\nabla|c_i|^m\|_1=m\||c_i|^{m-1}\nabla c_i\|_1
\le m\|c_i\|_m^{m-1}\|\nabla c_i\|_m$, we get
$$ 
|A_i(t)| \le C \||c_i(t)|^m\|_{L^1(\partial\Omega)}\le C\||c_i(t)|^m\|_{1,1}
\le C\|c_i(t)\|^{m-1}_m\|c_i(t)\|_{1,m}.
$$
It follows from \eqref{NLestim1} and \eqref{controlAB} that 
$$\begin{aligned}
t^{k/2}\|c_i(t)-S(t)c_{0,i}\|_{k,p}
&\le C L_{0,i,T}t^{k/2} \int_0^t (t-s)^{-(k+1)/2} s^{-{\blue \gamma a}/2} \diff \! s\\
&\le C L_{0,i,T}t^{\frac12(1-{\blue \gamma a})} \int_0^1 (1-\tau)^{-(k+1)/2} \tau^{-{\blue \gamma a}/2} \diff \! \tau,
\end{aligned}$$
where the integrals are convergent, owing to $k<1$ and ${\blue \gamma a\le 1}$. 
Consequently,
\be\label{estimMiT1}
t^{k/2}\|c_i(t)-S(t)c_{0,i}\|_{k,p}\le C  L_{0,i,T}T^{\frac12(1-{\blue \gamma a})}, \ k\in[0,1),
\ \ 0<t<T. 
\ee
Taking $k=0$ and using \eqref{smooth1} we get \eqref{estimMiTpart1},
and, taking $k={\blue \gamma}$, we deduce \eqref{estimMiTpart2} for  $p>\max(m,1)$. .
Next consider \eqref{estimMiTpart2} in the case  $p=\max(m,1)$, hence $\gamma=1$, 
$\theta=\frac{1}{2p}$.
To overcome the divergence of the integral in 
\eqref{NLestim1} for $k=1$, we instead use \eqref{smooth2} 
with $\ell=1/2p$, $k=1$, $q=p$ and \eqref{NewRepres}, to write
$$\|c_i(t)-S(t)c_{0,i}\|_{1, p} \le C\int_0^t (t-s)^{-1+\frac{1}{4p}} |A_i(s)| \|c_i(s)\|_{1/2p,p} \diff \!  s.$$
Since $\|c_i(s)\|_{1/2p,p}\le \|c_i(s)\|_{1, p}^{1/2p}\|c_i(s)\|_{p}^{1-(1/2p)}$ by interpolation,  \eqref{controlAB} yields
$$ |A_i(s)| \|c_i(s)\|_{1/2p,p}\le 
C\|c_i(t)\|^{a+\frac{1}{2p}}_{1,p}\|c_i(t)\|^{m-a+1-\frac{1}{2p}}_p
\le L_{\theta,i,T}s^{-\frac{a}{2}-\frac{1}{4p}}.$$
We then have
$$\begin{aligned}
t^{1/2}\|c_i(t)-S(t)c_{0,i}\|_{1,p}
&\le C L_{\theta,i,T}t^{1/2}\int_0^t (t-s)^{-1+\frac{1}{4p}} s^{-\frac{a}{2}-\frac{1}{4p}}\diff \! s
\le C L_{\theta,i,T}T^{\frac12(1-{\red a})},
\end{aligned}$$
owing to $a\le 1$,
hence \eqref{estimMiTpart2}  for $p=\max(m,1)$.

We next consider the following two cases separately: supercritical or small data critical, and general critical case.

{\bf Step 3.} {\it Supercritical and small data critical cases.}
Set $\eta_0:=(2C_0)^{-(m+1)/m}$, where $C_0$ is from \eqref{estimMiTpart1}-\eqref{estimMiTpart2}. Assume either 
\be\label{defT0cases}
\begin{cases}
&\hbox{$p>m$} \\
\noalign{\vskip 1mm}
&\hbox{or} \\
\noalign{\vskip 1mm}
&\hbox{$p=m$, \ $\|c_0\|_p<\eta_0$,} 
\end{cases}
\qquad \hbox{ and set }\ 
T_0:=
\begin{cases}
&\hbox{$[1+(2C_0)^{m+1}\|c_0\|_p^m]^{-\frac{2}{1-\gamma a}}$} \\ 
\noalign{\vskip 1mm}
&\hbox{ } \\
\noalign{\vskip 1mm}
&\hbox{$1$} 
\end{cases}
\ee
(in the first case we have $\gamma a<1$ by \eqref{casesgamma2}, and the second case corresponds to Proposition~\ref{thm2prop}(ii)).

We first claim that 
\be\label{estimMiTb0}
M_{i,T}+N_{i,T}<2C_0\|c_0\|_p \ \ \hbox{for all $T\le\tilde T_i:=\min(T_0,T_i)$}.
\ee
Indeed, by \eqref{estimMiTpart1}, \eqref{estimMiTpart2} and \eqref{smooth1}, we have
\be\label{estimMiT}
M_{i,T}+N_{i,T}\le C_0\|c_0\|_p+C_0(M_{i,T}+N_{i,T})^{m+1}T^{\frac12(1-\gamma a)},\quad 0<T<\min(1,T_i).
\ee
Moreover, for each $i$,  by \eqref{Regulci} we have
\be\label{limsupMiT}
\limsup_{T\to 0} \,(M_{i,T}+N_{i,T}) \le \|c_{0,i}\|_p\le C_0\|c_0\|_p.
\ee
Therefore, if \eqref{estimMiTb0} fails then, by \eqref{estimMiT} and continuity there exist
$i$ and a minimal $T<\min(T_0,T_i)$ such that
$$M_{i,T}+N_{i,T}=2C_0\|c_0\|_p\le C_0\|c_0\|_p+C_0(2C_0\|c_0\|_p)^{m+1}T^{\frac12(1-\gamma a)},$$
hence (supposing $c_0\not\equiv 0$ without loss of generality) 
$1\le(2C_0)^{m+1}\|c_0\|_p^mT_0^{\frac12(1-\gamma a)}$: a contradiction.

 In this paragraph we consider the consider when $\Omega$ is smooth. 
We show by a bootstrap argument that, 
for each $\eps>0$, the solutions $c_i$ 
enjoy the uniform higher regularity estimate:
\be\label{estimciHolder}
\sup_i \|c_i\|_{C^{2+\alpha,1+(\alpha/2)}(\overline\Omega\times[\eps,\tilde T_i))}\le C(M,\eps).
\ee
To this end, setting $\tilde p=np/(n-1)$ (or any finite $\tilde p$ if $n=1$)
and using the Sobolev inequality, \eqref{estimMiTb0}  and $\gamma\ge 1/p$
guarantees that, for any $\eps>0$,
\be\label{estimciHolderbasic}
\|c_i(t)\|_{\tilde p}\le C\|c_i(t)\|_{1/p,p}\le C(M,\eps) \ \ \hbox{for all $t\in[\eps,\tilde T_i)$}.
\ee
On the other hand, by the argument leading to \eqref{estimMiTb0}, applied with $\tilde p$ instead of $p$, with $t_0$ as initial time
and $c_i(t_0)$ as initial data, it follows that 
$(t-t_0)^{1/2\tilde p}\|c_i(t)\|_{1/\tilde p,\tilde p}\le \tilde C_0\|c_i(t_0)\|_{\tilde p}$
for all $t, t_0$ such that $\eps\le t_0<t<\tilde T_i$
and $t-t_0\le s_0:=[1+(2\tilde C_0)^{m+1}\|c_i(t_0)\|_{\tilde p}^m]^{-\frac{2}{1-\gamma a}}$,
with $\tilde C_0=\tilde C_0(p,m,\Omega,\gamma)>0$.
By \eqref{estimciHolderbasic} we have $s_0\ge \bar s_0(\eps,M)>0$. For any $t\in [2\eps,\tilde T_i)$,
choosing $t_0=t-\min(\eps,\bar s_0(\eps,M))\ge \eps$, 
we deduce that
$\|c_i(t)\|_{1/\tilde p,\tilde p}\le C(M,\eps)\min(\eps,\bar s_0(\eps,M))^{-1/2\tilde p}=\tilde C(M,\eps)$.
Repeating iteratively the argument, we obtain, for any finite $q\in[p,\infty)$,
$$\|c_i(t)\|_{1/q,q}\le C(M,\eps,q) \ \ \hbox{for all $t\in[\eps,\tilde T_i)$.}$$ 
Property \eqref{estimciHolder}
is then a consequence of Lemma~\ref{lem:Holdersmoothin}.
It follows in particular from \eqref{estimciHolder} that $T_i>\tilde T_i$, i.e. $T_i>T_0$,
where $T_0$ is defined by~\eqref{defT0cases}.
We may then pass to the limit to show that (some subsequence of) $c_i$ converges
in $C^{2,1}_{loc}(\overline\Omega\times(0,T_0))$
to a classical solution of the PDE and of the boundary conditions in \eqref{pbmP} on $(0,T_0)$.
Owing to \eqref{estimMiTb0}, $c$ moreover satisfies
\be\label{estimMTb0}
\sup_{t\in(0,T_0)} t^{\gamma/2}\|c(t)\|_{\gamma,p}\le C\|c_0\|_p, 
\ee
and
\be\label{estimciHolderC}
\|c\|_{C^{2+\alpha,1+(\alpha/2)}(\overline\Omega\times[\eps,T_0])}\le C(M,\eps),
\quad 0<\eps<T_0.
\ee

 Now, if $\Omega$ is a cylinder \eqref{defCyl}, by the above arguments and \eqref{regC2alphaB} in 
Lemma~\ref{lem:Holdersmoothin}, we obtain
the convergence, in the space $C^{2,1}_{loc}((\Omega\cup\Gamma)\times(0,T_0))
\cap C^{1,0}_{loc}(\overline\Omega\times(0,T_0))$, of a subsequence $c_i$ to a classical solution $c\in E_{T_0}$.
As for estimate \eqref{estimciHolderC}, it becomes
\be\label{estimciHolderCcyl}
\|c\|_{C^{2+\alpha,1+\alpha/2}(\Sigma\times[\eps,T])}\le C(M,\eps,\Sigma),
\qquad \|c\|_{C^{1+\alpha,\alpha/2}(\overline\Omega\times[\eps,T])}\le C(M,\eps),
\ee
for any compact subset $\Sigma\subset \Omega\cup\Gamma$.

{\bf Step 4.} {\it Critical case $p=m\ge1$.}
In this case we need to treat the quantities $M_{i,T}$ and $N_{i,T}$ separately and make use of
Lemma~\ref{lem:compactsmoothin}. Recalling \eqref{DeftildeK},
by Lemma~\ref{lem:compactsmoothin}, we have
\be\label{estimMiTcrit0}
\delta(T)=\delta_{1,p,\tilde{\mathcal{K}}}(T)
:=\sup_{\phi\in\tilde{\mathcal{K}},\, t\in(0,T)} t^{1/2}\|S(t)\phi\|_{1,p}
\to 0,\quad\hbox{as } T\to 0.
\ee
We claim that there exists $T_0=T_0(\mathcal{K})\in(0,1)$ such that
\be\label{estimMiTcrit2}
M_{i,T}\le 2C_0\|c_0\|_p
\quad\hbox{and}\quad
N_{i,T}\le 2 \delta(T),\quad\hbox{for all $T<\min(T_i,T_0)$.}
\ee
 
 By \eqref{estimMiTpart1}, \eqref{estimMiTpart2}, with $\gamma=a=1$ and $\theta=\frac{1}{2p}$, we have
\be\label{estimMiTcrit}
M_{i,T}\le C_0\|c_0\|_p+C_0M_{i,T}^{p} N_{i,T}
\quad\hbox{and}\quad
N_{i,T}\le \delta(T)+C_0M_{i,T}^{p-\theta}N_{i,T}^{1+\theta},\quad 0<T<T_i.
\ee
By \eqref{estimMiTcrit0} 
there exists $T_0=T_0(\mathcal{K})\in(0,1)$ such that
$$\delta(T_0)\le \eta:=\min\bigl\{(2^{p+2}C_0^p\|c_0\|_p^{p-1})^{-1},
\bigl(2^{p+2}C_0^{p+1-\theta}\|c_0\|_p^{p-\theta}\bigr)^{-1/\theta}\bigr\}.$$
Let $\tilde T_i:=\min(T_i,T_0)$ and 
$T'_i=\sup\big\{T\in (0,\tilde T_i);\ 
M_{i,T}\le 2C_0\|c_0\|_p\ \hbox{and}\ N_{i,T}\le 2\delta(T_0)\big\}>0$. 
By the first part of \eqref{estimMiTcrit}, we have
$$M_{i,T}\le C_0\|c_0\|_p+2C_0\delta(T_0)(2C_0\|c_0\|_p)^p 
\le \ts\frac32 C_0\|c_0\|_p,\quad 0<T<T'_i.$$
Next, by the second part of \eqref{estimMiTcrit}, we deduce that
$$N_{i,T}\le \delta(T)+C_0(2C_0\|c_0\|_p)^{p-\theta}(2\delta(T_0))^\theta N_{i,T}
\le  \delta(T)+\ts\frac14 N_{i,T},\quad 0<T<T'_i,$$ 
hence $N_{i,T}\le \frac43 \delta(T)\le \frac43 \delta(T_0)$ for all $T<T'_i$.
Consequently, $T'_i=\tilde T_i$
since otherwise, by continuity, $M_{i,T}=2C_0\|c_0\|_p$ or $N_{i,T}=2\delta(T_0)$
for $T=T'_i<\tilde T_i$, a contradiction.
Claim \eqref{estimMiTcrit2} follows.

Next, again first considering the consider when $\Omega$ is smooth, 
we establish the uniform higher regularity 
\be\label{estimciHolderCrit}
\sup_i \|c_i\|_{C^{2+\alpha,1+(\alpha/2)}(\overline\Omega\times[\eps,\tilde T_i))}<\infty, \quad \eps>0.
\ee
To this end, a little more care is needed than in the supercritical case, because the existence time just obtained 
is uniform on compact subsets of initial data (but not on merely bounded subsets).
Thus starting from \eqref{estimMiTcrit2} and using the compactness of the imbedding
$W^{1,p}(\Omega)\subset L^{\hat p}(\Omega)$ where $\hat p=2np/(2n-p)$
(or any finite $\hat p$ if $p\ge 2n$), we see that, for each $\eps>0$,
$$\hat{\mathcal{K}}_\eps:=\overline{\{c_i(t),\  t\in[\eps,T_i),\, i\in\N\}} {}^{L^{\hat p}}$$
is a compact subset of $L^{\hat p}(\Omega)$.
By the argument leading to \eqref{estimMiTcrit2}, applied with $\hat p$ instead of $p$, with any $t_0\in[\eps,\tilde T_i)$ as initial time
and initial data $c_i(t_0)\in \hat{\mathcal{K}}_\eps$,  there exists $\hat T_0=\hat T_0(\hat{\mathcal{K}}_\eps)\in(0,1)$ such that,
$$\sup\,\Bigl\{s^{1/2}\|c_i(t_0+s)\|_{1,\hat p};\ t_0\ge \eps,\, 0<s<\min(\hat T_0,\tilde T_i-t_0),\, i\in \N\Bigr\}<\infty,$$
hence
$$\sup\,\Bigl\{\|c_i(t)\|_{1,\hat p};\ 2\eps\le t<\tilde T_i,\, i\in \N\Bigr\}<\infty.$$
It follows that $\|c_i(t)\|_{1,\hat p}\le C(K,\eps)$ for all $t\in[\eps,\tilde T_i)$.
Repeating iteratively the argument, we obtain, for any finite $q\in[p,\infty)$,
$$\sup\,\Bigl\{\|c_i(t)\|_{1,q};\ \eps\le t<\tilde T_i,\, i\in \N\Bigr\}<\infty.$$ 
Property \eqref{estimciHolderCrit} is then a consequence of Lemma~\ref{lem:Holdersmoothin}.
It follows in particular from \eqref{estimciHolderCrit} that $T_i>\tilde T_i$, i.e. $T_i>T_0$.
We may then pass to the limit to show that (some subsequence of) $c_i$ converges
in $C^{2,1}_{loc}(\overline\Omega\times(0,T_0))$
to a classical solution of the PDE and of the boundary conditions in \eqref{pbmP} on $(0,T_0)$.
Owing to \eqref{estimMiTcrit2}, it moreover satisfies
\be\label{estimMTb0crit}
\sup_{t\in(0,T)} t^{1/2}\|c(t)\|_{1,p}\le 2 \delta(T),\quad 0<T<T_0,
\quad\hbox{where } T_0=T_0(\mathcal{K})>0.
\ee
Now, if $\Omega$ is a cylinder \eqref{defCyl}, 
the above remains true with the same changes as in the last paragraph of Step~3.

{\bf Step 5.} {\it Continuity in $L^p(\Omega)$ at $t=0$.} 
From \eqref{estimMiT1} with $k=0$
and the fact that $T_i>T_0$, we obtain
\be\label{estimcontMN}
\|c_i(t)-S(t)c_{0,i}\|_p\le C_0 L_{0,i,T}T^{\frac12(1-{\blue \gamma a})} 
\quad \quad 0<t<T<T_0.
\ee

In the case $p>m$ 
(hence $\gamma a<1$ by \eqref{casesgamma2}), this along with \eqref{estimMiTb0} yields
$$\|c_i(t)-S(t)c_{0,i}\|_p\le C\|c_0\|_p^{m+1}t^{\frac12(1-{\blue \gamma a})},\quad i\in\N,\ 0<t<T_0.$$
Letting $i\to\infty$, we get
$$\|c(t)-S(t)c_0\|_p\le C\|c_0\|_p^{m+1}t^{\frac12(1-{\blue \gamma a})},\quad 0<t<T_0,$$
and since $\|c_0-S(t)c_0\|_p\to 0$ as $t\to 0$, the property follows.

In the case $p=m$,  
\eqref{estimcontMN} along with \eqref{estimMiTcrit2}, 
yields
$$\|c_i(t)-S(t)c_{0,i}\|_p\le C\|c_0\|_p^p\delta(t),\quad i\in\N,\ 0<t<T_0$$
and we conclude similarly by using \eqref{estimMiTcrit0}.
\qed

	We note that, at this stage (cf.~Steps 3 and 4), we have in particular obtained the following:
			
\begin{lemm}\label{lem:c0i} 
Under the assumptions of Theorem~\ref{thm1} or Proposition~\ref{thm2prop}(i),
there exist $T_0\in(0,T^*)$, a sequence 
			$c_{0,i}\in C^\infty_0(\Omega)$ and solutions $c, c_i$ of \eqref{pbmP} with initial data $c_0, c_{0,i}$,
			respectively, such that $c_i\in E_{T_0}\cap C(\overline Q_{T_0})$ and $c_i$ converges to $c$
in $C^{2,1}_{loc}((\Omega\cup\Gamma)\times(0,T_0))\cap C^{1,0}_{loc}(\overline\Omega\times(0,T_0))$.
If $p>m$, then we may take, for any $\eps>0$
$$T_0:=[1+C_\eps\|c_0\|_p^m]^{-\frac{2p}{p-m}-\eps}.$$
If $c_0\ge 0$, we may assume that the $c_{0,i}$ are nonnegative.

\end{lemm}

\subsection{Proof of Theorem~\ref{thm1} and Proposition~\ref{thm2prop}: uniqueness} 

We shall more precisely prove the following proposition,
which immediately implies the uniqueness part of Theorem~\ref{thm1} and Proposition~\ref{thm2prop}.

\begin{prop}\label{PropUniqLp}
Assume either $m\ge 1$, $f(s)=c^m$, and set $p=m$, or $m\in(0,1)$, $f$ as in Proposition~\ref{thm2prop}(i), and set $p=1$. 
Let $c_0\in L^p(\Omega)$, $T>0$ 
and $\tilde c, \hat c\in C([0,T];L^p(\Omega))\cap E_T$ 
be solutions of \eqref{pbmP}. Then $\tilde c\equiv \hat c$ on $[0,T]$.
\end{prop}

In view of the proof, for $p\ge 1$ and $T>0$, we define
$$X_{p,T}=\Bigl\{z\in C([0,T];L^p(\Omega))\cap C(0,T;W^{1,p}(\Omega));\ 
\sup_{t\in(0,T)} t^{1/2}\|z(t)\|_{1,p}<\infty\Bigr\}$$
and, for $z\in X_{p,T}$, we set 
$$M_T(z)= \sup_{t\in(0,T)} \|z(t)\|_p,
\qquad
N_T(z)= \sup_{t\in(0,T)} t^{1/2}\|z(t)\|_{1,p}.$$
We start with the following lemma.

\begin{lemm} \label{lem:stat_sol}
Let $f, m, p$ be as in Proposition~\ref{PropUniqLp}. Let $T\in(0,1)$, $c_0\in L^p(\Omega)$ and $c_1,c_2\in X_{p,T}\cap E_T$ 
be classical solutions of \eqref{pbmP} on $(0,T)$.
There exists $\eps_0(p,f,\Omega,\|c_0\|_p)>0$ such that, if
\be\label{Hypuniq1}
\bar N_T:=\max(N_T(c_1),N_T(c_2)) \le \eps_0,
\ee
then $c_1=c_2$ on $(0,T)$.
\end{lemm}

\begin{proof}
Set $\bar M_T:=\max(M_T(c_1),M_T(c_2))$
and $A_j(t)=\int_{\partial\Omega}f(c_j(\cdot,t))\nu \dsigma$
for $j\in\{1,2\}$.
Note that our assumptions guarantee that
%$|f(c)|\le C|c|^p$ 
$|f(c_1)-f(c_2)|\le C(|c_1|^{p-1}+|c_2|^{p-1})|c_1-c_2|$,
and that \eqref{traceimb} implies
$$\|c\|^p_{L^p(\partial\Omega)} =\||c|^p\|_{L^1(\partial\Omega)} \le C\bigl(\||c|^p\|_1+\|\nabla|c|^p\|_1\bigr)
\le C\|c\|_p^{p-1}\|c\|_{1,p}.$$
Using H\"older's inequality, it follows (omitting the variable $s$ for conciseness) that
\be\label{A1A2a}
\begin{aligned}
|A_1-A_2| 
&\le C\bigl(\|c_1\|_{L^p(\partial\Omega)}^{p-1}+\|c_2\|_{L^p(\partial\Omega)}^{p-1})\|c_1-c_2\|_{L^p(\partial\Omega)} \\
&\le C\bigl(\|c_1\|_p^{p-1}\|c_1\|_{1,p}+\|c_2\|_p^{p-1}\|c_2\|_{1,p}\bigr)^{\frac{p-1}{p}}
\bigl(\|c_1-c_2\|_p^{p-1}\|c_1-c_2\|_{1,p}\bigr)^{\frac{1}{p}}.
\end{aligned}
\ee
Setting $\mathcal{M}_T=M_T(c_1-c_2)$, $\mathcal{N}_T=N_T(c_1-c_2)$, we deduce that
\be\label{A1A2}
|A_1(s)-A_2(s)|
\le Cs^{-1/2}\bigl(\bar N_T\bar M_T^{p-1}\bigr)^{\frac{p-1}{p}}
\mathcal{N}_T^{\frac{1}{p}}\mathcal{M}_T^{\frac{p-1}{p}},
\quad
|A_j(s)|\le  Cs^{-1/2}\bar N_T \bar M_T ^{p-1}
\ee
(replacing $c_1$ or $c_2$ with $0$ in \eqref{A1A2a} for the latter).
Consequently, by \eqref{smooth2} and~\eqref{NewRepres}, 
$$\begin{aligned}
\|c_j(t)\|_p
&\le \|S(t)c_0\|_p+C\int_0^t (t-s)^{-1/2}|A_j(s)|\|c_j(s)\|_p \diff \! s
\le \|c_0\|_p+CI\bar N_T\bar M_T^p
\end{aligned}$$
where $I=\int_0^t (t-s)^{-1/2}s^{-1/2} \diff \! s=C$.
Therefore, $\bar M_T \le \|c_0\|_p+C\bar N_T\bar M_T^p$ 
and, by \eqref{Hypuniq1}
we deduce
\be\label{estimAiuniq2a}
\bar M_T\le 2\|c_0\|_p.
\ee

On the other hand, by \eqref{smooth2},
for $\ell\in[0,\frac{1}{p}]$ and $k\in[\ell,1]$, we have
\be\label{KdeltaDiff}
t^{\frac{k}{2}}\|[c_1-c_2](t)\|_{k,p}
\le Ct^{\frac{k}{2}}\int_0^t\ (t-s)^{-\frac{1+k-\ell}{2}}\phi_\ell(s)\diff \! s,
\quad \phi_\ell(s):=\bigl\|(A_2c_2-A_1c_1)(s)\bigr\|_{\ell,p}.
\ee
Writing
$$
\begin{aligned}
\phi_\ell(s) 
&=\bigl\|A_1(s)(c_1(s)-c_2(s))+(A_1(s)-A_2(s))c_2(s)\bigr \|_{\ell,p} \\
&\le |A_1(s)|\,\|c_1(s)-c_2(s)\|_{\ell,p}+|A_1(s)-A_2(s)| \,\|c_2(s)\|_{\ell,p}
\end{aligned}
$$
and using \eqref{A1A2} and the interpolation inequality
$$\|c\|_{\ell,p}\le \|c\|_{1,p}^\ell\|c\|_p^{1-\ell}
\le N_T^\ell(c)M_T^{1-\ell}(c) t^{-\ell/2}$$
 for $c=c_i$ and $c=c_1-c_2$,
it follows that
\be\label{KdeltaDiff2}
\phi_\ell(s) 
\le Cs^{-\frac{1+\ell}{2}}
\Bigl\{\bar N_T\bar M_T^{p-1}\mathcal{N}_T^\ell\mathcal{M}_T^{1-\ell}+\bigl(\bar N_T\bar M_T^{p-1}\bigr)^{\frac{p-1}{p}}
\,\bar N_T^\ell\bar M_T^{1-\ell}\mathcal{N}_T^{\frac{1}{p}}\mathcal{M}_T^{\frac{p-1}{p}}\Bigr\}.
\ee

Now, applying \eqref{KdeltaDiff}, \eqref{KdeltaDiff2} with $k=\ell=0$, we obtain
$$\|[c_1-c_2](t)\|_p
\le CI_1\Bigl\{\bar N_T\bar M_T^{p-1}\mathcal{M}_T+\bigl(\bar N_T\bar M_T^{p-1}\bigr)^{\frac{p-1}{p}}
\,\bar M_T\mathcal{N}_T^{\frac{1}{p}}\mathcal{M}_T^{\frac{p-1}{p}}\Bigr\},
$$
where $I_1=\int_0^t\ (t-s)^{-1/2} s^{-1/2} \diff \! s=C$,
hence
$$\mathcal{M}_T\le C
\Bigl\{\bar N_T\bar M_T^{p-1}\mathcal{M}_T+\bigl(\bar N_T\bar M_T^{p-1}\bigr)^{\frac{p-1}{p}}
\,\bar M_T\mathcal{N}_T^{\frac{1}{p}}\mathcal{M}_T^{\frac{p-1}{p}}\Bigr\}.
$$
Using assumption \eqref{Hypuniq1} and Young's inequality, this implies
\be\label{KdeltaDiff3}
\mathcal{M}_T\le 
K\mathcal{N}_T,\quad K=C
\bigl(\bar N_T\bar M_T^{p-1}\bigr)^{p-1}\bar M_T^p.
\ee
Then setting $\ell=1/2p$ and applying \eqref{KdeltaDiff}, \eqref{KdeltaDiff2} with $k=1$, we obtain
$$\begin{aligned}
t^{\frac12}\|[c_1-c_2](t)\|_{1,p}
&\le CI_2
\Bigl\{\bar N_T\bar M_T^{p-1}\mathcal{N}_T^\ell\mathcal{M}_T^{1-\ell}+\bigl(\bar N_T\bar M_T^{p-1}\bigr)^{\frac{p-1}{p}}\,\bar N_T^\ell\bar M_T^{1-\ell}\mathcal{N}_T^{\frac{1}{p}}\mathcal{M}_T^{\frac{p-1}{p}}\Bigr\},
\end{aligned}$$
where $I_2=t^{\frac12}\int_0^t (t-s)^{-1+\frac{1}{4p}} s^{-\frac12-\frac{1}{4p}}\diff \! s=C$.
Combining with \eqref{KdeltaDiff3}, we deduce
$$\begin{aligned}
\mathcal{N}_T
&\le C\Bigl\{\bar N_T\bar M_T^{p-1}K^{1-\ell}
+\bigl(\bar N_T\bar M_T^{p-1}\bigr)^{\frac{p-1}{p}}
\,\bar N_T^\ell\bar M_T^{1-\ell}K^{\frac{p-1}{p}}\Bigr\}\mathcal{N}_T.
\end{aligned}$$
Using assumption \eqref{Hypuniq1} again, it follows that $\mathcal{N}_T\le \mathcal{N}_T/2$, 
hence $c_1\equiv c_2$.
\end{proof}

\medskip

\begin{proof}[Proof of Proposition~\ref{PropUniqLp}]
We make use of an argument from \cite{Br,BC96}. 
We only consider the case $m=p\ge 1$, the case $m<p=1$ being easier.
First note that $\mathcal{K}:=\tilde c([0,T])\cup \hat c([0,T])$ is a compact subset of~$L^p(\Omega)$.
Set $M=\sup_{t\in[0,T]} \bigl(\|\tilde c(t)\|_p+\|\hat c(t)\|_p\bigr)$, and let $\delta(t)=\delta_{1,p,\mathcal{K}}(t)$ 
be defined by \eqref{compactsmoothing1}. 

Fix any $\tau\in(0,T)$. By the existence part of Theorem~\ref{thm1} and property \eqref{compactexist}, there exist
 $T_0>0$ independent of $\tau$ and a solution 
$c_\tau$ of \eqref{pbmP} on $(0,T_0)$ with initial data 
$\tilde c(\tau)$
and such that, for all $t\in(0,T_0)$,
\be\label{Ntctau1}
N_t(c_\tau)=\sup_{s\in(0,t)}s^{1/2}\|c_\tau(s)\|_{1,p}\le C\delta(t) \to 0,\quad t\to 0.
\ee
Set $\tilde c_\tau(t)=\tilde c(\tau+t)$ and
$$N_t(\tilde c_\tau)=\sup_{s\in(0,t)}s^{1/2}\|\tilde c_\tau(s)\|_{1,p}.$$
Since $\tilde c\in E_T$, 
we have $\lim_{t\to 0}N_t(\tilde c_\tau)=0$ (where the convergence need not be uniform in $\tau$).
It thus follows from Lemma \ref{lem:stat_sol} and \eqref{Ntctau1} 
that $c_\tau(t)=\tilde c_\tau(t)$ 
for $t>0$ small (depending on~$\tau$). 
We claim that 
\be\label{Ntctau3}
c_\tau(t)=\tilde c_\tau(t)\quad\hbox{for all $t\in(0,T_\tau)$, \ \ with $T_\tau:= \min(T_0,T-\tau)$.}
\ee
Indeed, otherwise, letting $t_0:=\sup \bigl\{s\in(0,T_\tau);\,c_\tau=\tilde c_\tau\hbox{ on $[0,s]$}\bigr\}$,
we would have $0<t_0<T_\tau$ and, since 
$c_\tau, \tilde c_\tau\in E_{T_{\tau}}$, 
the above argument would yield $c_\tau(t)=\tilde c_\tau(t)$ for $t-t_0>0$ small: a contradiction.

Now, from \eqref{Ntctau1} and \eqref{Ntctau3}, we have in particular, 
$$s^{1/2}\|\tilde c(\tau+s)\|_{1,p}\le C\delta(t),\quad 0<s<t<\min(T_0,T-\tau).$$
For any $t\in(0,\min(T_0,T))$, letting $\tau\to 0$, we obtain 
$$s^{1/2}\|\tilde c(s)\|_{1,p}\le C\delta(t),\quad 0<s<t<\min(T_0,T)$$
and this obviously remains true for $\hat c$.
Applying Lemma~\ref{lem:stat_sol} again, we deduce that $\tilde c=\hat c$ for $t>0$ small.
Then arguing as in the proof of \eqref{Ntctau3}, we conclude that $\tilde c=\hat c$ on $[0,T]$.
\end{proof}

\subsection{Proof of Theorem~\ref{thm1} and Proposition~\ref{thm2prop}(i): positivity and continuation}
  
In view of the positivity proof, we start with a maximum principle for smooth solutions.  

\begin{lemm} \label{lempos1}
Let $f$ be 
continuous on $[0,\infty)$, $c_0\in C(\overline\Omega)$, $T>0$, 
and let $c\in E_T 
\cap C(\overline Q_T)$ be a classical solution of \eqref{pbmP} on $[0,T]$. 

\begin{itemize}
\item[(i)] If $c_0\ge 0$, then $c\ge 0$.
	\smallskip
	
\item[(ii)]  If in addition $c_0\not\equiv 0$, then 
$c>0$ in $(\Omega\cup\Gamma)\times(0,T)$.
\end{itemize}
\end{lemm}

\begin{proof}
(i) We use a Stampacchia type argument similarly as in, e.g.,  \cite[Proposition 52.8 and Remark 52.9]{QSb}.
Recalling the notation \eqref{NewRepres2}, we
note that $A(t)$ is bounded on $[0,T]$ owing to our assumptions.
Set $c_-=\max(-c,0)$. Multiplying the PDE in \eqref{pbmP} by $-c_-$, integrating by parts
and using the boundary conditions, we obtain, for a.e.~$t\in(0,T)$,
$$\begin{aligned}	
\frac12\frac{\diff}{\dt}\int_\Omega (c_-)^2(t)\dx
&=-\int_\Omega (\Delta c)c_-\dx +A(t)\cdot\int_\Omega c_-\nabla c\dx\\
&=-\int_\Omega |\nabla c_-|^2\dx+\int_\Gamma (A(t)\cdot\nu) (c_-)^2 \dsigma -A(t)\cdot\int_\Omega c_-\nabla c_-\dx\\
&\le -\int_\Omega |\nabla c_-|^2\dx +\frac12\int_\Omega |\nabla c_-|^2\dx
+C_1\Bigl(\int_\Gamma (c_-)^2 \dsigma+\int_\Omega c_-^2\dx\Bigr).
\end{aligned}	$$
We note that this calculation is allowed also when $\Omega$ is a cylinder \eqref{defCyl}
owing to \eqref{regulH1L2}.
Applying the trace inequality $\int_\Gamma v^2 d\sigma\le \eps\int_\Omega |\nabla v|^2\dx+C_2(\eps)\int_\Omega v^2\dx$
with $v=c_-$ and $\eps=(2C_1)^{-1}$, it follows that
$$	\frac12\frac{\diff}{\dt}\int_\Omega (c_-)^2(t)\dx
\le C_1(1+C_2(\eps))\int_\Omega c_-^2\dx.$$
Since $c_-(\cdot,0)=0$ by assumption, the assertion follows by integration.

(ii) Since $c\ge 0$, $c\not\equiv 0$ solves an equation of the form 
$c_t-\Delta c=b(t)\cdot\nabla c$ with $b\in C([0,T];\R^n)$,
we have $c>0$ in $\Omega\times(0,T]$ by the strong maximum principle.
On the other hand, we cannot have $c(x_0,t_0)=0$ with $t_0>0$ and $x_0\in \Gamma$ (a regular boundary point), 
because the boundary conditions then imply $\nu\cdot\nabla c(x_0,t_0)=0$,
 contradicting the Hopf lemma. 
 \end{proof}

\begin{proof}[Completion of proof of the positivity part of Theorem~\ref{thm1} and Proposition~\ref{thm2prop}(i)]
By Lemma \ref{lem:c0i}, there exist $T_0\in(0,T^*)$,
$c_{0,i}\in C^\infty_0(\Omega)$, $c_{0,i}\ge 0$,
 and solutions $c_i$ of \eqref{pbmP} with initial data $c_{0,i}$, 
 such that $c_i\in E_{T_0}\cap C(\overline Q_{T_0})$ and $c_i$ converges to $c$
in $C^{2,1}_{loc}((\Omega\cup\Gamma)\times(0,T_0))$.
It follows from Lemma~\ref{lempos1}(i) 
that, for each $i\in\N$, 
$c_i\ge 0$ in $\overline\Omega\times[0,T_0]$, 
hence $c\ge 0$ on $[0,T_0]$. Since $c$ is a classical solution on $[T_0,T^*)$, 
Lemma~\ref{lempos1}(i) guarantees that $c\ge 0$ on $[0,T^*)$.

Finally assume for contradiction that $u(x_0,t_0)=0$ for some $(x_0,t_0)\in(\Omega\cup\Gamma)\times(0,T^*)$.
Since $c$ is a classical solution on $(0,T^*)$, 
Lemma~\ref{lempos1}(ii) implies $c=0$ on $(0,t_0]$.
Since $c(t)\to c_0\not\equiv 0$ in $L^p$ as $t\to 0$, this is a contradiction.
\end{proof}

\begin{proof}[Completion of proof of the continuation part of Theorem~\ref{thm1} and Proposition~\ref{thm2prop}(i)]
For $p$ $>m\ge 1$ or $p=1>m$, 
by Lemma \ref{lem:c0i}, the maximal classical solution $c$ of \eqref{pbmP} exists on a time interval of length
$T_0:=[1+C_\eps\|c_0\|_p^m]^{-\frac{2p}{p-m}-\eps}$,
with any $\eps>0$. By a time shift, it follows that
$$T^*-t\ge[1+C_\eps\|c(t)\|_p^m]^{-\frac{2p}{p-m}-\eps},\quad 0<t<T^*.$$
If $T^*<\infty$, we thus have 
$$\|c(t)\|_p\ge C_\eps(T^*-t)^{-\frac{p-m}{2pm}+\eps},\quad t\to T^*,$$
which in particular implies the assertion.
\end{proof}

\begin{remark} \label{remsignf}
All the conclusions of Theorem~\ref{thm1} remain valid when $f(c)=-|c|^{m-1}c$ for $m\ge 1$,
as well as \eqref{regulH1L2} with $T=T^*$ if $\Omega$ is a cylinder~\eqref{defCyl}.
Moreover, property \eqref{estimciHolderC} with $p>m$ and $T_0$ given by \eqref{defT0cases} remains true.
Indeed the proofs in this section made no use of the sign of the nonlinear term.
\end{remark}

\section{Proof of Theorems~\ref{thm2}, \ref{thm2b} and \ref{thm1aneg}}\label{sec:proof:theorems}
\begin{proof}[Proof of Theorem~\ref{thm2}(i)] 
By Proposition~\ref{thm2prop}(i), problem \eqref{pbmP} admits a unique, maximal classical solution
	$c\in E_{T^*} 
	\cap C([0,T^*);L^1(\Omega))$.
	 Moreover, since $c_0\ge 0$, we have $c\ge 0$, 
	 hence $\|c(t)\|_1\equiv \|c_0\|_1$ by \eqref{Conserv}.
	 The last part of Proposition~\ref{thm2prop}(i) then guarantees that $T^*=\infty$
	 and property \eqref{estimciHolderC} ensures \eqref{boundeps}.
\end{proof}

We next compute the evolution of the entropy and of the $L^p$ norms,
for classical solutions of problems of the form \eqref{pbmPA}.

\begin{lemm} \label{lemFunctionals}
Let $T>0$ and $b\in C(0,T;\R^n)$.
Let $v\in E_T$, 
with also $v\in L^2_{loc}((0,T);H^2(\Omega))\cap H^1_{loc}((0,T);L^2(\Omega))$
if $\Omega$ is a cylinder \eqref{defCyl}.
Assume $v$ is a classical solution~of 
\begin{subequations}\label{pbmPA}
\begin{align}
	\hfill\partial_t v&=\nabla\cdot\bigl(\nabla v+b(t)v\bigr),
	 &&x\in\Omega,\ 0<t<T,\\
	\hfill 0&=\bigl(\nabla v+b(t)v\bigr)\cdot \nu,&&x\in\partial\Omega,\ 0<t<T.
\end{align}
\end{subequations}
\begin{itemize}
\item[(i)] Assume 
\be\label{Hypvpos}
v>0\ \hbox{ in $\Omega\times(0,T)$}
\ee
and set $\phi(t)=\int_\Omega v\log v\dx$. Then $\phi\in W^{1,1}_{loc}(0,T)$, 
$\int_\Omega  v^{-1}|\nabla v|^2\dx\in L^1_{loc}(0,T)$ and
$$\phi'(t)=-\int_\Omega  v^{-1}|\nabla v|^2\dx -b(t) \cdot \int_\Omega \nabla v\dx, \quad a.e. \ t\in (0,T).$$

\item[(ii)] Assume that either $p>1$ and \eqref{Hypvpos} holds, or that $p$ is an even integer. 
Then $\phi(t):=\int_\Omega v^p\dx\in H^1_{loc}(0,T)$,
$\int_\Omega  v^{p-2}|\nabla v|^2\dx\in L^\infty_{loc}(0,T)$ and
$$\phi'(t)=-p(p-1)\left\{\int_\Omega  v^{p-2}|\nabla v|^2\dx+b(t) \cdot \int_\Omega v^{p-1}\nabla v\dx\right\},\quad 
a.e.~t\in (0,T).$$

\end{itemize}
\end{lemm}

\begin{proof} (i) Fix $\eps>0$. 
For all $t\in (0,T)$, multiplying the PDE in \eqref{pbmPA} by $\log(v+\eps)$ and using the boundary conditions, we get
$$\begin{aligned}
\frac{d}{dt}\int_\Omega v\log (v+\eps)\dx
&=\int_\Omega v_t\Bigl(\log (v+\eps)+\frac{v}{v+\eps}\Bigr)\dx\\
&=\int_\Omega \Bigl(\log (v+\eps)+\frac{v}{v+\eps}\Bigr)\ \nabla\cdot\Bigl(\nabla v+b(t)v\Bigr)\dx\\
&=-\int_\Omega \frac{\nabla v}{v+\eps}\Bigl(1+\frac{\eps}{v+\eps}\Bigr)\cdot\Bigl(\nabla v+b(t)v\Bigr)\dx.
\end{aligned}$$
Integrating in time, we get, for all $0<t_1<t_2<T$,
$$
\begin{aligned}
\int_{t_1}^{t_2}
&\int_\Omega  \frac{|\nabla v|^2}{v+\eps} \Bigl(1+\frac{\eps}{v+\eps}\Bigr)\dx\dt \\
&=-\Bigl[ \int_\Omega v\log (v+\eps)\dx\Bigr]_{t_1}^{t_2}
- \int_{t_1}^{t_2} b(t) \cdot \int_\Omega \frac{v}{v+\eps}\Bigl(1+\frac{\eps}{v+\eps}\Bigr)\nabla v\dx\dt.
\end{aligned}
$$
We may then pass to the limit $\eps\to 0$ via monotone (resp., dominated) convergence in the LHS (resp., RHS).
It follows that
$$\int_{t_1}^{t_2}\int_\Omega  \frac{|\nabla v|^2}{v}\dx\dt =-\Bigl[ \int_\Omega v\log v\dx\Bigr]_{t_1}^{t_2}
- \int_{t_1}^{t_2} b(t) \cdot \int_\Omega \nabla v\dx\dt, \quad 0<t_1<t_2<T.$$
This in particular implies the finiteness of the LHS and readily yields the assertion.

(ii) We consider the case \eqref{Hypvpos} with $p>1$, 
the other being similar (and easier, $\eps=0$ being sufficient).
Fix $\eps>0$. For a.e.~$t\in (0,T)$, multiplying the PDE in \eqref{pbmPA} by $(v+\eps)^{p-1}$ and using the boundary conditions, we get
$$\begin{aligned}
\int_\Omega v_t(v+\eps)^{p-1}\dx
&=\int_\Omega (v+\eps)^{p-1}\ \nabla\cdot\Bigl(\nabla v+b(t)v\Bigr)\dx\\
&=-(p-1)\int_\Omega (v+\eps)^{p-2}\nabla v\cdot\Bigl(\nabla v+b(t)v\Bigr)\dx,
\end{aligned}$$
hence
$$(p-1)\int_\Omega  (v+\eps)^{p-2}|\nabla v|^2\dx 
=-\int_\Omega v_t(v+\eps)^{p-1}\dx-(p-1)b(t) \cdot \int_\Omega v(v+\eps)^{p-2}\nabla v\dx.$$
We may then pass to the limit $\eps\to 0$ via monotone (resp., dominated) convergence in the LHS (resp., RHS),
so that
$$(p-1)\int_\Omega  v^{p-2}|\nabla v|^2\dx 
=-\int_\Omega v_tv^{p-1}\dx-(p-1)b(t) \cdot \int_\Omega v^{p-1}\nabla v\dx.$$
This in particular guarantees that the LHS is finite for a.e.~$t\in(0,T)$ and locally bounded.
Since $\phi\in H^1_{loc}(0,T)$, owing to the regularity of $v$, and $\phi'(t)=p\int_\Omega v_tv^{p-1}\dx$,
the assertion follows.
\end{proof}

In view of the proof of Theorems~\ref{thm2}(ii) and \ref{thm2b}, 
by means of Lemma~\ref{lemFunctionals}, we establish the following a priori estimate.

\begin{prop}\label{PropAprioriLp}
	Assume that $f(c)= c^m$, with $m\ge 1$ and let $p>1$ satisfy $m\le p<2m$. 
	Let $c_0\in L^p(\Omega)$, $c_0\ge 0$, $c_0\not\equiv 0$.
\begin{itemize}
\item[(i)] We have
$$		\frac{\di }{\dt} \|c(t)\|_p^p
			\le \frac{4(p-1)}{p}
  \Bigl\{K_0\|c\|_p^m-1 \Bigr\}\int_\Omega |\nabla c^{p/2}|^2 \dx, \quad a.e.~t\in(0,T^*),$$
  where $K_0=\frac{m}{p-1} |\Omega|^{(p-m)/p} $.
\smallskip

\item[(ii)] In particular, if $\|c_0\|_p<K_0^{-1/m}$, then 
  $\|c(t)\|_p$ 
  is nonincreasing for $t\in (0,T^*)$.
  \end{itemize}
\end{prop}

\begin{proof} 
(i) We know that $c>0$ in $\Omega\times(0,T)$, so that Lemma~\ref{lemFunctionals}(ii) gives
$$\frac{\di }{\dt} \|c(t)\|_p^p
		=  -\frac{4(p-1)}{p} \int_\Omega |\nabla c^{p/2}|^2 \dx    +(p-1)A(t) \cdot  \int_\Omega   \nabla c^p \dx.
$$
On the other hand, by the Cauchy-Schwarz and the H\"older inequalities, we have
$$\left| \int_\Omega   \nabla c^p \dx \right| \le 2 \left (\int_\Omega c^p\dx  \right)^{1/2} \left (\int_\Omega|\nabla c^{p/2}|^2\dx \right)^{1/2}$$
	and, using the divergence theorem and recalling 
	the notation \eqref{NewRepres2} and
	$m\le p<2m$,
$$\begin{aligned}
	|A(t)|
	&=\left |\int_{\partial\Omega}c^m \nu \dsigma\right|=\left |\int_\Omega\nabla c^m\dx \right| \\
	&\le \frac{2m}{p} \left(\int_\Omega c^{2m-p}\dx \right) ^{1/2} \left (\int_\Omega|\nabla c^{p/2}|^2\dx  \right)^{1/2} \\
	&\le \frac{2m}{p} |\Omega|^{(p-m)/p}\left(\int_\Omega c^p\dx \right)^{(2m-p)/2p} \left (\int_\Omega|\nabla c^{p/2}|^2\dx  \right)^{1/2}.
\end{aligned}$$
Consequently, 
$$\begin{aligned}
		\frac{\di }{\dt} \|c(t)\|_{L^p}^p
			&\le  -\frac{4(p-1)}{p} \int_\Omega |\nabla c^{p/2}|^2 \dx     
			+  \frac{4m}{p} |\Omega|^{(p-m)/p}\left(\int_\Omega c^p\dx \right)^{m/p} \int_\Omega|\nabla c^{p/2}|^2\dx\\
			&\le\frac{4(p-1)}{p}
  \Bigl\{-1+  \frac{m}{p-1} |\Omega|^{(p-m)/p} \|c\|_p^m \Bigr\}\int_\Omega |\nabla c^{p/2}|^2 \dx,
\end{aligned}$$
which proves the proposition.

(ii) This is a direct consequence of (i). 
			\end{proof}

We shall also use the following exponential convergence lemma.

\begin{lemm}\label{cvexp}
Let $f:\R\to\R$ be locally Lipschitz. There exist $\eta_2=\eta_2(f,\Omega)>0$ and $\lambda=\lambda(\Omega)>0$ with the following property.
If
$c\in E_\infty 
\cap C([0,\infty);L^1(\Omega))$ is a global classical solution of \eqref{pbmP} such that $|\bar c_0|\le\eta_2$ and
\be\label{hyplemexp1}
\lim_{t\to\infty} \|c(t)-\bar c_0\|_\infty=0,
\ee
 where $\bar c_0=\frac{1}{|\Omega|}\int_\Omega c_0(x)\dx$,
then the exponential convergence property \eqref{stabilc} holds for some $C>0$.
\end{lemm}

\begin{proof}
In this proof, $K, K_1$ will denote generic positive constants depending only on $f,\Omega$
and $C$ a generic positive constant possibly depending on the solution $c$.
Assume $|\bar c_0|\le 1$. By assumption \eqref{hyplemexp1},
there exists $T_0>0$ such that 
\be\label{decayat0}
\sup_{t\ge T_0} \|c(t)\|_\infty\le 2.
\ee
Recalling the notation \eqref{NewRepres2} and using also the continuity of~$f$, we have 
\be\label{decayat}
\lim_{t\to\infty}A(t)=\int_{\partial \Omega}f(\bar c_0) \nu \dsigma
=f(\bar c_0)\int_{\partial \Omega} \nu \dsigma=0.
\ee
Set $w:=c-\bar c_0$. By \eqref{decayat0} and the local Lipschitz continuity of $f$, we get, for all $t\ge T_0$,
$$|A(t)|=\Bigl|\ds\int_{\partial\Omega}f(c)\nu \dsigma\Bigr|
=\Bigl|\ds\int_{\partial\Omega}\bigl(f(c)-f(\bar c_0)\bigr)\nu \dsigma\Bigr|
\le\ds\int_{\partial\Omega}\bigl|f(c)-f(\bar c_0)\bigr| \dsigma
\le K\ds\int_{\partial\Omega} |w| \dsigma.$$
Since $\int_\Omega w(t)\dx=0$, it follows from the trace and the Poincar\'e-Wirtinger inequalities that
\be\label{decayat2}
|A(t)|\le K\|w\|_{H^1}\le K\Bigl(\int_\Omega |\nabla w|^2 \dx\Bigr)^{1/2},\quad t\ge T_0.
\ee

Next, the function $w$ satisfies
\be\label{pbmw}
\begin{aligned}
	\hfill\partial_t w&=\nabla\cdot\Bigl(\nabla w-A(t)w\Bigr),	 &&x\in\Omega,\ t>0,\\
	\hfill \partial_\nu w&=(A(t)\cdot \nu)c,&&x\in\partial\Omega,\ t>0. \\
\end{aligned}
\ee
Multiplying the PDE in \eqref{pbmw} by $w$ (recalling \eqref{regulH1L2}
if $\Omega$ is a cylinder~\eqref{defCyl}),
integrating by parts and using the boundary conditions, we obtain, for a.e. $t>0$,
$$
\begin{aligned}
		\frac{1}{2}\frac{\di }{\dt} \|w(t)\|_2^2
		&=\int_\Omega ww_t\dx = \int_\Omega w\nabla\cdot\Bigl(\nabla w-A(t)w\Bigr)\dx \\
		&=  -\int_\Omega |\nabla w|^2 \dx    +A(t) \cdot  \int_\Omega  w\nabla w \dx 
		+\int_{\partial\Omega} w\bigl(\partial_\nu w-(A(t)\cdot\nu)w\bigr) \dsigma \\
				&=  -\int_\Omega |\nabla w|^2 \dx  +A(t) \cdot  \int_\Omega w\nabla w \dx 
		+\bar c_0A(t)\cdot\int_{\partial\Omega} w\nu \dsigma \\
				&=  -\int_\Omega |\nabla w|^2 \dx  +A(t) \cdot  \int_\Omega \bigl(w+\bar c_0\bigr)\nabla w \dx.
		\end{aligned}
$$
By the Cauchy-Schwarz and the Poincar\'e-Wirtinger inequalities, we have
$$ \int_\Omega |w||\nabla w|\dx  \le  \Bigl(\int_\Omega |\nabla w|^2\dx \Bigr)^{1/2}\Bigl(\int_\Omega w^2\dx \Bigr)^{1/2}\le K\int_\Omega |\nabla w|^2\dx.$$
This along with \eqref{decayat2} yields
$$
		\frac{1}{2}\frac{\di }{\dt} \|w(t)\|_2^2\le \Bigl\{-1+K\bigl(|A(t)|+|\bar c_0|\bigr)\Bigr\} \int_\Omega |\nabla w|^2\dx ,
		\quad a.e.~t\ge T_0.
$$
Take $\eta_2=\min(1,(4K)^{-1})$ and $T_1>T_0$ sufficiently large so that $\sup_{t\ge T_1}|A(t)|\le \eta_2$.
Using again the Poincar\'e-Wirtinger inequality, we get
$$
	\frac{\di }{\dt} \|w(t)\|_2^2\le - \int_\Omega |\nabla w|^2\dx  \le - K_1  \|w(t)\|_2^2, \quad a.e.~t\ge T_1.
$$
By integration, we get
$$\|w(t)\|_2^2\le Ce^{-K_1t},\quad t\ge T_1$$
and, by \eqref{decayat0} and interpolation, it follows that for all $q\in[2,\infty)$,
$$\|w(t)\|_q\le Ce^{-\frac{K_1}{q}t},\quad t\ge T_1.$$
On the other hand, by \eqref{decayat0}  and parabolic regularity
(or Proposition~\ref{locexist-cyl}(vi) if $\Omega$ is a cylinder~\eqref{defCyl}), we have
$\sup_{t\ge T_0} \|c(t)\|_{C^1}<\infty$. Choosing some $q>2n$, a further interpolation and Sobolev inequality yield
$$\|w(t)\|_\infty\le K\|w(t)\|_{1/2,q}\le  K\|w(t)\|_{1,q}^{1/2}\|w(t)\|_q^{1/2}\le Ce^{-\frac{K_1}{2q}t},\quad t\ge T_1,$$
which proves the lemma.
\end{proof}

\begin{proof}[Proof of Theorem~\ref{thm2b}]
 Let $K_0, \eta_0$ be given by Propositions~\ref{PropAprioriLp} and \ref{thm2prop}(ii), respectively.
 Assume that $\|c_0\|_m\le \eta_0$ if $m=1$ or that $\|c_0\|_m\le\min((2K_0)^{-1/m},\eta_0)$ if $m>1$.
By Proposition~\ref{PropAprioriLp}(ii) (or by mass conservation if $m=1$), 
\be\label{monotoneLm}
t \mapsto \|c(t)\|_m \ \ \hbox{is nonincreasing and  $\le \eta_0$ for $t\in (0,T^*)$.}
\ee
Setting $q=nm/(n-1)>m$ (or any finite $q>m$ if $n=1$), it follows from \eqref{regulc2b} and a time shift
that $\limsup_{t\to T^*}\|u(t)\|_q<\infty$.
Theorem~\ref{thm1}(iii) then guarantees that $T^*=\infty$.
Also, by \eqref{estimciHolderC} or \eqref{estimciHolderCcyl}, we have
\be\label{regulC1}
\|c\|_{C^{1+\alpha,\alpha/2}(\overline\Omega\times[1,\infty))}<\infty.
\ee
To show the uniform convergence to $\bar c_0$, the idea is to use the Liapunov functional given by  
Proposition~\ref{PropAprioriLp} for small initial data, along with the conservation of the $L^1$ norm.
To this end, let us first consider the case $m>1$. 
In view of \eqref{monotoneLm}, we may set $\ell:=\lim_{t\to\infty} \|c(t)\|_m\in [0,\infty)$.
Putting $\tilde c_j=c(x,t+j)$ for $j\in\N$, we see that each $\tilde c_j(x,t)$ is a solution of \eqref{pbmP}, with initial data $c(\cdot,j)$.
By \eqref{regulC1}, 
the sequence $(\tilde c_j)$ is precompact in $C_{loc}(\overline\Omega\times[0,\infty))$.
For any $\tau>0$, by Proposition~\ref{PropAprioriLp}(i) with $p=m$, we have
\be\label{VarLiapFonct}
\begin{aligned}
\frac{2(m-1)}{m}\int_0^\tau \int_\Omega |\nabla \tilde c_j^{m/2}|^2 \dx
&= \frac{2(m-1)}{m}\int_j^{j+\tau} \int_\Omega |\nabla  c^{m/2}|^2 \dx \\
&\le \|c(j)\|_m^m-\|c(j+\tau)\|_m^m \to \ell-\ell =0,\quad j\to\infty.
\end{aligned}
\ee
Let $\tilde c$ be a cluster point of $(\tilde c_j)$, i.e. $\tilde c_{j_k}\to \tilde c$ in $C_{loc}(\overline\Omega\times[0,\infty))$
for some subsequence $j_k$. It follows from \eqref{VarLiapFonct} that $\tilde c(x,t)=\tilde c(t)$.
Since $\|c(t)\|_1=\|c_0\|_1$ for all $t\ge 1$, we deduce that $\tilde c\equiv \frac{1}{|\Omega|}\|c_0\|_1$. 
We have thus proved that $\lim_{t\to\infty} \|c(t)-\bar c_0\|_\infty=0$ in the case $m>1$.
In the case $m=1$, we note from \eqref{regulc2b} that
$\|c(1/2)\|_q\le C(\Omega)\|c_0\|_1$, with $q=n/(n-1)>1$ (or any finite $q>1$ if $n=1$).
Assuming $\|c_0\|_1$ possibly smaller, we may thus apply Proposition~\ref{PropAprioriLp} with some $p\in(1,2)$
and the above argument works exactly as before.

Finally, to check the exponential decay assertion, it suffices to apply Lemma~\ref{cvexp}, observing that
$|\bar c_0|<\eta_2$ for $\|c_0\|_m$ small.
\end{proof}

\begin{proof}[Proof of Theorem~\ref{thm2}(ii)]
It is completely similar to the proof of Theorem~\ref{thm2b} in the case~$m=1$.
\end{proof}

For the proof of Theorem~\ref{thm1aneg} we need the following lemma 
(which will be also used in the proof of Theorem~\ref{thmBUrate}).

\begin{lemm} \label{lem:L2b}
Let $T\in(0,\infty]$, $b\in C(0,T;\R^n)$ and $v\in E_T$, 
with also $v\in L^2_{loc}((0,T);H^2(\Omega))\cap H^1_{loc}((0,T);L^2(\Omega))$
if $\Omega$ is a cylinder \eqref{defCyl}. Assume that $v$
is a classical solution of 
$$\begin{aligned}
	\hfill v_t&=\nabla\cdot\Bigl(\nabla v+b(t)v\Bigr),
	 &&x\in\Omega,\ 0<t<T,\\
	\hfill 0&=\Bigl(\nabla v+b(t)v\Bigr)\cdot \nu,&&x\in\partial\Omega,\ 0<t<T.
\end{aligned}$$
If, for some $t_0\in(0,T)$,
$b\in L^2(t_0,T;\R^n)$ then, for each $p\in[1,\infty)$, we have $v\in L^\infty(t_0,T;L^p(\Omega))$.
\end{lemm}

\begin{proof}
Assume without loss of generality that $p$ is an even integer. 
	By Lemma~\ref{lemFunctionals}(ii), we have
\be\label{EnergyArg}
		\frac{1}{p(p-1)}\frac{\di}{\dt} \|v(t)\|_p^p
		=  -\int_\Omega v^{p-2}|\nabla v|^2 \dx    - b(t) \cdot  \int_\Omega  v^{p-1}\nabla v \dx,\quad 
		a.e.~t\in (0,T).
		\ee
Since, by the Cauchy-Schwarz inequality, 
$$\begin{aligned}
\Bigl|b(t)\cdot\int_\Omega  v^{p-1}\nabla v \dx\Bigr|
&\le |b(t)|\Bigl(\int_\Omega v^{p-2}|\nabla v|^2 \dx\Bigr)^{1/2}\Bigl(\int_\Omega v^p \dx\Bigr)^{1/2} \\
&\le |b(t)|^2\int_\Omega v^p \dx + \int_\Omega v^{p-2}|\nabla v|^2 \dx,
		\end{aligned}$$
it follows that $\frac{\di }{\dt} \|v(t)\|_p^p\le p(p-1) |b(t)|^2\|v(t)\|_p^p$. By integration in time, we get
$$\|v(t)\|_p^p\le \|v(t_0)\|_p^p \ \exp\Bigl\{p(p-1)\int_{t_0}^T |b(t)|^2 \dt\Bigr\}<\infty,\quad t_0<t<T,$$
which proves the lemma.
\end{proof}

\begin{proof}[Proof of Theorem~\ref{thm1aneg}]
Local existence-uniqueness and positivity are guaranteed by Remark~\ref{remsignf}.

Let us show global existence.
First assume $m>1$ and set $\phi(t)=\|c(t)\|_m^m$. 
	By Lemma~\ref{lemFunctionals}(ii)
	with $p=m$ and 
$b(t)=$ $\int_{\partial\Omega}c^m\nu\dsigma=\int_\Omega\nabla c^m\dx$, we have
$$\begin{aligned}
\phi'(t)
		&=  -m(m-1)\int_\Omega c^{m-2}|\nabla c|^2 \dx    - m(m-1)b(t) \cdot \int_{ \Omega}c^{m-1}\nabla c\dy \\
		&= -\frac{4(m-1)}{m} \int_{\Omega} |\nabla c^{m/2}|^2 \dx  - (m-1)
		\left(\int_{ \Omega}\nabla c^m\dx\right)^2\le 0.
		\end{aligned}$$
Next consider the case $m=1$ and set $\phi(t)=\int_\Omega (c\log c+1)\dx\ge 0$
(since $s\log s+1>0$ for $s>0$).
By Lemma~\ref{lemFunctionals}(i), we have
$$\phi'(t)
		=  -\int_\Omega c^{-1}|\nabla c|^2 \dx    - b(t) \cdot \int_{ \Omega}\nabla c\dy 
		= -4\int_{\Omega} |\nabla c^{1/2}|^2 \dx  - \left(\int_{ \Omega}\nabla c\dx\right)^2\le 0.
$$
		In either case, fixing some $t_0\in(0,\min(1,T^*))$, it follows that 
\be\label{cvintnablacm}
\int_{t_0}^{T^*}\int_{\Omega} |\nabla c^{m/2}|^2 \dx \dt+ \int_{t_0}^{T^*}\left(\int_{ \Omega}\nabla c^m\dx\right)^2\dt
		\le C\phi(t_0)<\infty,
		\ee
		hence in particular $b\in L^2(t_0,T^*)$.
		It follows from Lemma~\ref{lem:L2b} that, for each $p\in[1,\infty)$, $c\in L^\infty(t_0,T^*;L^p(\Omega))$.
		By Remark~\ref{remsignf} we deduce that $T^*=\infty$ and that 
		$c\in E_\infty$.
			
		Finally, since $\phi'\le 0$, there exists $\ell:=\lim_{t\to\infty} \phi(t)\in [0,\infty)$ and using \eqref{cvintnablacm} and arguing similarly as in the proof of Theorem~\ref{thm2b}, we conclude that 
		 $\lim_{t\to\infty} \|c(t)-\bar c_0\|_\infty=0$. 
	\end{proof}

\section{Proof of Theorem~\ref{thm2sharp}}\label{sec:sharp}

The proof relies on entropy and on
a refined version of the smoothing effects in Section~\ref{Sec:proof_th1} (cf.~\eqref{estimMTb0}).
The latter, which is formulated in Proposition~\ref{lem:decomp} hereafter, is based on a decomposition of the solution between a small $L^1$ part and a bounded part. 
To this end we need some notation.
For given $\tau>0$, we denote
$$X_\tau=\bigl\{v\in L^\infty(0,\tau;L^1(\Omega))\cap L^\infty_{loc}((0,\tau]);W^{1,1}(\Omega));\,M_\tau(v)<\infty\bigr\},$$
\be\label{DefdecompM}
M_\tau(v)={\rm ess}\hskip -5pt \sup_{t\in(0,\tau)} t^{1/2}\|v(t)\|_{1,1},\quad
Y_\tau=L^\infty(0,\tau; C(\overline\Omega)),
\ee
and, for all $(v_1,v_2)\in X_\tau\times Y_\tau$,
\be\label{DefdecompH}
H_\tau(v_1,v_2)={\rm ess}\hskip -5pt\sup_{t\in(0,\tau)} \Bigl(\|v_1(t)\|_1 + t^{1/2}\|v_1(t)\|_{1,1}+ \tau^{1/2}\|v_2(t)\|_\infty\Bigr).
\ee
We then introduce the $(X_\tau+Y_\tau)$-norm, defined for $v\in X_\tau+Y_\tau$~by
$$\mathcal{N}_\tau(v)=\inf\bigl\{H_\tau(v_1,v_2);\ (v_1,v_2)\in \mathcal{E}_\tau(v)\bigr\},$$
where 
$$\mathcal{E}_\tau(v)=\bigl\{(v_1,v_2)\in X_\tau\times Y_\tau;\ v(t)=v_1(t)+v_2(t)\ \hbox{for a.e.~$t\in(0,\tau)$}\bigr\}.$$

\begin{prop} \label{lem:decomp}
Let $f(c)=c$.
There exist constants $\eps_0, C_0>0$ (depending only on $\Omega$) with the following property. 
Let $T, R>0$ and let $c\in E_T$ 
be a classical solution of \eqref{pbmP} on $[0,T]$ such that,
for some $(c_{0,1},c_{0,2})\in L^1(\Omega)\times C(\overline\Omega)$, $c_0:=c(\cdot,0)$ can be decomposed as
$$ 
c_0=c_{0,1}+c_{0,2},\ \hbox{ with }\ \|c_{0,1}\|_1\le \eps_0,\ \|c_{0,2}\|_\infty\le R.
$$ 
Then we have
\be\label{estimdecomp1}
\mathcal{N}_\tau(c) \le C_0,\quad 0<\tau<\min(\eps_0R^{-2},T).
\ee
In particular, for $p=n/(n-1)$ (any finite $p$ if $n=1$), we have
\be\label{estimdecomp2}
\|c(t)\|_p \le \bar C_0t^{-1/2},\quad 0<t<\min(\eps_0R^{-2},T),
\ee
with $\bar C_0=\bar C_0(\Omega, a)>0$.
\end{prop}

For the proof of Proposition~\ref{lem:decomp}, we first need some basic properties of $\mathcal{N}_\tau(v)$ as a function of $\tau$
when $v$ is sufficiently regular.

\begin{lemm} \label{lem:decomp0} 
Let $v\in C([0,T);W^{1,1}(\Omega))\cap L^\infty_{loc}([0,T);C(\overline\Omega))$ and set $h(\tau)=\mathcal{N}_\tau(v)$ for $\tau\in(0,T)$.
Then:
\begin{itemize}[topsep=-2pt]\setlength\itemsep{-3pt}
\item[(i)] $h$ is nondecreasing;

\medskip
\item[(ii)] $h$ is continuous;

\medskip
\item[(iii)] 
$\lim_{\tau\to 0} h(\tau)=0$.
\end{itemize}
\end{lemm}

\begin{proof}
(i) Let $0<\tau_1<\tau_2<T$. For all $(v_1,v_2)\in \mathcal{E}_{\tau_2}(v)$, we have 
$(v_1,v_2)_{|(0,\tau_1)}\in \mathcal{E}_{\tau_1}(v)$ and $H_{\tau_1}\bigl((v_1,v_2)_{|(0,\tau_1)}\bigr)\le H_{\tau_2}(v_1,v_2)$ by \eqref{DefdecompH}, 
hence $h(\tau_1)\le h(\tau_2)$.

(ii) Let $\tau\in(0,T)$. We first check the continuity on the left. 
Set $\ell=\lim_{t\to \tau^-} h(t)$ and take any $\eps>0$. Choose an increasing sequence $(\tau_j)_{j\ge 0}$ such that
$\tau_0=0$, $\tau_1\ge (1+\eps)^{-2}\tau$ and $\lim\tau_j=\tau$. 
For each $j\ge 1$, there exists $(v_1^j,v_2^j)\in \mathcal{E}_{\tau_j}(v)$ such that
$H_{\tau_j}(v_1^j,v_2^j)\le h(\tau_j)+\eps\le \ell+\eps$.
For $i\in\{1,2\}$, we may define $\bar v_i$ on $[0,\tau)$ by $\bar v_i:=v_i^j$ on $[\tau_{j-1},\tau_j)$ for each $j\ge 1$.
Using \eqref{DefdecompH} and $\tau^{1/2}\le(1+\eps)\tau_j^{1/2}$,
we then have 
$$H_\tau(\bar v_1,\bar v_2)\le(1+\eps)\,\sup_{j\ge 1}H_{\tau_j}(v_1^j,v_2^j)\le (1+\eps)(\ell+\eps),$$
so that  in particular $(\bar v_1,\bar v_2)\in \mathcal{E}_\tau(v)$.
Consequently, $\ell\le h(\tau)\le H_\tau(\bar v_1,\bar v_2)\le (1+\eps)(\ell+\eps)$, hence
$h(\tau)=\lim_{t\to \tau^-} h(t)$.

We next check the continuity on the right.
Let $\eps>0$. Take $(v_1,v_2)\in \mathcal{E}_\tau(v)$ 
such that 
$H_\tau(v_1,v_2)\le h(\tau)+\eps$
and fix a representative of $(v_1,v_2)$, still denoted by $(v_1,v_2)$. 
By the regularity of $v$, there exists $\eta>0$ such that
$$\|v(t)-v(s)\|_1+(\tau+\eta)^{1/2}\|v(t)-v(s)\|_{1,1} \le\eps\ \hbox{ for all $t,s\in(\tau-\eta,\tau+\eta)$}$$
and we may also assume $\bigl(\ts\frac{\tau+\eta}{\tau-\eta}\bigr)^{1/2}\le 1+\eps$.
Next, by the definition of $\mathcal{E}_\tau(v)$ and $H_\tau(v_1,v_2)$, we may select 
 $\hat\tau\in(\tau-\eta,\tau)$ such that $v(\hat\tau)=v_1(\hat\tau)+v_2(\hat\tau)$ and
$$\|v_1(\hat\tau)\|_1 + \hat\tau^{1/2}\|v_1(\hat\tau)\|_{1,1}+ \tau^{1/2}\|v_2(\hat\tau)\|_\infty\le 
H_\tau(v_1,v_2).$$
We then define $(\hat v_1,\hat v_2)\in \mathcal{E}_{\tau+\eta}$ by
$(\hat v_1,\hat v_2)=(v_1,v_2)$ on $[0,\tau]$ and
$(\hat v_1,\hat v_2)=(v-v_2(\hat\tau),v_2(\hat\tau))$ on $(\tau,\tau+\eta)$.
For $t\in(\tau,\tau+\eta)$, we have 
$$\begin{aligned}
\|\hat v_1(t)\|_1
&=\|v(t)-v_2(\hat\tau)\|_1=\|v(t)-v(\hat\tau)+v_1(\hat\tau)\|_1
\le  \|v_1(\hat\tau)\|_1+\eps,\\
t^{1/2}\|\hat v_1(t)\|_{1,1}
&=t^{1/2}\|v(t)-v_2(\hat\tau)\|_{1,1}=t^{1/2}\|v(t)-v(\hat\tau)+v_1(\hat\tau)\|_{1,1} \\
&\le \bigl(\ts\frac{\tau+\eta}{\tau-\eta}\bigr)^{1/2}\hat\tau^{1/2}\|v_1(\hat\tau)\|_{1,1}+\eps
\le (1+\eps)\hat\tau^{1/2}\|v_1(\hat\tau)\|_{1,1}+\eps,\\
(\tau+\eta)^{1/2}\|\hat v_2(t)\|_\infty
&=(\tau+\eta)^{1/2}\|v_2(\hat\tau)\|_\infty
\le(1+\eps)\tau^{1/2}\|v_2(\hat\tau)\|_\infty,
\end{aligned}$$
hence
$$h(\tau+\eta)\le H_{\tau+\eta}(\hat v_1,\hat v_2)\le(1+\eps) H_\tau(v_1,v_2)+2\eps.$$
Consequently,
$h(\tau+\eta)\le (1+\eps)h(\tau)+2\eps$.
Since $h$ is nondecreasing and $\eps>0$ is arbitrary, we conclude that $\lim_{t\to \tau^+} h(t)=h(\tau)$.

(iii) Set $K=\|v\|_{L^\infty((0,T/2);C(\overline\Omega))}$. Then $h(\tau)\le H_\tau(0,v)\le K\tau^{1/2}\to 0$ as $\tau\to 0$.
\end{proof}

\begin{proof}[Proof of Proposition~\ref{lem:decomp}]
Let $c$ satisfy the assumptions of the proposition, where $\eps_0>0$ will be determined below.
In this proof $C$ will denote a generic constant depending only on~$\Omega$. 
Also, we shall abbreviate 
$$\mathcal{E}_\tau=\mathcal{E}_\tau(c)\equiv \{(c_1,c_2)\in X_\tau\times Y_\tau;\ c(t)=c_1(t)+c_2(t)\ \hbox{for a.e.~$t\in(0,\tau)$}\}.$$

Fix $\tau\in(0,T)$ and pick any $(c_1,c_2)\in \mathcal{E}_\tau$.
Using \eqref{NewRepres}, we may decompose the solution $c$ of \eqref{pbmP} as $c=\tilde c_1+\tilde c_2$, where
$$
\tilde c_i(t)=S(t)c_{0,i}+\int_0^t A(s) \cdot K_\nabla(t-s)c_i(\cdot,s)\diff \! s,\quad 0<t<\tau,\ i\in\{1,2\}.
$$
By \eqref{smooth1b}, \eqref{smooth2}, for $t\in(0,\tau)$, we have
\be\label{estimdecompBasic}
\begin{aligned}
\|\tilde c_1(t)\|_1&\le \|c_{0,1}\|_1+C\int_0^t (t-s)^{-1/2}|A(s)| \|c_1(s)\|_1 \diff \! s, \\
t^{1/2}\|\tilde c_1(t)\|_{1,1}&\le C\|c_{0,1}\|_1+Ct^{1/2}\int_0^t (t-s)^{-3/4}|A(s)| \|c_1(s)\|_{1/2,1}\diff \! s, \\
\|\tilde c_2(t)\|_\infty&\le \|c_{0,2}\|_\infty+C\int_0^t (t-s)^{-1/2}|A(s)| \|c_2(s)\|_\infty \diff \! s
\end{aligned}
\ee
 (and Lemma~\ref{lem:gaussestim} also guarantees that
 $\tilde c_2(t)\in C(\overline\Omega)$ for all $t\in(0,\tau)$).
On the other hand, since
$$A(t)=\int_{\partial\Omega}(c_1+c_2)(t)\nu \dsigma = \int_\Omega \nabla c_1(t) \dx+ \int_{\partial\Omega}c_2(t)\nu \dsigma,
\quad {\rm a.e.}~t\in (0,\tau),$$
(where the last boundary integral is well defined, owing to 
$c_2\in L^\infty(0,\tau;C(\overline\Omega))$), we have
\be\label{estimAdecomp}
|A(s)|\le C\|c_1(s)\|_{1,1}+ C\|c_2(s)\|_\infty,   \quad s\in (0,\tau).
\ee

Next recall the definition of $M_\tau, H_\tau$ in \eqref{DefdecompM}, \eqref{DefdecompH} and also set
$$L_\tau(c_1)=\sup_{t\in(0,\tau)} \|c_1(t)\|_1,\quad P_\tau(c_2)=\sup_{t\in(0,\tau)} \|c_2(t)\|_\infty.$$
Using \eqref{estimdecompBasic}, \eqref{estimAdecomp} and $\|c_1(s)\|_{1/2,1}\le \|c_1(s)\|_1^{1/2}\|c_1(s)\|_{1,1}^{1/2}$, we obtain
$$\begin{aligned}
\|\tilde c_1(t)\|_1
&\le \eps_0+CL_\tau(c_1) \sup_{t\in(0,\tau)}\int_0^t (t-s)^{-1/2}\bigl(s^{-1/2}M_\tau(c_1)+P_\tau(c_2)\bigr)  \diff \! s \\
&\le \eps_0+CL_\tau(c_1)\bigl[M_\tau(c_1)+\tau^{1/2} P_\tau(c_2)\bigr]
\le \eps_0+CH_\tau^2(c_1,c_2),
\end{aligned}$$
$$\begin{aligned}t^{1/2}\|\tilde c_1(t)\|_{1,1} 
&\le C\eps_0+CL_\tau^{1/2}(c_1) M_\tau^{1/2}(c_1) \\
&\qquad\sup_{t\in(0,\tau)}t^{1/2}\int_0^t (t-s)^{-3/4}\bigl(s^{-1/2}M_\tau(c_1)+P_\tau(c_2)\bigr) 
s^{-1/4}\diff \! s, \\
&\le C\eps_0+C\bigl(L_\tau(c_1)+M_\tau(c_1)\bigl)\bigl[M_\tau(c_1)+\tau^{1/2} P_\tau(c_2)\bigr]
\le C\eps_0+CH_\tau^2(c_1,c_2),
\end{aligned}$$
$$\begin{aligned}
\tau^{1/2}\|\tilde c_2(t)\|_\infty 
&\le R\tau^{1/2}+C\tau^{1/2}P_\tau(c_2)\sup_{t\in(0,\tau)}\int_0^t (t-s)^{-1/2}\bigl(s^{-1/2}M_\tau(c_1)+P_\tau(c_2)\bigr)  \diff \! s\\
&\le R\tau^{1/2}+C\tau^{1/2}P_\tau(c_2)\bigl[M_\tau(c_1)+\tau^{1/2} P_\tau(c_2)\bigr]
\le R\tau^{1/2}+CH_\tau^2(c_1,c_2).
\end{aligned}$$
If follows in particular that $(\tilde c_1,\tilde c_2)\in X_\tau\times Y_\tau$, hence $(\tilde c_1,\tilde c_2)\in \mathcal{E}_\tau$ and,
adding up, we get
$$ 
h(\tau):=\mathcal{N}_\tau(c)\le H_\tau(\tilde c_1,\tilde c_2)\le C\eps_0+R\tau^{1/2}+CH_\tau^2(c_1,c_2).$$
Taking infimum over $(c_1,c_2)\in \mathcal{E}_\tau$, we get, for some $C_1=C_1(\Omega,a)>0$,
\be\label{phitau1}
h(\tau)\le C_1\eps_0+C_1h^2(\tau),\quad 0<\tau<\min(T,\tau_0),
\quad\hbox{where $\tau_0=\eps_0 R^{-2}$.}
\ee

Let now $J=\{\tau\in(0,T);\ h(\tau)\le 2C_1\eps_0\}$. 
Owing to the regularity of $c$, the assumptions of Lemma~\ref{lem:decomp0} are satisfied by $v=c$.
By properties (i)(iii) in that lemma, $J$ is a nonempty interval 
with left endpoint $0$. 
Now choose $\eps_0=(10\,C_1^2)^{-1}$ and assume for contradiction that $\bar \tau:=\sup J<\tau_0$.
Then, by \eqref{phitau1} and the continuity of $h$ 
(cf.~Lemma~\ref{lem:decomp0}(ii)), we have
$h(\bar\tau)=2C_1\eps_0\le C_1\eps_0+4C_1^3\eps_0^2$
hence $1\le 4C_1^2\eps_0=1/2$: a contradiction.
Thus \eqref{estimdecomp1} is proved.

Finally, to deduce property \eqref{estimdecomp2} from \eqref{estimdecomp1},
for any $\tau<\min(\eps_0R^{-2},T)$, 
we pick $(c_1,c_2)$ such that $\mathcal{N}_\tau(c) \le 2H_\tau(c_1,c_2)$. 
By Sobolev imbedding and H\"older's inequality, we then have
$$\|c(t)\|_p\le \|c_1(t)\|_p+\|c_2(t)\|_p\le C(\|c_1(t)\|_{1,1}+\|c_2(t)\|_\infty\bigr)\le CC_0t^{-1/2},
\quad 0<t<\tau.$$
\end{proof}

\begin{proof}[Proof of Theorem~\ref{thm2sharp}]
Consider the entropy function $\phi(t):=\int_\Omega (c\log c+1)\dx>0$.
For a.e.~$t\in (0,T^*)$, 
using Lemma~\ref{lemFunctionals}(i), 
the Cauchy-Schwarz inequality and $M\le 1$, we have
\be\label{entropydissip}
\begin{aligned}
\phi'(t)
		&=  -\int_\Omega c^{-1}|\nabla c|^2 \dx  + \left|\int_{ \Omega}\nabla c\dx\right|^2 \\
		&\le -\int_\Omega c^{-1}|\nabla c|^2 \dx+\int_\Omega c \dx  \int_\Omega c^{-1}|\nabla c|^2 \dx
=(M-1) \int_\Omega c^{-1}|\nabla c|^2 \dx \le 0.
		\end{aligned}
		\ee
		Fix some $T_0\in (0,T^*)$. Since $c$ is a classical solution for $t\in(0,T^*)$, we deduce that
$$\int_\Omega (1+c\log c)(t)\le K:=\phi(T_0)<\infty,\quad T_0\le t<T^*.$$

Next, in view of applying Proposition~\ref{lem:decomp}, we decompose $c=c_1+c_2$, with %where
 $c_1=c\chi_{\{c>R\}}$, $c_2=c\chi_{\{c\le R\}}$,
where $R>1$ will be determined below.
For each $t_0\in [T_0,T^*)$, we have $(c_1(t_0),c_2(t_0))\in L^1(\Omega)\times C(\overline\Omega)$,
with $\|c_2(t_0)\|_\infty\le R$ and
$$\|c_1(t_0)\|_1\le \frac{1}{\log R}\int_{\{c(\cdot,t_0)>R\}} (c\log c)(t_0)\dx
\le \frac{1}{\log R}\int_\Omega (1+c\log c)(t_0)\dx\le \frac{K}{\log R}.$$
Now we choose $R>e^{K/\eps_0}$, where $\eps_0$ is given by Proposition~\ref{lem:decomp}.
Applying  Proposition~\ref{lem:decomp} (after a time shift, taking $t_0$ as new time origin), it follows that
$$\|c(t_0+s)\|_p \le \bar C_0s^{-1/2},\quad 0<s<\min(\eps_0R^{-2},T^*-t_0),\ T_0\le t_0<T^*.$$
Let $s_0=\frac12\min(\eps_0R^{-2},T^*-T_0)$ and $T_1=T_0+s_0<T^*$. For any $t\in[T_1,T^*)$, applying this with 
$t_0:=t-s_0$ yields
$$\|c(t)\|_p \le \bar C_0s_0^{-1/2},\quad T_0\le t<T^*.$$
We deduce from Theorem~\ref{thm1}(iii) that $T^*=\infty$.
Also, by \eqref{estimciHolderC} or \eqref{estimciHolderCcyl}, we have
$$
\|c\|_{C^{1+\alpha,\alpha/2}(\overline\Omega\times[1,\infty))}<\infty.
$$
Finally, if $M<1$, then \eqref{entropydissip} guarantees that
$\int_{T_1}^\infty \int_{\Omega} |\nabla c^{1/2}|^2 \dx \dt<\infty$
and \eqref{stabilc1} follows similarly as in the proof of Theorem~\ref{thm2b} 
or \ref{thm1aneg}.
\end{proof}

\section{Proof of Theorem~\ref{thmBU}} \label{sec:blowup}

We prove the result for $n\ge 2$ (the case $n=1$ is easier and can be handled by very minor changes).
We split the (regular part of the) boundary of the cylinder $\Omega=(0,L)\times B'_R$ in three pieces, denoting 
$$\Sigma_1=\partial\Omega\cap\{x_1=0\},\quad \Sigma_2=\partial\Omega\cap\{x_1=L\},\quad 
\Sigma_3=\partial\Omega\cap\{0<x_1<L\}.$$
Recalling 
\be\label{defaA0}
A(t)=\int_{\partial\Omega}c^m\nu \dsigma=\int_\Omega \nabla c^m \dx,
\ee
we denote
\be\label{defaA}
a(t)=e_1\cdot A(t)=\int_\Omega (c^m)_{x_1} \dx
=\int_{\Sigma_2}c^m \dsigma-\int_{\Sigma_1}c^m \dsigma.
\ee
We shall use the following monotonicity property.

\begin{prop}\label{lemmonotcyl}
Assume $m\ge 1$, $f(c)=c^m$ and \eqref{hypcyl1}-\eqref{hypcyl3}.

\begin{itemize}[topsep=-1pt]\setlength\itemsep{-1pt}

\item[(i)] Let $n\ge 2$. For each $t\in (0,T^*)$,  $c(\cdot,t)$ is axisymmetric with respect to $e_1$.
\smallskip

\item[(ii)] If $c_{0,x_1}\le 0$ and $c_{0,x_1}\not\equiv 0$, then $c_{x_1}<0$ in $\Omega\times(0,T^*)$.
\end{itemize}
\end{prop}

\begin{proof}
(i) This is an immediate consequence of the uniqueness of the maximal solution.

\smallskip

(ii) We prove the result for $n\ge 2$.
We shall show that $c_{x_1}\le 0$.
To this end we need only consider the case when 
\be\label{c0pos}
\min_{\overline\Omega}c_0>0.
\ee
Indeed, then considering $c_{0,j}=c_0+j^{-1}$ and passing to the limit via the continuous dependence property in 
Proposition~\ref{locexist-cyl}(ii),
it will follow that $c$ is nonincreasing in the direction~$x_1$.

As a consequence of (i), we have $\int_\Omega \partial_{x_i} c^m\dx=0$ for each $i\in\{2,\dots,n\}$, hence
\be\label{defaA02}
A(t)= e_1\int_\Omega \partial_{x_1} c^m \dx.
\ee
Consider the auxiliary problem
\begin{subequations}\label{pbmPhat}
\begin{align}
	\hfill\partial_t \hat c&=\nabla\cdot\Bigl(\nabla \hat c+\Bigl|\int_\Omega \partial_{x_1} \hat c^m \dx\Bigr| \,\hat c \,e_1\Bigr),
	 &&x\in\Omega,\ t>0,\\
	\hfill 0&=\Bigl(\nabla \hat c+\Bigl|\int_\Omega \partial_{x_1} \hat c^m\dx\Bigr| \,\hat c \,e_1\Bigr)\cdot\nu,&&x\in\partial\Omega,\ t>0, \\
	\hfill \hat c(x,0)&=c_0(x),&&x\in\Omega.
\end{align}
\end{subequations}
By Proposition~\ref{locexist-cyl} and Remark~\ref{remPsi},
problem~\eqref{pbmPhat} has a unique maximal classical solution 
$\hat c\in C(\overline\Omega\times[0,\hat T))$ such that
\be\label{regulCalpha0}
\hat c\in C_{loc}^{2+\alpha,1+\alpha/2}((\Omega\cup\Gamma)\times(0,\hat T))
\cap C_{loc}^{1+\alpha,\alpha/2}(\overline\Omega\times(0,\hat T))
\cap C_{loc}^{\alpha,\alpha/2}(\overline\Omega\times(0,\hat T))
\ee
for some $\alpha\in(0,1)$, where $\hat T$ denotes its existence time.
Now set $z=\partial_{x_1}\hat c$. 
We have 
\be\label{regulCalpha0z}
z\in C_{loc}^{1+\alpha,\alpha/2}((\Omega\cup\Gamma)\times(0,\hat T))
\cap C_{loc}^{\alpha,\alpha/2}(\overline\Omega\times(0,\hat T)),
\ee
 owing to \eqref{regulCalpha0}.
By (\ref{pbmPhat}a), \eqref{regulCalpha0} and interior parabolic regularity, we get
$z\in C^{2,1}(\Omega\times(0,\hat T))$ and $z$ satisfies
\be\label{zPDE}
z_t-\Delta z=b(t) \partial_{x_1}z \ \hbox{ in $\Omega\times(0,\hat T)$, \quad where } 
b(t)=\Bigl|\int_\Omega \partial_{x_1} \hat c^m \dx\Bigr|.
\ee
We shall show $z\le 0$ by the Stampacchia maximum principle argument.

On the flat boundary parts, we have:
\be\label{zSigma1}
z=-b(t)\hat c\quad\hbox{on $(\Sigma_1\cup \Sigma_2)\times(0,\hat T)$.}
\ee
As for the boundary part $\Sigma_3\times(0,\hat T)$, where $\nu=(0,R^{-1}x')\perp e_1$, we have
$\hat c_\nu(x_1,x',t)=(0,R^{-1}x')\cdot\nabla\hat c(x_1,x',t)=0$, hence $z_\nu(x_1,x',t)=(0,R^{-1}x')\cdot\nabla(\partial_{x_1})\hat c(x_1,x',t)=0$,
so that
\be\label{zSigma2}
z_\nu=0\quad\hbox{on $\Sigma_3\times(0,\hat T)$.}
\ee
In particular, by \eqref{zPDE} and interior-boundary parabolic regularity (outside of the nonsmooth part of $\partial\Omega$), it follows that
\be\label{zPDEregul}
z\in C^{2,1}((\Omega\cup\Sigma_3)\times(0,\hat T)).
\ee

Since \eqref{zPDEregul} does not guarantee the integrability of $\Delta z$ in $\Omega$, we shall integrate on $\Omega_\eta:=(\eta,L-\eta)\times B'_R$ for $\eta>0$.
Denote $\Sigma'_\eta=\{\eta,L-\eta\}\times B'_R$ and $\Sigma''_\eta=(\eta,L-\eta)\times \partial B'_R$.
Fix $0<\eps<T<\hat T$.
Since $\hat c>0$ in $\overline\Omega\times[0,\hat T)$ 
by \eqref{c0pos} and Proposition~\ref{locexist-cyl}, it follows from 
\eqref{regulCalpha0z} and \eqref{zSigma1} that, for all $\eta>0$ sufficiently small,
$z\le 0$ on $\Sigma'_\eta\times(\eps,T)$.
Now multiplying \eqref{zPDE} by $z_+$, integrating by parts
and using \eqref{zSigma2}, we have, for a.e.~$t\in(\eps,T)$,
$$\begin{aligned}	
\frac12\frac{\diff}{\dt}\int_{\Omega_\eta} z_+^2(t)\dx
&=\int_{\Omega_\eta} (\Delta z)z_+\dx +b(t)\int_{\Omega_\eta} (\partial_{x_1}z)z_+\dx\\
&=-\int_{\Omega_\eta} |\nabla z_+|^2\dx+\int_{\Sigma'_\eta} z_\nu z_+\dsigma +\int_{\Sigma''_\eta} z_\nu z_+\dsigma+b(t)\int_{\Omega_\eta} (\partial_{x_1}z_+)z_+\dx\\
&\le -\int_{\Omega_\eta} |\nabla z_+|^2\dx +\int_{\Omega_\eta} |\partial_{x_1}z_+|^2\dx +b^2(t)\int_{\Omega_\eta} z_+^2\dx
\le b^2(t)\int_{\Omega_\eta} z_+^2\dx.
\end{aligned}	$$
Observe that $z\in C([0,\hat T);L^2(\Omega))$ by \eqref{regulH1pApp},
$z(\cdot,0)\le 0$ by assumption, and $B:=\sup_{t\in(0,T)}|b|$ $<\infty$.
Integrating in time, we get $\int_{\Omega_\eta} z_+^2(t)dx\le e^{2BT^2}\int_{\Omega_\eta} z_+^2(\eps)dx$.
Letting $\eta\to 0$ and then $\eps\to 0$, we deduce that $z\le 0$ in $\Omega\times(0,\hat T)$.

It follows in particular that, for all $t\in (0,\tilde T)$,
$b(t)=-\int_\Omega \partial_{x_1} \hat c^m \dx$ so that,
recalling \eqref{defaA0} and \eqref{defaA02},
$\hat c$ solves the original problem \eqref{pbmP} on $(0,\hat  T)$,
hence $\hat T\le T^*$ and $\hat c=c$ on $(0,\hat  T)$ by uniqueness. Finally, we cannot have $\hat T<T^*$, since otherwise
$\hat T<\infty$ and $\limsup_{t\to \hat T} \|\hat c(t)\|_\infty=\limsup_{t\to \hat T} \|c(t)\|_\infty<\infty$: a contradiction.
We have thus proved that $c_{x_1}=z\le 0$ on $(0,T^*)$.
In view of \eqref{zPDE}, the assertion then follows from the strong maximum principle.
\end{proof}

\begin{proof}[Proof of Theorem \ref{thmBU}]
{\bf Step 1.} {\it Auxiliary functional.}
For $t\in(0,T^*)$, we set 
$$\phi(t)=\int_\Omega  x_1c\dx\ge 0.$$
Using $\int_\Omega c=M$ and $(\nabla c-A(t)c)\cdot\nu=0$ a.e.~on the boundary, we have
\be\label{phiprime1}
\begin{aligned}	
\phi'
&=\int_\Omega x_1c_t
 =\int_\Omega x_1\nabla\cdot(\nabla c-A(t)c) \dx \\
&=-\int_\Omega e_1\cdot(\nabla c-A(t)c) \dx
=-\int_\Omega c_{x_1}\dx+Ma(t).
\end{aligned}	
\ee

{\bf Step 2.} {\it Case $m=1$.} 
Since $c_{x_1}\le 0$, we note that
\be\label{phiprime1a}
\int_{\Sigma_1}c(t)d\sigma
\ge L^{-1}\int_0^L \Bigl(\int_{|x'|<R}c(x_1,x',t)\dx'\Bigr)\dx_1
= L^{-1} \int_\Omega c\dx=ML^{-1}
\ee
and
\be\label{phiprime1b}
\begin{aligned}	
\int_{\Sigma_2}c(t)d\sigma
&=2L^{-2}\int_0^L x_1\Bigl(\int_{\Sigma_2}c(t)d\sigma\Bigr)\dx_1 \\
&\le 2L^{-2}\int_0^L x_1\Bigl(\int_{|x'|<R}c(x_1,x',t)\dx'\Bigr)\dx_1\le 2L^{-2}\phi(t),
\end{aligned}	
\ee
hence, by \eqref{defaA},
\be\label{phiprime2}
-a(t)\ge L^{-1}(M-2L^{-1}\phi(t)).
\ee
Since $M>1$, \eqref{phiprime1} and \eqref{phiprime2} yield
\be\label{phiprime3}
\phi'=(M-1)a(t)\le -L^{-1}(M-1)[M-2L^{-1}\phi(t)].
\ee
If
\be\label{phiprime4}
\phi(0)<ML/2,
\ee
then \eqref{phiprime3} implies that $\phi$ is nonincreasing in time, hence
$$\phi'\le -L^{-1}(M-1)[M-2L^{-1}\phi(0)]<0,\quad 0<t<T^*.$$
But since $\phi\ge 0$, the latter necessarily implies
$$T^*\le \frac{\phi(0)}{L^{-1}(M-1)[M-2L^{-1}\phi(0)]}<\infty.$$
To complete the proof in the case $m=1$ it thus suffices to check that \eqref{phiprime4}
holds under our assumptions on $c_0$.
To this end, using the change of variable $z=(x_1/L)^2$, we observe that
\be\label{phiprime5}
\begin{aligned}	
\phi(0)&=
\int_0^L x_1\Bigl(\int_{|x'|<R}c_0(x_1,x')\dx'\Bigr)\dx_1
=\frac{L^2}{2}\int_0^1 \Bigl(\int_{|x'|<R}c_0(L\sqrt z,x')\dx'\Bigr)\dz \\
&\le\frac{L^2}{2}\int_0^1 \Bigl(\int_{|x'|<R}c_0(Lz,x')\dx'\Bigr)\dz \\
&=\frac{L}{2}\int_0^L \Bigl(\int_{|x'|<R}c_0(x_1,x')\dx'\Bigr)\dx_1=\frac{ML}{2},
\end{aligned}	
\ee
since $c_0$ is nonincreasing w.r.t.~$x_1$. We claim that the inequality in \eqref{phiprime5} is strict.
Indeed, otherwise, since $c_0$ is nonincreasing w.r.t.~$x_1$, we have
$c_0(Lz,x')=c_0(L\sqrt z,x')$ for all $z\in [0,1]$ and $x'\in B'_R$. 
It follows that $c_0(x_1,x')=c_0(L\rho^{2^i},x')$ for all $x_1\in[L\rho^{2^i},L\rho]$, $x'\in B'_R$, $\rho\in(0,1)$ and $i\ge 0$ integer.
Letting $i\to\infty$, we get $c_0(x_1,x')\equiv c_0(0,x')$ on $[0,L\rho]$ for all $\rho\in(0,1)$, hence $\partial_{x_1}c_0\equiv 0$: 
a contradiction with our assumption.

{\bf Step 3.} {\it Case $m>1$.} 
In this case, \eqref{phiprime1} yields
$$\phi'=\int_{\Sigma_1}c(t)d\sigma+M\int_{\Sigma_2}c^m(t)d\sigma-\int_{\Sigma_2}c(t)d\sigma-M\int_{\Sigma_1}c^m(t)d\sigma.$$
To control the second term, we use estimate \eqref{estimKx1} from Theorem~\ref{BUprofileaisym}, that we will prove independently below
(note that \eqref{estimKx1} holds {\it without} assuming $T^*<\infty$).
Dropping the third term and using H\"older's inequality for the last term, it follows that
\be\label{cmell0}
\phi'\le\int_{\Sigma_1}c(t)d\sigma+M(M_0L^{-1})^m|B'_R|
-M|B'_R|^{1-m}\Bigl(\int_{\Sigma_1}c(t)d\sigma\Bigr)^m.
\ee

Pick $\ell\in(0,L/2]$ to be fixed later. Take $t\in(0,T^*)$.
If $\phi(t)\le \ell M/2$, then we have
$$\begin{aligned}	
\int_{\Sigma_1}c(t)d\sigma
&\ge \ell^{-1}\int_0^\ell \Bigl(\int_{|x'|<R}c(x_1,x',t)\dx'\Bigr)\dx_1\\
&=  \ell^{-1}\Bigl\{M-\ell^{-1}\int_\ell^L x_1\Bigl(\int_{|x'|<R}c(x_1,x',t)\dx'\Bigr)\dx_1\Bigr\}
\ge  \ell^{-1}(M-\ell^{-1}\phi(t)),
\end{aligned}$$
hence
\be\label{cmell1}
\int_{\Sigma_1}c(t)d\sigma\ge  M\ell^{-1}/2.
\ee
Assume $\ell \le (4|B'_R|)^{-1}ML/M_0$. If $\phi(t)\le \ell M/2$ then
$\int_{\Sigma_1}c(t)d\sigma\ge M\ell^{-1}/2\ge 2^{1/m}(M_0L^{-1})|B'_R|$, hence
$$M|B'_R|^{1-m}\Bigl(\int_{\Sigma_1}c(t)d\sigma\Bigr)^m\ge 2M(M_0L^{-1})^m|B'_R|,$$
so that \eqref{cmell0} implies
\be\label{cmell2}
\begin{aligned}	
\phi'(t)
&\le\int_{\Sigma_1}c(t)d\sigma-\frac{M}{2}|B'_R|^{1-m}\Bigl(\int_{\Sigma_1}c(t)d\sigma\Bigr)^m \\
&=-\left\{\frac{M}{2} |B'_R|^{1-m}\Bigl(\int_{\Sigma_1}c(t)d\sigma\Bigr)^{m-1}-1\right\}\int_{\Sigma_1}c(t)d\sigma.
\end{aligned}
\ee

Now choose 
$$\ell=\min\left\{\frac{L}{2},(4|B'_R|M_0)^{-1}ML,2^{-(m+1)/(m-1)} |B'_R|^{-1}M^{m/(m-1)}\right\},$$
hence in particular 
$(M/2)^m \ge 2(|B'_R|\ell)^{m-1}$.
Setting $K:=\ell M/2$, it follows from \eqref{cmell1}, \eqref{cmell2} that, for all $t\in(0,T^*)$,
$$\phi(t)\le K\Longrightarrow \phi'(t)
\le-\Bigl\{(M/2)^m (|B'_R|\ell)^{1-m}-1\Bigr\}M\ell^{-1}/2\le-M\ell^{-1}/2<0.$$
Since $\phi\ge 0$ (and $\phi$ is continuous at $t=0$), we thus have $T^*<\infty$ whenever $\phi(0)<K$.
\end{proof}

\section{Proof of results on blow-up asymptotics} \label{sec:blowup:asym}

\begin{proof}[Proof of Theorem~\ref{BUprofileaisym}]

(i) Set 
$$\mu(x',t)=\int_0^L c(x_1,x',t)\dx_1,\quad (x,t)\in(0,T^*)\times B'_R.$$
Denote $\Delta_{x'}=\sum^n_{i=2}\partial^2_{x_i}$. For $(x',t)\in(0,T^*)\times B'_R$, we have
$$\begin{aligned}	
\mu_t-\Delta_{x'}\mu&=\int_0^L \bigl(c_t-\Delta_{x'}\bigr)(x_1,x',t)\dx_1
=\int_0^L  \bigl(\partial^2_{x_1}c-(A(t)\cdot e_1)\partial_{x_1}c\bigr)(x_1,x',t)\dx_1\\
&=\Bigl[\bigl(\partial_{x_1}c-(A(t)\cdot e_1)c\bigr)(x_1,x',t)\Bigr]_{x_1=0}^{x_1=L}
=e_1\cdot\Bigl[\bigl(\nabla c-A(t)c\bigr)(x_1,x',t)\Bigr]_{x_1=0}^{x_1=L}=0.
\end{aligned}$$
On the other hand, pick any $(x',t)\in\partial B'_R\times(0,T^*)$ and
denote by $\nu'$ the outer normal vector to $\partial B'_R$ at $x'$.
For each $x_1\in(0,L)$, the outer normal vector $\nu$ to $\partial\Omega$ at $(x_1,x')$ satisfies $\nu=(0,\nu')\perp e_1$,
hence $c_\nu(x_1,x',t)=(e_1\cdot\nu)A(t)c(x_1,x',t)=0$. 
Therefore, $\nu'\cdot\nabla_{x'}\mu(x',t)=\int_0^L \nu\cdot\nabla c(x_1,x',t)dx_1=0$.
Consequently, $\mu$ solves the heat equation with homogeneous Neumann boundary conditions.
It follows from the maximum principle that
$$\mu(x',t)\le \|\mu(\cdot,0)\|_{L^\infty(B'_R)}=\Bigl\|\int_0^L c_0(x_1,\cdot)\dx_1\Bigr\|_{L^\infty(B'_R)}=M_0.$$
Since $c_{x_1}\le 0$, we deduce that
$$x_1c(x_1,x',t)\le\int_0^{x_1} c(y,x',t)\dy\le \mu(x',t)\le M_0,$$
hence the assertion.

(ii) Since $c\ge 0$ and $c_x\le 0$ by Proposition~\ref{lemmonotcyl}, we first note that
\be\label{boundmonot}
xc(x,t)\le \int_0^x c(y,t)\dy \le \|c(t)\|_1=\|c_0\|_1.
\ee
Next we introduce the auxiliary functional
$$\phi:=c_x+c^{m+1}-Nc,\quad N=(L^{-1}\|c_0\|_1)^m. $$
We have $(c_x-Nc)_t - (c_x-Nc)_{xx} = -a(t) (c_x-Nc)_x$ and
$$(c^{m+1})_t - (c^{m+1})_{xx} = (m+1)c^m(c_t - c_{xx}) - m(m+1)c^{m-1}(c_x)^2 
\le  -a(t)(c^{m+1})_x$$
hence
$$\phi_t - \phi_{xx}\le  -a(t) \phi_x,$$
with
$$\phi= a(t)c+c^{m+1}-Nc= 
[c^m-c^m(0,t)+c^m(L,t)-N]c\le 0\ \hbox{ on $\{0,L\}\times(0,T^*)$,}$$
owing to \eqref{boundmonot}.
By the maximum principle, we deduce that 
$\phi \le C_1:=\max_{[0,L]}\phi_+(0,x)$, hence
\be\label{supersolprofile}
 c_x+c^{m+1}\le Nc+C_1 \quad x\in(0,L),\ t\in(0,T^*).
 \ee
By direct computation, if the constant $\lambda$ is sufficiently large, then
$\psi(x):=(mx)^{-1/m}+\lambda$ satisfies 
$$\begin{aligned}
\psi_x+\psi^{m+1}
&= -(mx)^{-(m+1)/m}+(mx)^{-(m+1)/m}\bigl[1+\lambda(mx)^{1/m}\bigr]^{m+1} \\
& \ge -(mx)^{-(m+1)/m}+(mx)^{-(m+1)/m}\bigl[1+(m+1)\lambda(mx)^{\frac{1}{m}}+\ts\frac{m(m+1)}{2}\lambda^2(mx)^{\frac{2}{m}}\bigr] \\
&=(m+1)\lambda\bigl[(mx)^{-1}+\ts\frac{m}{2}\lambda(mx)^{(1-m)/m}\bigr] \\
&\ge (m+1)\lambda\bigl[(mx)^{-1}+\ts\frac{1}{2}\lambda m^{1/m}L^{(1-m)/m}\bigr] \\
& \ge N[(mx)^{-1/m}+\lambda]+C_1=N\psi+C_1
\end{aligned}$$
on $[0,L]$, hence $\psi$ is a supersolution of \eqref{supersolprofile}, which tends to $\infty$ as $x\to 0$.
The conclusion follows.
\end{proof}

\begin{proof}[Proof of Theorem~\ref{thmBUrate}]
We modify an argument from \cite{W82} (see also \cite{FQ} for a related problem
with local nonlinear boundary conditions).
Let $t_0\in(0,T^*)$. By \eqref{NewRepres} and \eqref{smooth2}, for all $t\in(0,T^*)$, we have
$$\begin{aligned}
\|c(t_0+t)\|_\infty
&\le \|c(t_0)\|_\infty +C\int_0^t (t-s)^{-1/2}|A(t_0+s)|\|c(t_0+s)\|_\infty\diff \! s\\
&\le \|c(t_0)\|_\infty +C\int_0^t (t-s)^{-1/2}\|c(t_0+s)\|_\infty^{m+1}\diff \! s,
\end{aligned}$$
where $C=C(\Omega,m)>0$. 
Since $T^*<\infty$ (hence $c(t_0)\not\equiv 0$), we may set
$$T=\inf \bigl\{\sigma\in(t_0,T^*);\, \|c(\sigma)\|_\infty\ge 2\|c(t_0)\|_\infty\bigr\}\in(t_0,T^*).$$
Then
$$\begin{aligned}
2\|c(t_0)\|_\infty
&=\|c(T)\|_\infty\le \|c(t_0)\|_\infty +C(2\|c(t_0)\|_\infty)^{m+1}\int_0^{T-t_0} (T-t_0-s)^{-1/2}\diff \! s \\
&\le\|c(t_0)\|_\infty +C(T-t_0)^{1/2}\|c(t_0)\|_\infty^{m+1},
\end{aligned}$$
hence $1\le C(T^*-t_0)^{1/2}\|c(t_0)\|_\infty^m$, i.e.~\eqref{rateBU1}.
As for property \eqref{rateBU2}, it is a direct consequence of Lemma~\ref{lem:L2b} and of Theorem~\ref{thm1}(iii).
\end{proof}

\section{Appendix: local theory for regular initial data in the case of cylinder.}

We here provide the necessary auxiliary results on the solvability and regularity of linear and nonlinear problems
with inhomogeneous Neumann boundary conditions
in the case of nonsmooth (cylindrical) domains.
Recall that the space $E_\tau$ and the notion of classical solution are defined in $\eqref{defEtau}$.
In what follows, for $c_0\in C(\overline\Omega)$, 
by a mild solution of \eqref{pbmP} on $[0,T)$, 
we mean a function $c\in C(\overline\Omega\times[0,T))$ such that
\be\label{NewRepresApp}
c(t)=S(t)c_0+\int_0^t K_\nabla(t-s)[A(s) c(\cdot,s)]\diff \! s,\quad 0<t<T,
\ee
where $K_\nabla, A$ are respectively defined in \eqref{defKnabla}, \eqref{NewRepres2}, and
the integral in \eqref{NewRepresApp} is absolutely convergent in $C(\overline\Omega)$.

We shall prove:

\goodbreak 

 \begin{prop} \label{locexist-cyl}
 Let $\Omega$ be given by \eqref{defCyl}, $f$ be locally Lipschitz continuous on $\R$ 
and $c_0\in C(\overline\Omega)$.

\begin{itemize}[topsep=-1pt]\setlength\itemsep{-1pt} 
\item[(i)] {\bf (Existence-uniqueness)} Problem \eqref{pbmP} admits a unique, maximally defined, mild solution
$c\in C(\overline\Omega\times[0,T^*))$.
We have $c\in E_{T^*}$ and $c$ is a classical solution of \eqref{pbmP} on $[0,T^*)$.
\smallskip

\item[(ii)] {\bf (Positivity)}  
If $c_0\ge 0$, then $c\ge 0$ in $\overline\Omega\times[0,T^*)$.
If $\min_{\overline\Omega} c_0>0$, then
$c>0$ in~$\overline\Omega\times[0,T^*)$.
\smallskip

\item[(iii)] {\bf (Additional Sobolev regularity)}  
The solution $c$ satisfies
\be\label{ustrong2}
c\in  L^2_{loc}((0,T^*);H^2(\Omega))\cap H^1_{loc}((0,T^*);L^2(\Omega))
\ee
and we have the implications
\be\label{regulH1pApp}
c_0\in H^1(\Omega) \Longrightarrow  c\in C([0,T^*);H^1(\Omega))
\ee
and, for any $p\in[1,\infty)$, 
\be\label{regulW1pApp}
c_0\in W^{1,p}(\Omega) \Longrightarrow  c\in L^\infty_{loc}([0,T^*);W^{1,p}(\Omega)).
\ee

\smallskip

\item[(iv)] {\bf (Continuous dependence)}  
In addition, if $c_{0,j}\to c_0$ in $C(\overline\Omega)$ and $T< T^*(c_0)$, then we have
$T^*(c_{0,j})>T$ and $\sup_{t\in[0,T]}\|c_j-c\|_\infty\to 0$ as $j\to\infty$.

\smallskip

\item[(v)] {\bf (Extended uniqueness and continuation)}  
Let $T>0$. If a function in the class $\tilde E_T:=C(\overline\Omega\times[0,T))\cap E_T$ 
is a classical solution of \eqref{pbmP}, then it is a mild solution.
Consequently uniqueness holds also in the class of classical solutions in $\tilde E_T$.
In addition, we have the implication
\be\label{contTast}
T^*<\infty \Longrightarrow \lim_{t\to T^*}\|c(t)\|_\infty=\infty.
\ee
\smallskip

\item[(vi)] {\bf (Schauder estimates)}  
 Let $\alpha\in(0,1)$, $\beta\in(0,1/4)$, $T_0,N_0>0$, $T\in(0,\min(T_0,T^*))$, $\eps\in(0,T)$ and assume that 
	\be\label{DefN0}
	\sup_{t\in[0,T]}\|c(t)\|_\infty\le N_0.
	\ee
For any compact subset $\Sigma\subset \Omega\cup\Gamma$, we have
\be\label{regulCalpha}
\|c\|_{C^{2+\beta,1+\beta/2}(\Sigma\times[\eps,T])}
+\|\nabla c\|_{C^{\beta}(\overline\Omega\times[\eps,T])}
+\|c\|_{C^{\alpha,\alpha/2}(\overline\Omega\times[\eps,T])}
\le C(N_0,\eps,\Sigma,T_0).
\ee
\end{itemize}
 \end{prop} 

\begin{remark} \label{remPsi}
By exactly the same proof, Proposition~\ref{locexist-cyl} remains valid if $A(c(t))$ is replaced by 
$A(c(t))= \Phi\bigl(\int_{\partial\Omega}f(c)\nu \dsigma\bigr)$,
where $\Phi:\R\to\R$ is, for instance, any globally Lipschitz continuous function.
\end{remark}

\begin{remark} \label{remcorners}
(i) The need for the results in this appendix is motivated by the fact that
global $C^{2,1}$ regularity up to the boundary is not expected
in general for inhomogeneous Neumann problems in a cylinder, as shown by the following 
simple counter-example for $n=2$: 
Consider the heat equation $u_t-\Delta u = 0$ (or even the Laplace equation $\Delta u = 0$) in the square $\Omega=(0,1)^2$
with the $C^\infty$, time independent Neumann condition $\partial_\nu u = y-x$.
On the boundary part $\{x=1,\ 0<y<1\}$, we have $u_x =  u_\nu = y-1$
 hence $u_{yx}=1$, and on $\{y=1,\ 0<x<1\}$, we have $u_y =  u_\nu = 1-x$ hence $u_{xy}=-1$.
Therefore $u$ cannot be $C^2$ in space at the corner $(1,1)$.
This phenomenon (which is in fact independent of the PDE) is a consequence
 of the fact that the normal vector field is discontinuous at the corner points.
 
(ii) The regularity proof in Proposition~\ref{locexist-cyl} will take advantage of the possibility to write $\Omega$ 
as a product of smooth domains, and its Neumann heat kernel as a product of the corresponding kernels.
This is an essential feature and the above regularity properties are rather specific to the case of cylindrical domains.
Indeed it is well known that, e.g.,~$C^1$ regularity up to the boundary is not true in general Lipschitz domains.
 Consider for instance the case of the plane sectorial domain (in polar coordinates)
 $\Omega=\{(r,\theta);\ r>0,\ 0<\theta<\lambda \pi\}$ with $\lambda\in(1,2)$, for which the 
 (time independent) function $u=u(r,\theta)=r^k\cos(k\theta)$ with $k=1/\lambda$
 satisfies $u\in C^2(\overline\Omega\setminus\{O\})\cap C(\overline\Omega)$
with $\Delta u=0$ in $\Omega$ and $u_\nu=0$ on $\partial\Omega\setminus\{O\}$,
but $u$ is not $C^1$ at the origin.
\end{remark}

We shall use the following basic solvability result for the inhomogeneous linear Neumann problem:
\begin{subequations}\label{pbmNeumannLinear}
\begin{align}
	\hfill\partial_t u&=\Delta u+h,	 &&x\in\Omega,\ t>0,\\
	\hfill \partial_\nu u&=g,&&x\in\partial\Omega,\ t>0, \\
	\hfill u(x,0)&=u_0(x),&&x\in\Omega.
\end{align}
\end{subequations}

\begin{prop} \label{prop:Cylinder}
Let $\Omega$ be given by \eqref{defCyl}, $T>0$ and set $Q_T:=\Omega\times(0,T)$.
Let $u_0\in H^1(\Omega)$, $h\in L^2(Q_T)$ and 
$g\in L^2(0,T;H^{1/2}(\partial\Omega))\cap H^{1/4}(0,T;L^2(\partial\Omega))$.

(i) There exists a unique 
\be\label{ustrong}
u\in L^2(0,T;H^2(\Omega))\cap H^1(0,T;L^2(\Omega))\cap C([0,T);H^1(\Omega))
\ee
which is a strong solution of \eqref{pbmNeumannLinear},
i.e.~$u$ solves the PDE a.e.~in $Q_T$ and the boundary conditions in the sense of traces.
Moreover, $u$ is given by the representation formula in $Q_T$:
\be\label{represApp}
\begin{aligned}
u(x,t)
&=\int_\Omega G(x,y,t)u_0(y)\diff \! y+\int_0^t \int_\Omega G(x,y,t-s) h(x,s) \diff \! y\diff \! s \\
&\qquad +\int_0^t \int_{\partial\Omega} G(x,y,t-s) g(y,s)\diff \! \sigma_y \diff \! s.
\end{aligned}
\ee

(ii) Let $\alpha\in(0,1)$, $\eps\in(0,T)$, $M>0$ and assume in addition that
$g\in C^{\alpha,\alpha/2}(\partial\Omega\times[\eps/2,T])$, 
$u, h\in L^\infty(\Omega\times(\eps/2,T))$, with
$$\|g\|_{C^{\alpha,\alpha/2}(\overline\Omega\times[\eps/2,T])}+\|u\|_{L^\infty(\Omega\times(\eps/2,T))}+\|h\|_{L^\infty(\Omega\times(\eps/2,T))}\le M.$$
Then, for all $\beta\in(0,\alpha/4)$, we have $u\in C^{1+\beta,\beta}(\overline\Omega\times[\eps,T])$  and
\be\label{uclass1}
\|u\|_{C^{1+\beta,\beta}(\overline\Omega\times[\eps,T])}\le C(M,\eps,\beta,\alpha).
\ee
\end{prop}

\begin{remark} \label{remtrace}
The trace operator is continuous $H^1(\Omega)\to H^{1/2}(\partial\Omega)$
(this fact is valid in bounded Lipschitz domains; cf.~\cite{Gag}).
For $u$ satisfying \eqref{ustrong}, we have $\nabla u\in L^2_{loc}(0,\tau;H^1(\Omega))$, 
hence $\partial_\nu u$ is well defined as an element of $L^2_{loc}(0,\tau;H^{1/2}(\Omega))$.
\end{remark}

Whereas Proposition~\ref{prop:Cylinder} is well known for smooth domains 
or for the homogeneous Neumann case $g=0$, 
this doesn't seem to be the case for the full problem \eqref{pbmNeumannLinear} in cylindrical domains.
We thus provide a proof. 
To this end we need the following lemma, 
which provides a suitable lifting 
of the Neumann trace operator.

\begin{lemm} \label{lem:Cylinder}
Let $\Omega$ be given by \eqref{defCyl}.
For any $g\in L^2(0,\tau;H^{1/2}(\partial\Omega))\cap H^{1/4}(0,\tau;L^2(\partial\Omega))$,
there exists a function
$v\in L^2(0,\tau;H^2(\Omega))\cap H^1(0,\tau;L^2(\Omega))$ 
such that $v_\nu=g$ on $\Sigma\times(0,\tau)$ in the sense of traces, 
and $v\in C([0,\tau];L^2(\Omega))$ with $v(0)=0$. Moreover, 
the map $g\mapsto v$ is linear continuous on the above spaces.
\end{lemm}

\begin{proof}[Proof of Lemma~\ref{lem:Cylinder}]
The regular part of $\partial\Omega$ can be split into the 3 pieces $\Sigma_1=(0,L)\times\partial B'_R$,
$\Sigma_2=\{0\}\times B'_R$ and $\Sigma_3=\{L\}\times B'_R$.
We consider only the part $\Sigma_1$, the other two can be treated similarly
(and are slightly simpler actually). 

Let $(\psi_p)_{p\in\N}$ with $\psi_p=\psi_p(\omega)$ be a Hilbert basis of $L^2(\Sigma_1)$ made of eigenfunctions of the Laplace-Beltrami operator on $\Sigma_1$ with homogeneous Neumann boundary conditions at $x_1=0,L$
and $\lambda_p\ge 0$ the corresponding eigenvalues
(note that $(\psi_p)$ can be obtained by tensorizing a Hilbert basis of $L^2(\partial B'_R)$
made of eigenfunctions of the Laplace-Beltrami operator on $\partial B'_R$ with the functions $(\cos(j\pi x_1/L))_{j\in\N}$).
Set $\phi_m=\sqrt{2/\tau}\sin(m\pi t/\tau)$ for $m\in\N$.
Let $\bar g=g_{|\Sigma_1\times(0,\tau)}$.
Observing that the set $\{\phi_m(t)\psi_p(\omega)\}_{m,p}$
forms a Hilbert basis of $L^2(\Sigma_1\times(0,\tau))$, we can expand $\bar g$ as
$$\bar g(\omega,t)=\sum_{m,p} a_{m,p}\phi_m(t)\psi_p(\omega),\quad
\hbox{with } \|\bar g\|_{L^2(\Sigma_1\times(0,\tau))}^2=\sum_{m,p} a^2_{m,p}.$$
Set $\rho=|x'|$ and fix a smooth function $\theta(s)\ge 0$ 
such that $\theta(0)=1$, $\theta'(0)=-1$ and $\theta(s)=0$ for $s\ge R/2$.
We then define
\be\label{defvlift}
v(x,t)=v(\omega,\rho,t)=\sum_{m,p} \frac{a_{m,p}}{k_{m,p}}\theta(k_{m,p}(R-\rho))\phi_m(t)\psi_p(\omega),\quad
k_{m,p}=2+\sqrt{\lambda_p+m}.
\ee
Direct computation yields
$$\begin{aligned}
\|v_t\|_{L^2(Q_\tau)}^2
&=C\sum_{m,p} \frac{m^2 a^2_{m,p}}{k^2_{m,p}}\int_0^R \theta^2(k_{m,p}(R-\rho))\rho^{n-2} d\rho
\le C\sum_{m,p} \frac{m^2 a^2_{m,p}}{k^3_{m,p}},\\
\|D^2_\rho v\|_{L^2(Q_\tau)}^2
&=C\sum_{m,p} k^2_{m,p}a^2_{m,p}\int_0^R (\theta'')^2(k_{m,p}(R-\rho))\rho^{n-2} d\rho
\le C\sum_{m,p} k_{m,p}a^2_{m,p},\\
\|D^2_\omega v\|_{L^2(Q_\tau)}^2
&\le C\|v\|_{L^2(Q_\tau)}^2+C\|\Delta_\Sigma v\|_{L^2(Q_\tau)}^2\\
&\le C\sum_{m,p} \frac{a^2_{m,p}(1+\lambda^2_p)}{k^2_{m,p}}\int_0^R 
\theta^2(k_{m,p}(R-\rho))\rho^{n-2} d\rho
=C\sum_{m,p} \frac{a^2_{m,p}(1+\lambda^2_p)}{k^3_{m,p}}
	\end{aligned}$$
	and
$$
\|g\|_{L^2(0,\tau;H^{1/2}(\Sigma))}^2
=C\sum_{m,p} (1+\sqrt{\lambda_p}) a^2_{m,p},\quad 
\|g\|_{H^{1/4}(0,\tau;L^2(\Sigma))}^2
=C\sum_{m,p} (1+\sqrt{m}) a^2_{m,p}.
$$
In view of the definition of $k_{m,p}$, we have
$$\frac{m^2+1+\lambda^2_p}{k^3_{m,p}}+k_{m,p}
\le C(1+\sqrt{\lambda_p}+\sqrt{m}).$$ 
It follows that $v\in L^2(0,\tau;H^2(\Omega))$
(using the cylindrical parametrization $(\rho,\omega)\in (0,R)\times \Sigma_1$ for $\Omega$),
as well as $v\in H^1(0,\tau;L^2(\Omega))\subset C([0,\tau];L^2(\Omega))$,
and $v$ satisfies the corresponding linear estimates.
Moreover, by \eqref{defvlift}, we see that $v_\nu=v_\rho(\omega,R,t)=\bar g(\omega,t)$ on $\Sigma_1\times(0,\tau)$
and $v_\nu=\pm v_{x_1}=0$ on $(\Sigma_2\cup\Sigma_3)\times(0,\tau)$ (in the sense of traces),
as well as $v(\cdot,0)=0$.
\end{proof}

\begin{proof}[Proof of Proposition~\ref{prop:Cylinder}]
(i) Let $v$ be given by Lemma~\ref{lem:Cylinder} and set $\tilde u=u-v$. 
Then problem \eqref{pbmNeumannLinear} is equivalent to the same problem for $\tilde u$
with $g$ replaced by the homogeneous boundary conditions $\tilde g=0$ 
and $h$ replaced by $\tilde h=h-v_t+\Delta v\in L^2(Q_T)$.
The existence of a strong solution of the latter is well known. 
This is for instance a consequence of the maximal regularity results in \cite{DN} (for cylindrical domains) 
or in \cite{Monn} and the references therein (for general Lipschitz bounded domains), combined with the fact 
 that the domain of the Neumann Laplacian over $L^2(\Omega)$ coincides with $H^2(\Omega)$
 (which follows easily by noting that a Hilbert basis of eigenfunctions $L^2(\Omega)$
 can be obtained by tensorizing Hilbert bases of eigenfunctions on $L^2(0,L)$ and on $L^2(B'_R)$).

Uniqueness follows from the usual energy property.
As for \eqref{represApp}, let us recall the standard argument of the proof for completeness: 
Fixing $\eps\in(0,t)$ and integrating by parts in time and space the expression 
$\int_0^{t-\eps} \int_\Omega G(x,y,t-s) (u_t-\Delta u)(y,s)\diff \! y\diff \! s$,
we obtain
$$ 
\begin{aligned}
&\int_\Omega G(x,y,\eps)u(y,t-\eps)\diff \! y-\int_\Omega G(x,y,t)u_0(y)\diff \! y 
=\int_0^{t-\eps} \int_\Omega G(x,y,t-s) h(y,s) \diff \! y\diff \! s \\
&\qquad -\int_0^{t-\eps} \int_{\partial\Omega} 
G(x,y,t-s) g(y,s)\diff \! \sigma_y  \diff \! s
\end{aligned}
$$
and it then suffices to pass to the limit $\eps\to 0$.

(ii) The proof of this assertion is somewhat technical and is split in several parts.

{\bf Step 1.} {\it Preparations and notation.}
Pick a function $\eta\in C^\infty(\R)$ such that $\eta=1$ on $[\eps,\infty)$
and $\eta=0$ on $(-\infty,\frac{\eps}{2}]$.
Then $\tilde u:=\eta(t)u(x,t)$ is a strong solution of \eqref{pbmNeumannLinear}  
in $Q_T$ with $u_0=0$ and $h, g$ replaced by
$\bar h:=\eta h+\eta' u$ and $\bar g:=\eta g$,
hence $\tilde u$ is given by the representation formula in $Q_T$:
\be\label{represApp2}
\tilde u(x,t)
=\int_0^t \int_{\partial\Omega} G(x,y,t-s) \bar g(y,s)\diff \! \sigma_y \diff \! s
+\int_0^t \int_\Omega G(x,y,t-s) \bar h(x,s) \diff \! y\diff \! s\equiv v+V.
\ee
Note that
\be\label{regulbargh}
\|\bar g\|_{C^{\alpha,\alpha/2}(\partial\Omega\times[0,T])}+\|\bar h\|_{L^\infty(Q_T)}\le C(M,\eps)
\ee
and that, by interior $L^p$ parabolic regularity and standard imbeddings,
$v,V\in C^{1,0}(Q_T)$.
Most of the work is required by the boundary term $v$.
The volume term $V$ will be handled by simpler arguments (see Step~4 of the proof).
We shall estimate the H\"older quotients of $D_xv$ uniformly over $Q_T$, 
which will in particular imply that $v$ is a $C^{1,0}$ function in $\overline{Q_T}$.

Writing $x=(x_1,x_2)$ with $x_1\in\Omega_1=(0,L)$ and $x_2\in\Omega_2=B'_R$,
we shall take advantage of the factorisation 
property $G(x,y,t) = G_1(x_1,y_1,t)\times$ $G_2(x_2,y_2,t)$ of the Neumann heat kernel $G$ of $\Omega = \Omega_1 \times \Omega_2$.
We have
$\partial\Omega = (\partial\Omega_1 \times \Omega_2) \cup (\Omega_1 \times \partial\Omega_2) \cup \mathcal C$, where $\mathcal C$ denotes the set of the corner points. Since the latter has zero surface measure, it does not contribute.
We will consider the contribution $v_1$ of $\Omega_1 \times \partial\Omega_2$, the other one being similar (and slightly simpler).
Denote 
$$Q_{i,T}=\Omega_i\times(0,T),\ \Sigma_T:=\partial\Omega_2\times(0,T),\ J=\{(t,s);\, 0<s<t<T\},
\ \tilde Q=\Omega_1\times\partial\Omega_2\times J.$$
For $(x,t)\in Q_T$, by Fubini's theorem, we can write:
$$v_1(x_1,x_2,t) = \int_0^t \int_{\partial\Omega_2} G_2(x_2,y_2,t-s) w(x_1,y_2,t,s) d\sigma_{y_2} ds
=:T_2(t)w(\underline{x_1},y_2,\underline{t},s)$$
where the partial function $w$ is defined on $\tilde Q$ by
$$w(x_1,y_2,t,s) = \int_{\Omega_1} G_1(x_1,y_1,t-s) \bar g(y_1,y_2,s) dy_1=:
S_1(t-s)\bar g(y_1,\underline{y_2},\underline{s}).$$
Here, the underlined variables act as parameters,
and $(S_1(\tau))_{\tau\ge 0}, (T_2(t))_{t\ge 0}$ respectively correspond to the Neumann heat semigroup on $\Omega_1$ and to the Neumann boundary heat operator on $\Omega_2$. Namely,
$T_2(t)$ is the solution operator $\psi\mapsto v(\cdot,t)$ of the problem
$v_t-\Delta v=0$ in $Q_{2,T}$, with $v=\psi$ on $\Sigma_T$ and $v(\cdot,0)=0$.

For a function $\phi=\phi(z,\xi)$ of the variables $(z,\xi)\in D$
and $\alpha,\beta\in(0,1]$, we will use the following notation
for H\"older brackets and norms with respect to $z$ (uniform in $\xi$):
$$[\phi]_{C^\alpha_z(D)}
=\sup_{(\hat z,\xi), (\bar z,\xi)\in D,\atop \hat z\ne\bar z}
\frac{\big|[\phi(z,\xi)]^{\hat z}_{\bar z}\big|}{|\hat z-\bar z|^\alpha},
\quad \|\phi\|_{C^\alpha_z(D)}=[\phi]_{C^\alpha_z(D)}+\|\phi\|_{L^\infty(D)},$$
where $[\phi(z,\xi)]^{\hat z}_{\bar z}:=\phi(\hat z,\xi)-\phi(\bar z,\xi)$, as well as
$$\|\phi\|_{C^{\alpha,\beta}_{z,\xi}(D)}=[\phi]_{C^\alpha_z(D)}+[\phi]_{C^\beta_\xi(D)}+\|\phi\|_{L^\infty(D)}$$
(and similarly for functions of more than two variables).
We record the following useful property on H\"older norm of functions that are themselves defined by a H\"older quotient:
 setting
$D':=\big\{(\hat z, \bar z, \xi);\ (\hat z, \xi)\in D,\, (\bar z, \xi)\in D,\,\hat z\ne\bar z\big\}$,
we have, for any $\gamma,\beta\in(0,1)$,
\be\label{Holder0}
	\left[\frac{[\phi(z,\xi)]^{\hat z}_{\bar z}}{|\hat z-\bar z|^{\gamma/2}}\right]_{C^{\beta/2}_\xi(D')}
\le 2[\phi]_{C^{\gamma,\beta}_{z,\xi}(D)},\quad \phi\in C^{\gamma,\beta}_{z,\xi}(D).
\ee
To see this,  putting $M=[\phi]_{C^{\gamma,\beta}_{z,\xi}(D)}$,  it suffices to write
$$	\big[\phi(\hat z,\xi)-\phi(\bar z,\xi)\big]_{\bar\xi}^{\hat\xi}
=\big[\phi(\hat z,\xi)\big]_{\bar\xi}^{\hat\xi}-\big[\phi(\bar z,\xi)\big]_{\bar\xi}^{\hat\xi}
=\big[\phi(z,\hat \xi)\big]_{\bar z}^{\hat z}-\big[\phi(z,\bar\xi)\big]_{\bar z}^{\hat z},$$
hence
$\big|[\phi(\hat z,\xi)-\phi(\bar z,\xi)]_{\bar\xi}^{\hat\xi}\big|\le 2M\min(|\hat z-\bar z|^\gamma,|\hat\xi-\bar\xi|^\beta)
\le 2M|\hat z-\bar z|^{\gamma/2}|\hat\xi-\bar\xi|^{\beta/2}$.

In the rest of the proof, $C,C_1,C_2$ will denote generic positive constants depending only on 
$\Omega, T$ and on the various H\"older exponents.
By  \eqref{productG1G2} and standard heat kernel estimates for $G_1, G_2$
in the regular domains $\Omega_1, \Omega_2$ (see, e.g., \cite[Section~2]{Mora}), 
we have, for $\ell\in\{0,1,2\}$,
	\be\label{GaussEstDt}
	|D_tD^i_y G_\ell(x,y,t)|\le C_1t^{-1-(n+i)/2} e^{-C_2|x-y|^2/t},\quad x,y\in\overline\Omega_\ell,\ 0<t<T,\ i\in\{0,1\},
	\ee
	where $G_0=G$ and $\Omega_0=\Omega$.
Also, we have the estimates (see \cite[Theorem~2.3]{Mora} and \cite[Theorem~4.30]{Lie}):
	\be\label{S1Holder1}
	\|S_1(\tau)\phi\|_{C^{\alpha,\alpha/2}_{x_1,\tau}(Q_{1,T})}\le C\|\phi\|_{C^\alpha_{x_1}(\Omega_1)},
\quad \phi\in C^\alpha(\Omega_1),
\ee
	\be\label{S2Holder1}
	\|D_{x_2}T_2(t)\psi\|_{C^{\alpha,\alpha/2}_{x_2,t}(Q_{2,T})}
\le C\|\psi\|_{C^{\alpha,\alpha/2}_{x_2,t}(\Sigma_T)},
\quad \psi\in C^{\alpha,\alpha/2}(\Sigma_T).
\ee

{\bf Step 2.} {\it H\"older estimate of $D_{x_2}v_1$.} 
We shall show that 
\be\label{HolderDx2v1xt}
\|D_{x_2}v_1\|_{C^{\alpha/2,\alpha/4}_{x,t}(Q_T)} \le C\|\bar g\|_{C^{\alpha,\alpha,\alpha/2}_{y_1,y_2,s}(\Omega_1\times \Sigma_T)}.
\ee
 In this step, unless otherwise specified, the norms of $v_1,w$ and $\bar g$
(and their derivatives) will be always taken over $Q_T, \tilde Q$ and $\Omega_1\times \Sigma_T$, respectively.

We first estimate the H\"older norm of the partial function $w$. Namely, we claim that 
	\be\label{S1HolderClaim1}
	\|w\|_{C^{\alpha,\alpha,\alpha/2,\alpha/2}_{x_1,y_2,t,s}}  
	\le C\|\bar g\|_{C^{\alpha,\alpha,\alpha/2}_{y_1,y_2,s}}.  
	\ee
It follows from \eqref{S1Holder1} that
	\be\label{S1Holder1b}
	\|w\|_{C^{\alpha,\alpha/2}_{x_1,t}}\le C\|\bar g\|_{C^\alpha_{y_1}}.
	\ee
Next writing
$\big[w(x_1,y_2,t,s)\big]_{\hat y_2}^{\bar y_2}
= S_1(t-s) \big(\big[\bar g(y_1,\underline{y_2},\underline{s})\big]_{\hat y_2}^{\bar y_2}\big)$
and using
	\be\label{S1Holder2}
	\|S_1(\tau)\phi\|_{L^\infty(\Omega_1)}\le \|\phi\|_{L^\infty(\Omega_1)},\quad \phi\in L^\infty(\Omega_1),\ \tau>0,
\ee
we get
	\be\label{S1Holder1c}
	\|w\|_{C^\alpha_{y_2}}\le \|\bar g\|_{C^\alpha_{y_2}}.
		\ee
For $0<\bar s<\hat s<t<T$, writing
$$\big[w(x_1,y_2,t,s)\big]_{\bar s}^{\hat s}
=S_1(t-\hat s)\Big(\big[\bar g(y_1,\underline{y_2},\underline{s})\big]_{\bar s}^{\hat s}\Big)
+\big[S_1(\tau)\bar g(y_1,\underline{y_2},\underline{\bar s})\big]_{\tau=t-\bar s}^{\tau=t-\hat s}
$$
and using \eqref{S1Holder2} (resp., \eqref{S1Holder1}) to control the first (resp., second) term on the right-hand side, we obtain
$\|w\|_{C^{\alpha/2}_s}\le C\|\bar g\|_{C^{\alpha,\alpha/2}_{y_1,s}}$ which, 
combined with \eqref{S1Holder1b}, \eqref{S1Holder1c}, yields~\eqref{S1HolderClaim1}.

We next estimate the H\"older norms of $D_{x_2}v_1$.
By \eqref{S2Holder1}, \eqref{S1HolderClaim1} we have 
\be\label{HolderfDx2a}
\|D_{x_2}v_1 
\|_{C^{\alpha}_{x_2}} 
\le C\|w 
\|_{C^{\alpha,\alpha/2}_{y_2,s}}\le C\|\bar g\|_{C^{\alpha,\alpha,\alpha/2}_{y_1,y_2,s}}.\ee
On the other hand, for $0<\bar t<\hat t<T$, writing
$$\frac{\big[D_{x_2}v_1(x_1,x_2,t)\big]_{\bar t}^{\hat t}}{|\hat t-\bar t|^{\alpha/4}}
=D_{x_2}T_2(\hat t)\biggl\{\frac{\big[w(\underline{x_1},y_2,\underline{t},s)\big]_{\bar t}^{\hat t}}{|\hat t-\bar t|^{\alpha/4}}\biggr\}
+\frac{\left[D_{x_2}T_2(t)w(\underline{x_1},y_2,\underline{\bar t},s)\right]_{\bar t}^{\hat t}}{|\hat t-\bar t|^{\alpha/4}}$$
and using \eqref{S2Holder1}, 
we obtain
$$	\begin{aligned}
\biggl|\frac{\big[D_{x_2}v_1(x_1,x_2,t)\big]_{\bar t}^{\hat t}}{|\hat t-\bar t|^{\alpha/4}}\biggr|
&\le C\biggl\|\frac{\big[w(x_1,y_2,t,s)\big]_{\bar t}^{\hat t}}{|\hat t-\bar t|^{\alpha/4}}\biggr\|_{C^{\alpha/2,\alpha/4}_{y_2,s}(\tilde Q')} 
+C\bigl\|w\bigr\|_{C^{\alpha/2,\alpha/4}_{y_2,s}}, 
	\end{aligned}$$
	where $\tilde Q'=\Omega_1\times\partial\Omega_2\times
	\big\{(\hat t,\bar t,s),\ 0<s<\bar t<\hat t<T\big\}$.
	By \eqref{Holder0} and \eqref{S1HolderClaim1}, it follows that
\be\label{HolderfDx2b}
	\big[D_{x_2}v_1 
	\big]_{C^{\alpha/4}_{t}} 
\le C\bigl\|w\bigr\|_{C^{\alpha,\alpha/2,\alpha/2}_{y_2,t,s}}  
\le C\|\bar g\|_{C^{\alpha,\alpha,\alpha/2}_{y_1,y_2,s}}.  
\ee
Also, writing
$$\frac{\big[D_{x_2}v_1(x_1,x_2,t)\big]_{\bar x_1}^{\hat x_1}}{|\hat x_1-\bar x_1|^{\alpha/2}}
=D_{x_2}T_2(t)\biggl\{\frac{\big[w(\underline{x_1},y_2,\underline{t},s)\big]_{\bar x_1}^{\hat x_1}}{|\hat x_1-\bar x_1|^{\alpha/2}}\biggr\},$$
setting $\tilde Q''=\big\{(\hat x_1,\bar x_1);\ \hat x_1,\bar x_1\in\Omega_1,\,\hat x_1\ne\bar x_1\}\times\partial\Omega_2\times J$ and using \eqref{Holder0} and \eqref{S1HolderClaim1},
we obtain, 
\be\label{HolderfDx2c}
\big[D_{x_2}v_1  
\big]_{C^{\alpha/2}_{x_1}}  
\le C\biggl\|\frac{\big[w(x_1,y_2,t,s)\big]_{\bar x_1}^{\hat x_1}}{|\hat x_1-\bar x_1|^{\alpha/2}}\biggr\|_{C^{\alpha/2,\alpha/4}_{y_2,s}(\tilde Q'')}
\le C\bigl\|w\bigr\|_{C^{\alpha,\alpha,\alpha/2}_{x_1,y_2,s}}  
\le C\|\bar g\|_{C^{\alpha,\alpha,\alpha/2}_{y_1,y_2,s}}. 
\ee
Combining \eqref{HolderfDx2a}-\eqref{HolderfDx2c}, we have proved \eqref{HolderDx2v1xt}.

{\bf Step 3.} {\it H\"older estimate of $D_{x_1}v_1$.}
We shall show that 
\be\label{HolderDx1v1xt}
\|D_{x_1}v_1\|_{C^{\beta,\beta}_{x,t}(Q_T)} \le C\|\bar g\|_{C^{\alpha,\alpha,\alpha/2}_{y_1,y_2,s}(\Omega_1\times \Sigma_T)},
\quad 0<\beta<\alpha/4.
\ee

We set 
$$[\tilde S_2(\tau)\psi](x_2)=\int_{\partial\Omega_2} G_2(x_2,y_2,\tau) \psi(y_2) d\sigma_{y_2},
\quad (x_2,\tau)\in Q_{2,T},$$
and rewrite $v_1$ as
\be\label{v1S1tildeS2}
v_1(x_1,x_2,t) = \int_0^t S_1(t-s)\big[z(y_1,\underline{x_2},\underline{t},\underline{s})\big] \,ds,
\quad z(y_1,x_2,t,s)=\tilde S_2(t-s)\bar g(\underline{y_1},y_2,\underline{s}),
\ee
where the partial function $z$ is defined on 
$\Omega\times J$. 
In this step, unless otherwise specified, the norms of $v_1$ and $\bar g$
(and their derivatives) will be taken over $Q_T$ and $\Omega_1\times \Sigma_T$, respectively.
For integers $i,j\ge 0$ with $i+j\le 1$, we have the heat kernel and semigroup estimates
$$
\int_{\partial\Omega_2}|D^i_\tau D^j_{x_2} G_2(x_2,y_2,\tau)|d\sigma_{y_2}
\le C\int_{\partial\Omega_2}\tau^{-\frac{n+2i+j}{2}}e^{-C_2|x_2-y_2|^2/\tau}d\sigma_{y_2}\le C\tau^{-\frac{1+2i+j}{2}},
$$
owing to \eqref{GaussEst} and \eqref{GaussEstDt}
(the last inequality follows by using local charts and flattening the boundary).
Consequently,
\be\label{GaussEstAppB}
\|D^i_\tau D^j_{x_2}\tilde S_2(\tau)\psi\|_\infty\le C\tau^{-(1+2i+j)/2}\|\psi\|_\infty,\quad \psi\in L^\infty(\Omega_1),\  i,j\ge 0,\ i+j\le 1,
\ee
and, for $\alpha,\beta\in[0,1)$, by \cite[Theorem~2.3]{Mora},
\be\label{Dx1S1}
\|D_{x_1}S_1(\tau)\phi\|_{C^{\beta,\beta/2}_{x_1,\tau}(\Omega_1\times(\eps,T))}\le C\eps^{-(1+\beta-\alpha)/2}\|\phi\|_{C^\alpha_{y_1}(\Omega_1)},
\quad \phi\in C^\alpha(\Omega_1),\ 0<\eps<T.
\ee

We first derive H\"older estimates for $z$. Let $\beta\in(0,1)$.
It follows from \eqref{GaussEstAppB} with $i=j=0$ that,
for $0<s<t<T$, 
$$
|z(y_1,x_2,t,s)| \le
C(t-s)^{-\frac12}\|\bar g\|_\infty,\ \
\big|\left[z(y_1,x_2,t,s)\right]_{\bar y_1}^{\hat y_1}\big|\le
C(t-s)^{-\frac12}\sup_{(y_2,s)\in \Sigma_T}\big|\left[g(y_1,y_2,s)\right]_{\bar y_1}^{\hat y_1}\big|,
$$
hence
\be\label{HolderfDx1a}
\|z(y_1,x_2,t,s)\|_{C^\alpha_{y_1}(\Omega)}\le C(t-s)^{-1/2} \|\bar g\|_{C^\alpha_{y_1}}. 
\ee
By interpolating \eqref{GaussEstAppB} for $i=0$ between $j=0$ and $j=1$, we get
\be\label{HolderfDx1a2}
\|z(y_1,x_2,t,s)\|_{C^\beta_{x_2}(\Omega)}\le  C(t-s)^{-(1+\beta)/2} \|\bar g\|_\infty 
\ee
and by interpolating \eqref{GaussEstAppB} for $j=0$ between $i=0$ and $i=1$, we obtain
\be\label{HolderfDx1a3}
\big|[z(y_1,x_2,t,s)]_{\bar t}^{\hat t}\big|\le  C(\bar t-s)^{-(1+\beta)/2}|{\hat t}-{\bar t}|^{\beta/2} \|\bar g\|_\infty,
\quad 0<s<\bar t<\hat t<T.
\ee

We next estimate the H\"older norms of $D_{x_1}v_1$ in space.
By \eqref{Dx1S1} and \eqref{HolderfDx1a}, we have
$$\begin{aligned}
\|D_{x_1}S_1(t-s)z(y_1,\underline{x_2},\underline{t},\underline{s})\|_{C^\beta_{x_1}(\Omega)} 
\le C(t-s)^{-1+\frac{\alpha-\beta}{2}}\|\bar g\|_{C^\alpha_{y_1}},
\quad 0<s<t<T, 
\end{aligned}$$
so that, 
\be\label{HolderfDx1a2b}
\|D_{x_1}v_1\|_{C^{\beta}_{x_1}}\le  C\|\bar g\|_{C^\alpha_{y_1}},  
\quad 0<\beta<\alpha.
\ee
Let $\gamma\in (0,1)$. For $0<s<t<T$,
by \eqref{Dx1S1} with $\beta=0$ and \eqref{Holder0}, we have
$$	\begin{aligned}
\bigg|\frac{[D_{x_1}S_1(t-s)z(y_1,\underline{x_2},\underline{t},\underline{s})]_{\bar x_2}^{\hat x_2}}{|\hat x_2-\bar x_2|^{\gamma}}\bigg|
&=\bigg|D_{x_1}S_1(t-s)\biggl(\frac{[z(y_1,\underline{x_2},\underline{t},\underline{s})]_{\bar x_2}^{\hat x_2}}{|\hat x_2-\bar x_2|^{\gamma}}\biggr)\bigg|\\
&\le C(t-s)^{-\frac12+\frac{\alpha}{4}} \bigg\|\frac{\big[z(y_1,x_2,t,s)\big]_{\bar x_2}^{\hat x_2}}{|\hat x_2-\bar x_2|^{\gamma}}\biggr\|_{C^{\alpha/2}_{y_1}(\Omega_1)}\\
&\le C(t-s)^{-\frac12+\frac{\alpha}{4}} \|z(y_1,x_2,t,s)\|_{C^{\alpha,2\gamma}_{y_1,x_2}(\Omega)}.
	\end{aligned}$$
Therefore, by \eqref{HolderfDx1a} and \eqref{HolderfDx1a2} with $\beta=2\gamma$, we have
$$\big[D_{x_1}S_1(t-s)
z(y_1,\underline{x_2},\underline{t},\underline{s})\big]_{C^{\gamma}_{x_2}(\Omega)}\le  C\|\bar g\|_{C^\alpha_{y_1}} (t-s)^{-1+\frac{\alpha}{4}-\gamma}$$ 
so that, 
\be\label{HolderfDx1b}
\big[D_{x_1}v_1\big]_{C^{\gamma}_{x_2}}\le  C\|\bar g\|_{C^\alpha_{y_1}}, 
\quad 0<\gamma<\alpha/4.
\ee

Finally, to estimate the H\"older norm of $D_{x_1}v_1$ in time, for $x\in\Omega$ and $0<\bar t<\hat t<T$, 
we write
\be\label{HolderI123}
\begin{aligned}
	\big|[D_{x_1}v_1(x,t)]^{\hat t}_{\bar t}\big|
	&\le \int_0^{\bar t} \bigl|\big[D_{x_1}S_1(t-s)
z(y_1,\underline{x_2}, \underline{\hat t},\underline{s})\big]_{t=\bar t}^{t=\hat t}\bigr| \diff \! s\\
	&\qquad +\int_{\bar t}^{\hat t} \bigl|D_{x_1}S_1(\hat t-s)z(y_1,\underline{x_2}, \underline{\hat t},\underline{s})\bigr|\,\diff \! s\\
	&\qquad + \int_0^{\bar t} \bigl|D_{x_1}S_1(\bar t-s)\big[z(y_1,\underline{x_2}, \underline{t},\underline{s})]_{\bar t}^{\hat t}\big]\bigr|\diff \! s \equiv  I_1+I_2+I_3.
	\end{aligned}\ee
Let $0<\gamma<\alpha$ and $0<s<\bar t<\hat t<T$. By \eqref{Dx1S1} 
with $\beta=\gamma$ and \eqref{HolderfDx1a} we have
$$
\Big|\big[D_{x_1}S_1(t-s)z(y_1,\underline{x_2},\underline{\hat t},\underline{s})\big]_{t=\bar t}^{t=\hat t}\Big|
\le C(\bar t-s)^{-1+\frac{\alpha-\gamma}{2}}|\hat t-\bar t|^{\gamma/2} \|\bar g\|_{C^\alpha_{y_1}}, 
$$
hence
\be\label{HolderI1}
|I_1|
\le C 
|\hat t-\bar t|^{\gamma/2} \|\bar g\|_{C^\alpha_{y_1}},\quad 0<\gamma<\alpha,  
\ee
By \eqref{Dx1S1} with $\beta=0$ and \eqref{HolderfDx1a}, we get
	$$\bigl|D_{x_1}S_1(\hat t-s)z(y_1,\underline{x_2}, \underline{\hat t},\underline{s})\bigr|
	\le C(\hat t-s)^{-1+\frac{\alpha}{2}}\|\bar g\|_{C^\alpha_{y_1}},$$  
hence
\be\label{HolderI2}
|I_2|
\le C|\hat t-\bar t|^{\alpha/2} \|\bar g\|_{C^\alpha_{y_1}}.\ee 
For given $0<s<\bar t<T$, set $\hat Q'_{s,\bar t}=\{(y_1,x_2,t);\ x_1\in\Omega_1,\,x_2\in\Omega_2,\,t\in(\bar t,T)\}$.
By \eqref{Holder0}, \eqref{HolderfDx1a} and \eqref{HolderfDx1a3} with $\beta=\gamma$, we have
	$$
\bigg\|\frac{\big[z(y_1,x_2,t,s)\big]_{\bar t}^{\hat t}}{|\hat t-\bar t|^{\frac{\gamma}{4}}}\biggr\|_{C^{\alpha/2}_{y_1}(\Omega_1)}
\le 2\|z(y_1,x_2,t,s)\|_{C^{\alpha,\gamma/2}_{y_1,t}(\hat Q'_{s,\bar t})}
\le C(\bar t-s)^{-\frac{1+\gamma}{2}}\|g\|_{C^\alpha_{y_1}}
$$
	hence, using \eqref{Dx1S1} with $\beta=0$,
	$$
	\Big|D_{x_1}S_1(\bar t-s)\big[z(y_1,x_2,t,s)\big]_{\bar t}^{\hat t}\Big|
	\le C(\bar t-s)^{-1+\frac{\alpha-\gamma}{2}} |\hat t-\bar t|^{\gamma/4} \|g\|_{C^\alpha_{y_1}}, 
$$
so that
\be\label{HolderI3}
|I_3|
\le C 
|\hat t-\bar t|^{\gamma/4} \|\bar g\|_{C^\alpha_{y_1}},\quad 0<\gamma<\alpha. 
\ee
Combining \eqref{HolderI123}-\eqref{HolderI3} yields
\be\label{HolderfDx1c}
\big[D_{x_1}v_1\big]_{C^{\gamma}_{t}}\le  C\|\bar g\|_{C^\alpha_{y_1}},
\quad 0<\gamma<\alpha/4, 
\ee
and this along with \eqref{HolderfDx1a2b}-\eqref{HolderfDx1b} gives \eqref{HolderDx1v1xt}.

{\bf Step 4.} {\it Volume term and conclusion.}
Owing to $G=G_1G_2$,  and the estimates in \cite[Theorem~2.2]{Mora} for $G_1, G_2$, we have
$$
|D_x G(\hat x,y,\hat t)-D_x G(x,y,t)|
\le Ct^{-\frac{n}{2}-\frac{1+\alpha}{2}}e^{-C_2d^2/t}
\bigl(|\hat x-x|^\alpha+|\hat t-t|^{\alpha/2}\bigr)$$
for $\hat x,x,y\in \Omega$, $0<t<\hat t<T$, where
 $d=|\hat x-y|\wedge |x-y|$.
Using this estimate, \eqref{GaussEst} with $i=1$, $j=0$, and similar (but simpler) argument as in Step~3, 
it follows that the volume term $V$ in \eqref{represApp2} satisfies
\be\label{HolderfDxV}
\|D_xV\|_{C^{\alpha,\alpha/2}_{x,t}(Q_T)}\le  C\|\bar h\|_\infty.
\ee
 Gathering estimates 
 \eqref{HolderDx2v1xt}, \eqref{HolderDx1v1xt} and \eqref{HolderfDxV}, we conclude that
 $u$ extends by continuity to a $C^1$ function on $\overline\Omega\times(0,T^*)$
and, also using \eqref{regulbargh}, $u$ satisfies the estimate in \eqref{uclass1}.
\end{proof}

The proof of Proposition~\ref{locexist-cyl} will also make use of the following property of the Neumann heat semigroup.
This is well known for smooth domains but, again, we could not find a reference for cylinders.

\begin{lemm}\label{lem:SGW}
 Let $\Omega$ be given by \eqref{defCyl}, $p\in[1,\infty)$ and $T>0$. Then
 \be\label{SGW1}
 \|S(t)\psi\|_{1,p} \le C\|\psi\|_{1,p},\quad \psi\in W^{1,p}(\Omega),\ 0<t\le T.  
  \ee
\end{lemm}

\begin{proof}
By \eqref{GaussEst} with $i=j=0$, we have
 \be\label{SGW2}
 \|S(t)\phi\|_p \le C\|S(t)\phi\|_p,\quad \phi\in L^p(\Omega),\ 0<t\le T.
 \ee
Write $x=(x_1,x_2)$ with $x_1\in\Omega_1=(0,L)$ and $x_2\in\Omega_2=B'_R$.
Let $t\in(0,T]$ and $\psi\in W^{1,p}(\Omega)$. 
Owing to \eqref{productG1G2}, we may write $v=\nabla_{x_2}S(t)\psi$ as 
$$v(x_1,x_2,t) = \int_{\Omega_1}G_1(x_1,y_1,t) 
\underbrace{\Bigl(\int_{\Omega_2} \nabla_{x_2}G_2(x_2,y_2,t) \psi(y_1,y_2)dy_1\Bigr)}_{w(y_1,x_2,t)}dy_2.$$
Denote by $S_i(t)$ the Neumann heat semigroup on $\Omega_i$.
Since $\Omega_2$ is smooth, we know (see, e.g.,~\cite[Theorem~51.1(iv) and Example~51.4(ii)]{QSb}) that \eqref{SGW1} is true for $\Omega$ replaced by $\Omega_2$.
For a.e.~$y_1\in\Omega_1$, since $w(y_1,\cdot,t)=\nabla_{x_2}(S_2(t)\psi(y_1,\cdot))$,
it follows that $\|w(y_1,\cdot,t)\|_{L^p(\Omega_2)}\le C\|\psi(y_1,\cdot)\|_{W^{1,p}(\Omega_2)}$.
On the other hand, for a.e.~$x_2\in\Omega_2$, 
by \eqref{SGW2} applied with $\Omega$ replaced by $\Omega_1$, 
since $v(\cdot,x_2,t)=S_1(t)w(\cdot,x_2,t)$, we have
$\|v(\cdot,x_2,t)\|_{L^p(\Omega_1)}\le C\|w(\cdot,x_2,t)\|_{L^p(\Omega_1)}$.
By Fubini's theorem, we then deduce that
$$\begin{aligned}
&\|\nabla_{x_2}S(t)\psi\|^p_p 
=\int_{\Omega_2}\Bigl(\int_{\Omega_1}|v(x_1,x_2,t)|^p dx_1\Bigr)dx_2
\le C\int_{\Omega_2}\Bigl(\int_{\Omega_1}|w(y_1,x_2,t)|^p dy_1\Bigr)dx_2\\
&=C\int_{\Omega_1}\Bigl(\int_{\Omega_2}|w(y_1,x_2,t)|^p dx_2\Bigr)dy_1
\le C\int_{\Omega_1}\Bigl(\int_{\Omega_2}(|\psi|^p+|\nabla\psi|^p)(y_1,y_2) dy_2\Bigr)dy_1
= C\|\psi\|^p_{1,p}.  
	\end{aligned}$$
Applying the same argument upon exchanging the roles of $x_1, x_2$, and 
combining with \eqref{SGW2} yields \eqref{SGW1}.
\end{proof}

\begin{proof}[Proof of Proposition~\ref{locexist-cyl}]
We divide it in several steps.

{\bf Step 1.} {\it Existence of a mild solution.}
We prove the existence of a unique, maximal solution $u$
in the (larger) class $L^\infty(0,T;C(\overline\Omega))$ of \eqref{NewRepresApp}
by a fixed point argument.
For $R,\tau>0$ to be chosen, we set
$$\mathcal{X}_\tau=L^\infty(0,\tau;C(\overline\Omega)), \quad \|c\|_{\mathcal{X}_\tau}:=\sup_{t\in (0,\tau)}\|c(t)\|_\infty,\quad 
\mathcal{X}_{\tau,R}=\bigl\{c\in \mathcal{X}_\tau,\ \|c\|_{\mathcal{X}_\tau}\le R\bigr\}$$
and note that $\mathcal{X}_{\tau,R}$ is a complete metric space.
We define the fixed point operator:
$$[\Phi(c)](t)=S(t)c_0+\int_0^t K_\nabla(t-s)[A(c(s)) c(\cdot,s)]\diff \! s,\quad 0<t<\tau.$$
Set $M_R=\sup_{|s|\le R}|f(s)|$,
$L_R=\sup\bigl\{|\frac{f(x)-f(y)}{x-y}|;\ |s_1|, |s_2|\le R,\ s_1\ne s_2\bigr\}$
and let $u,v\in \mathcal{X}_{\tau,R}$ and $0<s<t<\tau$. 
Since $x\mapsto D_yG(x,y,t)$ is continuous in $\overline\Omega$ for each $y\in\overline\Omega$, 
it follows from \eqref{smooth1b} with $q=r=\infty$ and dominated convergence that 
$K_\nabla(t-s)[A(u(s))u(\cdot,s)]\in C(\overline\Omega)$.
On the other hand, we easily get
$\|A(u(s))u(\cdot,s)\|_\infty\le C_\Omega M_R\|u(\cdot,s)\|_\infty$ and
$\|A(u(s))u(\cdot,s)-A(v(s))v(\cdot,s)\|_\infty\le C_\Omega M_R(L_R+1)\|u-v\|_{\mathcal{X}_\tau}$
with $C_\Omega>0$.
Recall that 
\be\label{resgulSc_0}
S(t)c_0\in C([0,\infty);C(\overline\Omega)),
\ee
owing to $c_0\in C(\overline\Omega)$, and that $\|S(t)c_0\|_\infty\le \|c_0\|_\infty$.
Using this and \eqref{smooth1b} with $q=r=\infty$, we obtain $\Phi(u)\in \mathcal{X}_\tau$ and 
$$\|\Phi(u)\|_{\mathcal{X}\tau}\le \|c_0\|_\infty+C_\Omega M_R\tau^{1/2}\|u\|_{\mathcal{X}_\tau},\ 
\|\Phi(u)-\Phi(v)\|_{\mathcal{X}\tau}\le C_\Omega M_R(L_R+1)\tau^{1/2}\|u-v\|_{\mathcal{X}_\tau}.$$
Choosing $R=1+\|c_0\|_\infty$ and $\tau=(2C_\Omega M_R(L_R+1))^{-2}$,
it follows that $\Phi$ is a contraction on $\mathcal{X}_{\tau,R}$ and the
Banach fixed point theorem yields the existence of a unique solution $c$ of \eqref{NewRepresApp} on $(0,\tau)$.

Denoting by $c\in L^\infty_{loc}([0,T^*);C(\overline\Omega))$ the maximally defined mild solution, we see that
the maximal existence time satisfies 
\be\label{lowerT}
T^*\ge C(\Omega) (M_R(L_R+1))^{-2},\quad R=1+\|c_0\|_\infty.
\ee
The continuous dependence property (assertion (iv)) follows from standard arguments
based on the above fixed point estimates and Gronwall's lemma.

\smallskip
{\bf Step 2.} 
{\it $W^{1,p}$ regularity.}
Let $p\in[1,\infty)$, $T_0>0$, $0<T<\min(T_0,T^*)$, $\eps\in(0,T)$ and 
assume \eqref{DefN0}.
We claim that
	\be\label{regulW1pApp0eps}
c\in L^\infty_{loc}((0,T^*);W^{1,p}(\Omega)),\quad 
\|c\|_{L^\infty(\eps,T;W^{1,p}(\Omega))}\le C(N_0,\eps,T_0).
\ee
(here and below the constants may also depend on $\Omega,f,p$).

For any $k\in[0,1)$, using \eqref{smooth2} for $q=p$ and $\ell=0$, we get 
	\be\label{W1pc1}
	\|c(t)-S(t)c_0\|_{k,p}\le C(N_0,T_0)t^{(1-k)/2},\quad t\in(0,T).
	\ee
For $t\in(0,T)$, by \eqref{smooth1} with $q=r=p$, it follows that
$\|c(t)\|_{k,p}\le  C(N_0,T_0)t^{-k/2}$.
Then using \eqref{smooth2} for $\ell=1/2p$, we obtain
	\be\label{W1pc2}
	\begin{aligned}
\|c(t)-S(t)c_0\|_{1,p}
&\le\sup_{s\in(0,T)} |A(s)| \int_0^t  (t-s)^{-1+\frac{1}{4p}} \|c(s)\|_{1/2p,p} \diff \! s \\
&\le  
C(N_0,T_0) \int_0^t  (t-s)^{-1+\frac{1}{4p}} s^{-1/4p} \diff \! s\le C(N_0,T_0).
\end{aligned}
\ee
Combining with \eqref{smooth1} for $q=r=p$ and $k=1$, we deduce \eqref{regulW1pApp0eps}.

Next assume $c_0\in W^{1,p}(\Omega)$. Using
Lemma~\ref{lem:SGW} to estimate $S(t)c_0$ in \eqref{W1pc1}, 
we get $K:=\sup_{t\in(0,T)}\|c(t)\|_{k,p}<\infty$ for $k\in[0,1)$.
The first inequality in \eqref{W1pc2} then implies
	\be\label{W1pc2B}
\|c(t)-S(t)c_0\|_{1,p}\le C(N_0,T_0) Kt^{1/4p},
\ee 
and a second application of Lemma~\ref{lem:SGW} yields \eqref{regulW1pApp}.
On other hand, in the case $p=2$, inequality \eqref{W1pc2B} combined with the last part of \eqref{ustrong} 
(applied with $h=g=0$) gives \eqref{regulH1pApp}.

\smallskip
{\bf Step 3.} 
{\it H\"older regularity.}
Let $\alpha\in(0,1)$, $T_0>0$, $0<T<\min(T_0,T^*)$, $\eps\in(0,T)$ and assume \eqref{DefN0}.
In the rest of the proof we denote by $\bar N_\eps$ a generic positive constant
depending only on $N_0,\eps,T_0$ (and $\Omega,f,\alpha$).

We claim that
	\be\label{timeHolder}
			c\in C^{\alpha,\alpha/2}_{loc}(\overline\Omega\times(0,T^*)),
			\quad \|c\|_{C^{\alpha,\alpha/2}(\overline\Omega\times[\eps,T])}\le \bar N_\eps
			\quad\hbox{and}\quad
c\in C(\overline\Omega\times[0,T^*)).		\ee
Let $0<t<t'<T$ and set $h=t'-t$. From \eqref{GaussEstDt} with $i=1$ and $\ell=0$ we deduce, for all 
$\psi\in L^\infty(\Omega)$ and $x\in\overline\Omega$,
$$\bigl|(K_\nabla(t')\psi-K_\nabla(t)\psi)(x)\bigr|\le C h \Bigl|\int_0^1 \int_\Omega (t+\theta h)^{-\frac{n+3}{2}}
e^{-\frac{C_2|x-y|^2}{t+\theta h}} \psi(y)\diff \! y\diff \! \theta\Bigr|
\le Ch t^{-\frac{3}{2}} \|\psi\|_\infty.$$
For any $\nu\in[0,1]$, this combined with \eqref{smooth1b} for $q=r=\infty$, implies
$$\bigl\|(K_\nabla(t')-K_\nabla(t))\psi\bigr\|_\infty\le C h^\nu t^{-\nu-1/2} \|\psi\|_\infty.$$
Set $V(\cdot,s)=[A(c(s)) c(\cdot,s)]$, $\tilde c(x,t)=\int_0^t K_\nabla(t-s)V(s)\diff \! s$.
Using $\sup_{s\in(0,T)} \|V(s)\|_\infty<C(N_0)$ and \eqref{smooth1b} for $q=r=\infty$,
it follows that, for any $\nu\in (0,1/2)$ and $x\in\overline\Omega$,
$$\begin{aligned}
	\|\tilde c(t')-\tilde c(t)\|_\infty
	&=\Bigl\| \int_t^{t'} K_\nabla(t'-s)V(s)\diff \! s
+\int_0^t \bigl(K_\nabla(t'-s)-K_\nabla(t-s)\bigr)V(s)\diff \! s\Bigr\|_\infty \\
	&\le \bar N_\eps\int_t^{t'} (t'-s)^{-1/2}\diff \! s
+\bar N_\eps(t'-t)^\nu\int_0^t (t-s)^{-\nu-1/2}\diff \! s\le \bar N_\eps(t'-t)^\nu.
	\end{aligned}$$
	This, along with similar arguments for $S(t)c_0$ using \eqref{GaussEstDt} with $i=0$, guarantees that
	$\|c\|_{C^{\alpha/2}(\eps,T;C(\overline\Omega))}\le \bar N_\eps$.
	Combining with \eqref{regulW1pApp0eps} and Morrey's imbedding (which holds since $\Omega$ admits a $C^1$ extension operator), we deduce the H\"older part of \eqref{timeHolder}.
	On the other hand, from \eqref{smooth1b} with $q=r=\infty$, we have $\|c(t)-S(t)c_0\|_\infty=\|\int_0^t K_\nabla(t-s)V(s)\diff \! s\|_\infty\le Ct^{1/2}$. In view of \eqref{resgulSc_0}, the
	continuity statement of \eqref{timeHolder} follows.

	In particular we have proved assertion (i).	
 	
\smallskip
{\bf Step 4.} {\it $H^2$ regularity and strong solution.} We shall now prove \eqref{ustrong2}.
Fix any $t_1, T$ with $0<t_1<T<T^*$, set $t_2=T-t_1$,
$a(t)=\int_{\partial\Omega}f(c)\nu \dsigma$, 
$c_1(t)=c(t+t_1)$, $a_1(t)=a(t+t_1)$ and
consider problem \eqref{pbmNeumannLinear} with $h=-a_1\cdot\nabla c_1$,
$g=(a_1\cdot\nu)c_1$ and $u_0=c(t_1)$.
By \eqref{regulW1pApp0eps}, \eqref{timeHolder} and standard imbedding and trace theorems, we have in particular
$g\in L^2(0,t_2;H^{1/2}(\partial\Omega))\cap H^{1/4}(0,t_2;L^2(\partial\Omega))$, as well as
$h\in L^2(\Omega\times(0,t_2))$.
Also, by \eqref{regulW1pApp0eps}, \eqref{timeHolder}, we have 
$\sup\,\{|\int_\Omega c(t_1)\nabla v|,\ v\in H^1,\,\|v\|_{H^1}=1\}<\infty$ 
hence $c(\cdot,t_1)\in H^1(\Omega)$ 
(whereas \eqref{regulW1pApp0eps} alone, being an a.e.~in time property, would not be sufficient).
We may then apply Proposition~\ref{prop:Cylinder} to deduce the existence of a unique strong solution $u$ of \eqref{pbmNeumannLinear} on $(0,t_2)$,
and $u$ satisfies the representation formula \eqref{represApp} with $\tau=t_2$.
Since $h=-a_1\cdot\nabla c_1$ and
$g=(a_1\cdot\nu)c_1$, integrating by parts as in the proof of Lemma~\ref{lem:repres}, we get
$$
u(x,t)=\int_\Omega G(x,y,t)c(y,t_1)\diff \! y 
+\int_0^t \int_\Omega \nabla_yG(x,y,t-s)\cdot \bigl(c_1(y,s) a_1(s)\bigr)\diff \! y \diff \! s.
$$
On the other hand, by standard manipulations on the Duhamel formula \eqref{NewRepresApp}, using Fubini's theorem and the fact that
$\nabla_yG(x,y,s_1+s_2)=\int_\Omega G(x,z,s_1)\nabla_yG(z,y,s_2) dz$ for all $s_1, s_2>0$, we see that
$$c(x,t_1+t)=\int_\Omega G(x,y,t)c(y,t_1)\diff \! y 
+\int_{t_1}^{t_1+t} \int_\Omega \nabla_yG(x,y,t_1+t-s)\cdot \bigl(c(y,s) a(s)\bigr)\diff \! y \diff \! s$$
hence, after a time shift, $u=c_1$ in $Q_2$.
Since $0<t_1<T<T^*$ are arbitrary, this proves~\eqref{ustrong2}
and $c$ is moreover a strong solution.
This, along with Step~2, completes the proof of assertion~(iii).

\smallskip
{\bf Step 5.} 
{\it Proof of assertion~(vi) (Schauder regularity).}
As a consequence of \eqref{timeHolder} and Proposition~\ref{prop:Cylinder}(ii) we have
\be\label{uclass1b}
\|\nabla c\|_{C^{\beta}(\overline\Omega\times[\eps,T])}\le \bar N_\eps,\quad 0<\beta<1/4.
\ee
It thus remains to prove that, for any compact subset $\Sigma\subset \Omega\cup\Gamma$,
\be\label{uclass2}
\|c\|_{C^{2+\beta,1+\beta/2}(\Sigma\times[\eps,T])}\le C(\bar N_\eps,\Sigma).
\ee

Fix a smooth domain $\Omega'$ such that $\Omega\subset\Omega'\subset (-L,2L)\times B'_R$.
Pick any $\delta\in(0,L/4)$ and a function $\theta\in C^\infty(\R)$ such that $\theta=1$ on $[\delta,L-\delta]$
and $\theta=0$ on $\R\setminus(\frac{\delta}{2},L-\frac{\delta}{2})$.
Also  
pick a function $\eta\in C^\infty(\R)$ such that $\eta=1$ on $[\eps,\infty)$
and $\eta=0$ on $(-\infty,\frac{\eps}{2}]$.
Then $\tilde c:=\theta\eta c=\theta(x_1)\eta(t)c(x,t)$ is a strong solution of \eqref{pbmNeumannLinear}  
in $\Omega'\times(0,T)$ with $u_0=0$ and
$$h=(\eta'\theta-\eta\theta'')c-\eta (a\theta+2\theta'e_1)\cdot\nabla c,\quad 
g=\eta[\theta a+ \theta' e_1]\cdot\nu c .$$
Owing to \eqref{timeHolder} and \eqref{uclass1b}, the functions $h, g$ satisfy
$\|h\|_{C^{\beta,\beta/2}(\overline\Omega'\times[0,T])}\le C(\bar N_\eps,\delta)$
and $\|g\|_{C^{1+\beta,\beta/2}(\partial\Omega'\times[0,T])}\le C(\bar N_\eps,\delta)$. It then follows from \cite[Theorem~4.31]{Lie}
 that \eqref{pbmNeumannLinear} admits a classical (hence strong) solution $v$, such that
$\|v\|_{C^{2+\beta,1+\beta/2}(\Omega'\times[0,T])}\le C(\bar N_\eps,\delta)$.
By uniqueness of the strong solution, we get $u=v$,  
hence 
$\|u\|_{C^{2+\beta,1+\beta/2}(([\delta,L-\delta]\times \overline B'_R)\times [\eps,T])}
\le C(\bar N_\eps,\delta)$.
Applying this argument with $\theta$ replaced by a spatial cut-off in the radial direction $|x'|$, 
we thus obtain \eqref{uclass2}.
Since $u$ is a strong solution by Step~4, it follows from Step~5 and \eqref{uclass2} that 
$c$ is a classical solution.

\smallskip
{\bf Step 6.} 
{\it Proof of assertions (ii) and (v).}
The nonnegativity statement in assertion (ii) is similar to that of Lemma~\ref{lempos1}(i) 
(using the regularity property \eqref{ustrong2}).
To prove the positivity part, thus assume $\sigma_0=\min_{\overline\Omega} c_0>0$, let $T\in(0,T^*)$ and set $K=\sup_{t\in(0,T)}|A(t)|$. 
Fix a positive function $\varphi\in C^2(\overline\Omega)$ such 
 that $\varphi_\nu\ge K\varphi$ on $\partial\Omega$
 (it suffices to consider $\varphi(x_1,x')=\varphi_1(x_1)\varphi_2(x')$,
 where the positive $C^2$ functions 
 $\varphi_1, \varphi_2$ satisfy the required property on $\{0,L\}$ and $\partial B'_R$, respectively).
 Set $\sigma_1=\min_{\overline\Omega} \varphi>0$, $b=-A-2\varphi^{-1}\nabla\varphi$ and 
 $\lambda=\sup_{\Omega\times(0,T)} (\varphi^{-1}\Delta\varphi+b\cdot\nabla\varphi)$.
 The function $v:=e^{\lambda t}\varphi c\ge 0$ satisfies 
	$$\begin{aligned}
	v_t-\Delta v
	&=e^{\lambda t}\bigl[\varphi(c_t-\Delta c+\lambda c)-2\nabla\varphi\cdot\nabla c-c\Delta\varphi\bigr]
=(\lambda-\varphi^{-1}\Delta\varphi) v+e^{\lambda t}(b\cdot\varphi \nabla c)\\
&=b\cdot\nabla v+(\lambda-\varphi^{-1}\Delta\varphi-b\cdot\nabla\varphi)v\ge b\cdot\nabla v,
 	\end{aligned}$$
along with $v_\nu=e^{\lambda t}(\varphi_\nu c+c_\nu\varphi) \ge (K+A(t)\cdot\nu)v\ge 0$ on $\Gamma$, and $v(\cdot,0)\ge \sigma_2:=\sigma_0\sigma_1$.
 Multiplying \eqref{pbmP} by $-(v-\sigma_2)_-$ and applying the Stampacchia argument in the proof of Lemma~\ref{lempos1}(i),
 we obtain $(v-\sigma_2)_-=0$, hence $c\ge \sigma_2 \|\varphi\|_\infty^{-1}e^{-\lambda t}$ on $(0,T)$, which implies the conclusion.

To prove assertion (v), let $c\in \tilde E_T$ be a classical solution of \eqref{pbmP}. 
We shall show that $u$ satisfies \eqref{NewRepresApp}.
The idea of the proof is similar to that of \eqref{represApp} (and of Lemma~\ref{lem:repres}),
but since $u$ is assumed to be only $C^1$ up to the boundary 
(and $C^2$ only outside of the corners), $\Delta u$ need not be integrable in $\Omega$.
To overcome this, we use an approximation of $\Omega$ by subdomains which avoid the corners.
Namely, denoting $\tilde\Gamma=\{0,L\}\times \partial B'_R$ the nonsmooth part of $\partial\Omega$
and $\tilde\Omega_r=\{x\in \Omega;\ {\rm dist}(x,\tilde\Gamma)>r\}$.
We may select a sequence $\{\Omega_j\}_{j\ge 1}$ of smooth domains
such that 
	\be\label{propOmegaj}
	\tilde\Omega_{1/j}\subset \Omega_j\subset \Omega
\quad\hbox{and}\quad
\lim_{j\to\infty}\int_{\partial\Omega_j\setminus\partial\Omega} \diff \! \sigma^j_y =0,
\ee
where $\diff \! \sigma^j_y$ denotes the surface measure on $\partial\Omega_j$.

Fix $\tau\in(0,T)$ and set $c_\tau(t)=c(t+\tau)$, $a_\tau(t)=\int_\Omega \nabla(f(c_\tau))dx$,
$h=-a_\tau\cdot\nabla c_\tau$.
Let $t\in(0,T-\tau)$, $\eps\in(0,t)$, $x\in\Omega$ and $j\ge 1$.
Integrating by parts in time and space the expression 
$\int_0^{t-\eps} \int_{\Omega_j} G(x,y,t-s) (u_t-\Delta u)(y,s)\diff \! y\diff \! s$
(where $G$ still denotes the Neumann heat kernel of $\Omega$),
we obtain
$$ 
\begin{aligned}
&\int_{\Omega_j} G(x,y,\eps)c_\tau(y,t-\eps)\diff \! y-\int_{\Omega_j} G(x,y,t)c(y,\tau)\diff \! y 
=\int_0^{t-\eps} \int_{\Omega_j} G(x,y,t-s) h(y,s) \diff \! y\diff \! s \\
&\qquad+\int_0^{t-\eps} \int_{\partial\Omega_j} 
[c_\tau(y,s)D_yG(x,y,t-s) -G(x,y,t-s) \nabla c_\tau(y,s)]\cdot\nu_j \diff \! \sigma^j_y  \diff \! s
\end{aligned}
$$ 
where $\nu_j$ denotes the outer normal to $\Omega_j$.
Using  \eqref{GaussEst},  \eqref{propOmegaj}, $\sup_{\overline\Omega\times[\tau,T]}(|c|+|\nabla c|)<\infty$,  $\partial_\nu G=0$ on 
$\partial\Omega$, we may pass to the limit $j\to\infty$,  
to obtain
$$ 
\begin{aligned}
&\int_\Omega G(x,y,\eps)c_\tau(y,t-\eps)\diff \! y-\int_\Omega G(x,y,t)c(y,\tau)\diff \! y 
=\int_0^{t-\eps} \int_\Omega G(x,y,t-s) h(y,s) \diff \! y\diff \! s \\
&\qquad-\int_0^{t-\eps} \int_{\partial\Omega} 
G(x,y,t-s) \partial_\nu c_\tau(y,s) \diff \! \sigma_y  \diff \! s.
\end{aligned}
$$ 
Using the boundary conditions in the last integral, passing to the limit $\eps\to 0$, 
then integrating by parts as in the proof of Lemma~\ref{lem:repres}
and finally letting $\tau\to 0$, we conclude that $c$ satisfies \eqref{NewRepresApp}.

To show \eqref{contTast}, assume for contradiction that $\sup_j \|c(t_j)\|_\infty<\infty$ for some sequence $t_j\to T^*$.
For $j$ sufficiently large, by Step~1 and in view of \eqref{lowerT}, we may find a mild solution $\tilde c$,
with initial data $c(t_j)$,  on some interval $[t_j,T_1]$ with $T_1>T^*$. Since $\tilde c$ is classical by 
assertion~(i) and coincides with $c$ on $[t_j,T^*)$ by the already established uniqueness part of assertion~(v),
this produces a classical solution on $[0,T_1]$, which is also a mild solution by \eqref{NewRepresApp},
hence contradicting the definition of $T^*$.
\end{proof}

\end{document}